\documentclass[11pt,letterpaper]{amsart}
\usepackage{cases}
\usepackage{mathrsfs}
\usepackage{color}
\usepackage{latexsym,amssymb}
\usepackage{a4wide}
\usepackage{amsthm}
\usepackage{amsmath}
\usepackage{graphicx}
\input xy
\xyoption{all}
\usepackage{verbatim}
\usepackage[lite,abbrev,msc-links]{amsrefs}
\usepackage{framed}
\usepackage{bm}

\numberwithin{equation}{section}
\begin{document}
	\theoremstyle{plain}
	\newtheorem{thm}{Theorem}[section]
	\newtheorem{lem}[thm]{Lemma}
	\newtheorem{cor}[thm]{Corollary}
	\newtheorem{cor*}[thm]{Corollary*}
	\newtheorem{prop}[thm]{Proposition}
	\newtheorem{prop*}[thm]{Proposition*}
	\newtheorem{conj}[thm]{Conjecture}
	\theoremstyle{definition}
	\newtheorem{construction}{Construction}
	\newtheorem{notations}[thm]{Notations}
	\newtheorem{question}[thm]{Question}
	\newtheorem{prob}[thm]{Problem}
	\newtheorem{rmk}[thm]{Remark}
	\newtheorem{remarks}[thm]{Remarks}
	\newtheorem{defn}[thm]{Definition}
	\newtheorem{claim}[thm]{Claim}
	\newtheorem{assumption}[thm]{Assumption}
	\newtheorem{assumptions}[thm]{Assumptions}
	\newtheorem{properties}[thm]{Properties}
	\newtheorem{exmp}[thm]{Example}
	\newtheorem{comments}[thm]{Comments}
	\newtheorem{blank}[thm]{}
	\newtheorem{observation}[thm]{Observation}
	\newtheorem{defn-thm}[thm]{Definition-Theorem}
	\newtheorem*{Setting}{Setting}

	\newcommand{\sA}{\mathscr{A}}
	\newcommand{\sB}{\mathscr{B}}
	\newcommand{\sC}{\mathscr{C}}
	\newcommand{\sD}{\mathscr{D}}
	\newcommand{\sE}{\mathscr{E}}
	\newcommand{\sF}{\mathscr{F}}
	\newcommand{\sG}{\mathscr{G}}
	\newcommand{\sH}{\mathscr{H}}
	\newcommand{\sI}{\mathscr{I}}
	\newcommand{\sJ}{\mathscr{J}}
	\newcommand{\sK}{\mathscr{K}}
	\newcommand{\sL}{\mathscr{L}}
	\newcommand{\sM}{\mathscr{M}}
	\newcommand{\sN}{\mathscr{N}}
	\newcommand{\sO}{\mathscr{O}}
	\newcommand{\sP}{\mathscr{P}}
	\newcommand{\sQ}{\mathscr{Q}}
	\newcommand{\sR}{\mathscr{R}}
	\newcommand{\sS}{\mathscr{S}}
	\newcommand{\sT}{\mathscr{T}}
	\newcommand{\sU}{\mathscr{U}}
	\newcommand{\sV}{\mathscr{V}}
	\newcommand{\sW}{\mathscr{W}}
	\newcommand{\sX}{\mathscr{X}}
	\newcommand{\sY}{\mathscr{Y}}
	\newcommand{\sZ}{\mathscr{Z}}
	\newcommand{\bZ}{\mathbb{Z}}
	\newcommand{\bN}{\mathbb{N}}
	\newcommand{\bQ}{\mathbb{Q}}
	\newcommand{\bC}{\mathbb{C}}
	\newcommand{\bR}{\mathbb{R}}
	\newcommand{\bH}{\mathbb{H}}
	\newcommand{\bD}{\mathbb{D}}
	\newcommand{\bE}{\mathbb{E}}
	\newcommand{\bP}{\mathbb{P}}
	\newcommand{\bV}{\mathbb{V}}
	\newcommand{\cV}{\mathcal{V}}
	\newcommand{\cF}{\mathcal{F}}
	\newcommand{\bfM}{\mathbf{M}}
	\newcommand{\bfN}{\mathbf{N}}
	\newcommand{\bfX}{\mathbf{X}}
	\newcommand{\bfY}{\mathbf{Y}}
	\newcommand{\spec}{\textrm{Spec}}
	\newcommand{\dbar}{\bar{\partial}}
	\newcommand{\ddbar}{\partial\bar{\partial}}
	\newcommand{\redref}{{\color{red}ref}}
	
	\title[Stratified Hyperbolicity of the moduli stack of stable minimal models, I] {Stratified Hyperbolicity of the moduli stack of stable minimal models, I}
	
	\author[Junchao Shentu]{Junchao Shentu}
	\email{stjc@ustc.edu.cn}
	\address{School of Mathematical Sciences,
		University of Science and Technology of China, Hefei, 230026, China}

	\begin{abstract}
		In this paper, we introduce a birationally admissible stratification on the Deligne-Mumford stack of stable minimal models (e.g., the KSBA moduli stack), such that the universal family over each stratum admits a simple normal crossing log birational model. We further demonstrate that each stratum is hyperbolic in the sense that every schematic generically finite covering of any closed substack is of logarithmic general type. This provides a partial answer to C.Birkar's question regarding the global geometry of the moduli of stable minimal models.
	\end{abstract}
	
	\maketitle
	
	\section{Introduction}
	The construction of compact moduli spaces for algebraic varieties is a foundational problem in algebraic geometry. Since the groundbreaking work of P. Deligne and D. Mumford \cite{Deligne1969} on the moduli of stable curves, significant efforts have been devoted to extending the construction of compact moduli spaces to higher-dimensional varieties. The primary goal is to establish a robust framework for studying families of algebraic varieties with prescribed properties. These endeavors have resulted in the development of a rich and diverse theory of moduli spaces, which has found applications across various areas of algebraic geometry and related fields. Notable examples include meaningful compactifications of moduli spaces for polarized abelian varieties \cite{Alexeev2002}, plane curves \cite{Hacking2004}, manifolds of general type \cite{Kollar1988,Kollar2010,Kollar2023}, polarized Calabi-Yau manifolds \cite{KX2020}, and $K$-polystable Fano manifolds \cite{Xu2020}, among others.
	Given $d \in \bN$, $c \in \bQ^{\geq 0}$, a finite set $\Gamma \subset \bQ^{> 0}$, and $\sigma \in \bQ[t]$, C. Birkar \cite{Birkar2022} introduced the concept of a $(d, \Phi_c, \Gamma, \sigma)$-stable minimal model as a triple $(X, B), A$, where $X$ is a projective scheme of finite type over ${\rm Spec}(\bC)$, and $A \geq 0$, $B \geq 0$ are $\bQ$-divisors on $X$ satisfying the following conditions:  
	\begin{itemize}  
		\item $\dim X = d$, $(X, B)$ forms an slc projective pair, and $K_X + B$ is semi-ample,  
		\item the coefficients of $A$ and $B$ lie in $c \bZ^{\geq 0}$,  
		\item $(X, B + tA)$ is slc, and $K_X + B + tA$ is ample for some $t > 0$,  
		\item ${\rm vol}(K_X + B + tA) = \sigma(t)$ for $0 \leq t \ll 1$, and  
		\item ${\rm vol}(A|_F) \in \Gamma$, where $F$ denotes any general fiber of the fibration $f: X \to Z$ determined by $K_X + B$.  
	\end{itemize}
	The moduli stack $\sM_{\rm slc}(d,\Phi_c,\Gamma,\sigma)$ of $(d,\Phi_c,\Gamma,\sigma)$-stable minimal models (see \cite{Birkar2022} or Section \ref{section_moduli}) is complete and admits a projective good coarse moduli space (\cite[Theorem 1.14]{Birkar2022}), which provides a meaningful compactification of the moduli of birational equivalence classes of projective manifolds with arbitrary Kodaira dimension. In this article, we investigate the global geometry of $\sM_{\rm slc}(d,\Phi_c,\Gamma,\sigma)$, with particular emphasis on its stratified hyperbolicity.  
	
	Research in this area was initiated by the pioneering works of A. N. Par\v{s}in \cite{Parsin1968}, S. Ju. Arakelov \cite{Arakelov1971}, and a series of contributions by Viehweg-Zuo \cite{VZ2001,VZ2002,VZ2003}, focusing on the principal stratum of certain moduli spaces where the fibers are smooth varieties. In this paper, we establish the stratified hyperbolicity properties of $\sM_{\rm slc}(d,\Phi_c,\Gamma,\sigma)$, encompassing both the principal stratum and the boundary strata.
	
	The main result of this paper is the following theorem:  
	\begin{thm}\label{thm_main}  
		Let $f^o:(X^o,B^o),A^o \to S^o$ be a birationally admissible family of $(d,\Phi_c,\Gamma,\sigma)$-stable minimal models over a smooth quasi-projective variety $S^o$, defining a generically finite classifying map $\xi^o:S^o \to \sM_{\rm slc}(d,\Phi_c,\Gamma,\sigma)$. Then $S^o$ is of log general type.  
	\end{thm}
    A family of stable minimal models $(X^o, B^o), A^o \to S^o$ is said to be \emph{birationally admissible} if $(X^o, A^o + B^o)$ admits a simple normal crossing log birational model (Definition \ref{defn_admissible}). Informally, the birationally admissible condition implies that the family does not possess degenerate fibers in the context of birational geometry. This condition is essential for Theorem \ref{thm_main}, as illustrated by the following examples:      
    \begin{enumerate}  
    	\item \emph{The presence of a degenerating fiber.} Let $f: X \to \bP^1$ be a Lefschetz pencil with $S \subset \bP^1$ denoting the set of its critical values, such that the general fibers of $f$ are canonically polarized $d$-folds with volume $v$. Then $f: (X, 0), 0 \to \bP^1$ forms a family of $(d, \Phi_0, \{1\}, v)$-stable minimal models. While $\bP^1$ is not hyperbolic, $\bP^1 \setminus S$ is hyperbolic ($|S| \geq 3$, as shown in \cite{VZ2001}).  
    	\item \emph{The presence of a degenerating polarization.} Let $E$ be an elliptic curve and $x_0 \in E(\bC)$. Define $X = E \times E$, and let $f: E \times E \to E$ denote the projection onto the first component. Let $A = \frac{1}{2}(E \times \{x_0\} \cup \Delta_E)$, where $\Delta_E \subset E \times E$ is the diagonal. Then $f: (X, 0), A \to E$ forms a family of $(1, \Phi_{\frac{1}{2}}, \{1\}, t)$-stable minimal models. Although the underlying family of elliptic curves is trivial, the family of polarizations is non-isotrivial and degenerates at $x_0$. Consequently, the entire base $E$ is not hyperbolic, whereas $E \setminus \{x_0\}$ is hyperbolic.  
    \end{enumerate} 
	\begin{exmp}[KSBA-stable families]
		A projective pair $(X, B)$ is said to be a KSBA-stable pair if $(X, B), 0$ is a stable minimal model. A family of KSBA-stable varieties is considered admissible if it admits a simple normal crossing log birational model. The base of a maximally variational admissible family of KSBA pairs satisfies the hyperbolicity properties stated in Theorem \ref{thm_main}.
	\end{exmp}
\subsection{Stratified hyperbolicity of moduli of stable minimal models}
Given a functorial semi-log desingularization procedure (c.f. \cite{Bierstone2013}, \cite[\S 10.4]{Kollar2013}), one can associate with $\sM_{\rm slc}(d,\Phi_c,\Gamma,\sigma)$ a stratification such that over each stratum $\sS$, the universal family is strictly birationally admissible (i.e., any base change admits a simple normal crossing log birational model). Such a stratification is referred to as a \emph{birationally admissible stratification} (Definition \ref{defn_BAstr}). A birationally admissible stratification always exists but is typically not unique. It is determined by the choice of the functorial desingularization functor.  A typical example arises in the KSBA compactification of the moduli space of principally polarized manifolds. In this case, the open substack parameterizing smooth fibers constitutes a birationally admissible stratum, while the birationally admissible stratification on the boundary is determined by the semi-log desingularization process. 

Let $S$ be a smooth variety. We say that $S$ is of \emph{log general type} if there exists a smooth compactification $\overline{S} \supseteq S$ such that $D = \overline{S} \setminus S$ is a divisor and $K_{\overline{S}} + D$ is big. The condition of being of log general type is independent of the choice of compactification.  
A smooth Deligne-Mumford stack $\sS$ is said to be \emph{hyperbolic} if, for every generically finite morphism $S \to \sS$ from a quasi-projective smooth variety $S$, the variety $S$ is of log general type.

A direct corollary of Theorem \ref{thm_main} is the following result.  
\begin{thm}\label{thm_main_moduli}
Each stratum of a birationally admissible stratification of $\sM_{\rm slc}(d, \Phi_c, \Gamma, \sigma)$ is hyperbolic.  
\end{thm}
For $\overline{\sM}_{g,n}$, the moduli stack of genus $g$ stable curves with $n$ marked points satisfying $2g - 2 + n > 0$, the boundary simple normal crossing divisor $\partial\sM_{g,n} := \overline{\sM}_{g,n} \setminus \sM_{g,n}$ induces a \emph{canonical stratification}. The distinct strata of this stratification correspond precisely to the distinct topological types of the fibers. This stratification is birationally admissible. Consequently, we derive the following result, which is novel as it elucidates the hyperbolicity of the strata on the boundary.  
\begin{cor}  
	Every stratum of the canonical stratification of $\overline{\sM}_{g,n}$ is hyperbolic. In particular, $\sM_{g,n}$ is hyperbolic.  
\end{cor}
\subsection{Historical accounts}
	Significant efforts have been devoted to investigating the hyperbolicity properties of various family of polarized projective manifolds (which are automatically admissible). For an incomplete list, see \cite{Kovacs1996,Kovacs2000,OV2001,VZ2001,Kovacs2002,Kovacs2003,KK2008,KK20082,KK2010,Patakfalvi2012,PS17,CP2019,Taji2021,WeiWu2023}.  
	
	In the case of families of canonically polarized manifolds (known as the Viehweg hyperbolicity conjecture), the result was first established by Campana-P\v{a}un \cite{CP2019}, building upon their work on foliations and the Viehweg-Zuo sheaves constructed in \cite{VZ2003}. By employing Saito's theory of Hodge modules \cite{MSaito1988}, Popa-Schnell \cite{PS17} extended the construction of Viehweg-Zuo sheaves to base spaces of arbitrary families of projective varieties with a geometric generic fiber admitting a good minimal model, thereby proving the corresponding hyperbolicity property for these families. Popa-Schnell's result was subsequently generalized by Wei-Wu \cite{WeiWu2023} to log smooth families of pairs of log general type. Additionally, there exist other admissible conditions ensuring hyperbolicity; see, for example, \cite{Kovacs2021,Park2022}.
	\subsection{Strategy and organization}
	The main strategy of this paper builds upon the work of Viehweg-Zuo \cite{VZ2001,VZ2003}, employing variations of Hodge structures to construct a Higgs sheaf, while incorporating the foliation theory developed by Campana-P\v{a}un \cite{CP2019}. A significant challenge arises from the presence of singularities and boundary divisors on the fibers, even after performing semi-resolutions of singularities. This motivates us to utilize variations of mixed Hodge structures to analyze families of pairs. However, there exist multiple types of mixed Hodge structures associated with singular pairs, such as singular cohomology, intersection cohomology, cohomology with compact supports, Borel-Moore cohomology, and others. Inspired by Fujino-Fujisawa's work \cite{Fujino2014}, which employs cohomology with compact supports to study families of simple normal crossing pairs, leading to the proof of the projectivity of the coarse moduli space of KSBA pairs \cite{Fujino2018,Kovacs2017}, we choose in this paper the Borel-Moore cohomology (the dual of cohomology with compact supports), along with its mixed Hodge structure and its variation within families, to construct a Higgs sheaf of Viehweg-Zuo type. 
	
	Given a datum $(d, \Phi_c, \Gamma, \sigma)$, an integer $r \gg 1$, and a rational number $0 < a \ll 1$, the coherent sheaves $f_\ast(\sO_X(rK_{X/S} + rB + raA))$ associated with families of $(d, \Phi_c, \Gamma, \sigma)$-stable minimal models $f: (X, B), A \to S$ over reduced base schemes $S$ yield a locally free coherent sheaf on the Deligne-Mumford stack $\sM_{\rm slc}(d, \Phi_c, \Gamma, \sigma)$. The determinant line bundle descends to an ample $\bQ$-line bundle on the coarse moduli space (\cite{Fujino2018,Kovacs2017}), thereby ensuring the projectivity of the coarse moduli space via Koll\'ar's criterion \cite{Kollar1990}.  
	
	The main contribution of this paper is the observation that the Koll\'ar-type polarization $\det f_\ast(\sO_X(rK_{X/S} + rB + raA))$ encapsulates substantial information about the moduli stack. More precisely, it is demonstrated that certain admissible variations of Borel-Moore cohomology contain $\det f_\ast(\sO_X(rK_{X/S} + rB + raA))$. This variation of mixed Hodge structures encodes critical information about the family. In particular the associated logarithmic Higgs bundle induces a non-zero morphism  
	$$
	\det f_\ast(\sO_X(rK_{X/S} + rB + raA)) \otimes K \to \Omega^{\otimes k}_S(\log D),
	$$  
	where $m > 0$ is some integer, $D$ is a simple normal crossing divisor on $S$ over which $f$ fails to be birationally admissible, and $K$ is a weakly positive coherent sheaf. As a result, $K_S + D$ is big due to the work of Campana-P\v{a}un \cite{CP2019}. In a subsequent paper, we will utilize this Viehweg-Zuo type Higgs sheaf to explore additional global geometric properties of $\sM_{\rm slc}(d, \Phi_c, \Gamma, \sigma)$.
	
	In Section 2, we conduct an in-depth analysis of the Borel-Moore cohomology of simple normal crossing projective pairs, their mixed Hodge structures, and their variations within families. The primary result of this section is that the Borel-Moore cohomologies of a projective simple normal crossing family of pairs carry a graded polarizable variation of mixed Hodge structures (Theorem \ref{thm_VMHS_BM}).
	
	In Section 3, we examine the lower canonical extension of the variation of Borel-Moore cohomologies. One of the key results of this section is the admissibility of the variation of Borel-Moore cohomologies associated with a simple normal crossing family of projective pairs (Theorem \ref{thm_VMHS_BM_admissible}). Although Fujino-Fujisawa \cite{Fujino2014} have established the theory of variations of cohomologies with compact supports, we still require a detailed investigation of the variation of Borel-Moore cohomology due to the necessity of constructing its Deligne's lower canonical extension geometrically (Theorem \ref{thm_log_Higgs_geo}). In this section, we do not assume unipotency of local monodromies, which enhances applicability.
	
	In Section 4, we provide a general construction of a Higgs sheaf of Viehweg-Zuo type associated with a projective morphism from a simple normal crossing pair to a smooth variety.
	
	In Section 5, we review the foundational constructions of Birkar \cite{Birkar2022} on stable minimal models and their moduli. Subsequently, we introduce the birationally admissible condition on families of stable minimal models (Definition \ref{defn_admissible}) and construct a Higgs sheaf of Viehweg-Zuo type associated with an admissible family of stable minimal models (Theorem \ref{thm_big_Higgs_sheaf}).
	
	The proofs of Theorem \ref{thm_main} and Theorem \ref{thm_main_moduli} are presented in Section 6.
    \section{Variation of Borel-Moore cohomology}
    \subsection{Basic notations}
    \begin{defn}[Simple normal crossing pair]\label{defn_SNC_family}
    	Let $S$ and $X$ be reduced schemes of finite type over $\operatorname{Spec}(\mathbb{C})$, and let $D \geq 0$ be a Weil $\bQ$-divisor on $X$. Let $f: X \to S$ be a morphism. The morphism $f: (X, D) \to S$ is said to be \emph{simple normal crossing over $S$} at a point $x \in X$ if there exists a Zariski open neighborhood $U$ of $x$ in $X$ that can be embedded into a scheme $Y$, which is smooth over $S$. In this embedding, $Y$ admits a regular system of parameters $(z_1, \dots, z_p, y_1, \dots, y_r)$ over $S$ at the point corresponding to $x = 0$, such that $U$ is defined by the monomial equation $z_1 \cdots z_p = 0$ and  
    	\[ D|_U = \sum_{i=1}^r a_i (y_i = 0)|_U, \quad \text{where } a_i \geq 0, \]  
    	over $S$.
    	
    	The pair $f: (X, D) \to S$ is said to be a \emph{simple normal crossing family over $S$} if it satisfies the simple normal crossing condition over $S$ at every point of $X$. A family is referred to as a \emph{log smooth family over $S$} if it is a simple normal crossing family over $S$ and each fiber $X_s := f^{-1}(s)$, for $s \in S$, is smooth (not necessarily connected).
    	
    	In the special case where $S = \operatorname{Spec}(\mathbb{C})$, the pair $(X, D)$ is referred to as a \emph{simple normal crossing pair} if $(X, D)$ forms a simple normal crossing family over $\operatorname{Spec}(\mathbb{C})$. In this context, $X$ has Gorenstein singularities and possesses an invertible dualizing sheaf $\omega_X$. The canonical divisor $K_X$ is defined up to linear equivalence via the isomorphism $\omega_X \simeq \sO_X(K_X)$.
    	
    	A simple normal crossing pair $(X, D)$ is called a \emph{log smooth pair} if $X$ is smooth (note that $X$ is not necessarily connected). In this case, the support $\operatorname{Supp}(D)$ has simple normal crossings.
    \end{defn}
\begin{defn}[Semistable morphism]\label{defn_semistable_cod1}
	Let $f: (X, \Delta) \to (S, D_S)$ be a morphism from a simple normal crossing pair $(X, \Delta)$ to a log smooth pair $(S, D_S)$. The morphism $f$ is called \emph{semistable} if the following conditions are satisfied:
	\begin{enumerate}
		\item $D_S$ is a reduced, smooth (not necessarily connected) divisor on $S$ such that $\operatorname{Supp}(f^{-1}(D_S))$ is a simple normal crossing divisor contained in $\Delta_{\text{red}}$. Note that the components of $f^{-1}(D_S)$ may have multiplicities.
		\item $f: (X, \Delta) \to S$ is a simple normal crossing family over $S \setminus D_S$.
		\item Each stratum of $(X, \Delta)$ lying in $f^{-1}(D_S)$ is mapped submersively by $f$ onto an irreducible component of $D_S$.
	\end{enumerate}
\end{defn}
\begin{rmk}\label{rmk_semi_log_resolution}
	Let $X$ be a reduced scheme of finite type over $\operatorname{Spec}(\mathbb{C})$, and let $D \geq 0$ be a Weil divisor on $X$. Let $S$ be a smooth variety and $f:(X,D)\to S$ be a projective morphism such that each stratum of $X$ dominates $S$. 
	By \cite{Bierstone2013} or \cite[\S 10.4]{Kollar2013}, there is a \emph{semi-log resolution} of singularities $\mu:(X',\Delta')\to(X,\Delta)$ such that the following conditions hold.
	\begin{itemize}
		\item $(X',\Delta')$ is a simple normal crossing pair.
		\item Let $X_{\rm sn}\subset X$ be the maximal Zariski open subset such that $(X_{\rm sn},X_{\rm sn}\cap D)$ is a simple normal crossing pair. $\mu$ is an isomorphism over $X_{\rm sn}$, and $\Delta'$ is the strict transform of $\Delta$.
		\item There is a dense Zariski open subset $U\subset S$ and a smooth (but not necessarily connected) divisor $D_U$ on $U$ such that ${\rm codim}_S(S\backslash U)\geq 2$ and $(f^{-1}(U),{\rm supp}(\Delta)\cap f^{-1}(U))\to (U,D_U)$ is a semistable morphism.
	\end{itemize}
\end{rmk}
    \subsection{The mixed Hodge structure on a log smooth projective pair}
    Let $M$ be a complex manifold. Let $DFW(M)$ denote the derived bi-filtered category on $M$ (cf. \cite{Deligne1971, Deligne1974, Steenbrink2008}). Consider an object $H = (K^\bullet, F^\bullet, W_\bullet) \in DFW(M)$ and let $m \in \mathbb{Z}$. The shift of $H$ by $m$, denoted as $H[m]$, is defined as  $$H[m]=(K^\bullet[m],F^\bullet(H[m]),W_\bullet(H[m]))$$ where
    $$F^p(H[m])=F^p[m],\quad W_k(H[m])=W_{k-m}[m].$$
    Its Tate twist $H(m)$ is defined as 
    $$((2\pi\sqrt{-1})^{2m}K^\bullet,F^\bullet(H(m)),W_\bullet(H(m)))$$ where
    $$F^p(H(m))=F^{p+m},\quad W_k(H(m))=W_{k+2m}.$$
    
    Let $(X, D)$ be a log smooth projective pair where $X$ is a projective smooth variety and $D$ is a simple normal crossing divisor on $X$. We now recall the mixed Hodge complex associated with $(X, D)$.
    
    Let $X^o := X \setminus D$, and let $j: X^o \to X$ be the open immersion. The rational structure is given by the complex of constructible sheaves $Rj_\ast \mathbb{Q}_{X^o}$ together with the canonical filtration (the rational weight filtration) $\tau = \{\tau_m := \tau^{\leq m}(Rj_\ast \mathbb{Q}_{X^o})\}_{m\in\bZ}$.
    
    Let $\Omega^\bullet_X(\log D)$ be the logarithmic de Rham complex of $(X, D)$. This complex is endowed with the naive filtration (the Hodge filtration) $F^\bullet = \{F^p := \Omega^{\geq p}_X(\log D)\}_{p \geq 0}$ and the weight filtration $W_\bullet$, characterized by
    \begin{align*}
    W_m \Omega^p_X(\log D) = \begin{cases}
    0, & m < 0 \\
    \Omega^p_X(\log D), & m \geq p \\
    \Omega^{p-m}_X \wedge \Omega^m_X(\log D), & 0 \leq m \leq p
    \end{cases}.
    \end{align*}
    \begin{prop}\emph{(\cite[\S 4.3]{Steenbrink2008})}\label{prop_logsmooth_mHC}
    	$\left((Rj_\ast\bQ_{X^o},\tau),(\Omega^\bullet_X(\log D),W_\bullet,F^\bullet)\right)$ is a mixed Hodge complex of sheaves. 
    \end{prop}
    Let $D = \bigcup_{i \in I} D_i$. For every non-empty subset $J \subset I$ and integer $k > 0$, denote
    \[ D_J = \bigcap_{i \in J} D_i \quad \text{and} \quad D_{[k]} = \coprod_{\substack{J \subset I \\ |J| = k}} D_J. \]
    We also set $D_{[0]} = X$. The Hodge filtration $F^\bullet$ induces a filtration on $\operatorname{Gr}^W_m \Omega^\bullet_X(\log D)$, which is denoted by $F^\bullet \operatorname{Gr}^W_m \Omega^\bullet_X(\log D)$.
    \begin{lem}\emph{(\cite[\S 4.3]{Steenbrink2008})}\label{lem_mHC_logsmoothpair}
    	The filtered complex $(\operatorname{Gr}^W_m \Omega^\bullet_X(\log D), F^\bullet \operatorname{Gr}^W_m \Omega^\bullet_X(\log D))$ is filtered quasi-isomorphic to $(\Omega^\bullet_{D_{[m]}}, F^\bullet)(-m)[-m]$, where $(\Omega^\bullet_{D_{[m]}}, F^\bullet)$ is the standard Hodge complex of sheaves on $D_{[m]}$. In particular, there is a natural isomorphism
    	\[ \operatorname{Gr}^p_F \operatorname{Gr}^W_m \Omega^\bullet_X(\log D) \simeq \Omega^{p-m}_{D_{[m]}}[-p]. \]
    \end{lem} 
    \subsection{The Gysin map}\label{section_Gysin_map}
    Let $(X,D)$ be a log smooth pair, and let $Y$ be a smooth divisor on $X$ that intersects transversally with every stratum of $D$. Consequently, $(Y, D \cap Y)$ forms a log smooth pair. For notational convenience, we use $j$ to denote both the open immersions $X^o := X \backslash D \to X$ and $Y^o := Y \backslash D \to Y$. In this section, we establish that the Gysin map $Rj_\ast\bQ_{Y^o}(-1)[-2] \to Rj_\ast\bQ_{X^o}$, which is the Verdier dual of the restriction map $Rj_\ast\bQ_{X^o} \to Rj_\ast\bQ_{Y^o}$, constitutes a morphism between mixed Hodge complexes of sheaves.
    
    Consider the residue sequence
    \begin{align}\label{align_residue_sequence}
    	0 \to \Omega^\bullet_{X}(\log D) \stackrel{\iota}{\to} \Omega^\bullet_{X}(\log (D+Y)) \stackrel{{\rm Res}_{Y}}{\to} \Omega^\bullet_{Y}(\log (D \cap Y))[-1] \to 0.
    \end{align}
    This sequence preserves the Hodge filtrations and weight filtrations in the following sense:
    \[
    \iota(F^p) \subset F^p, \quad \iota(W_m) \subset W_m,
    \]
    and
    \[
    {\rm Res}_{Y}(F^p) \subset F^{p-1}[-1], \quad {\rm Res}_{Y}(W_m) \subset W_{m-1}[-1].
    \]
    \begin{lem}
    	The sequence (\ref{align_residue_sequence}) is bi-filtered exact, meaning that the induced sequence
    	\[
    	0 \to {\rm Gr}^p_F {\rm Gr}^W_m \Omega^\bullet_{X}(\log D) \stackrel{\iota}{\to} {\rm Gr}^p_F {\rm Gr}^W_m \Omega^\bullet_{X}(\log (D+Y)) \stackrel{{\rm Res}_{Y}}{\to} {\rm Gr}^{p-1}_F {\rm Gr}^W_{m-1} \Omega^\bullet_{Y}(\log (D \cap Y))[-1] \to 0
    	\]
    	is exact for all $m$ and $p$.
    \end{lem}
    \begin{proof}
    	According to Lemma \ref{lem_mHC_logsmoothpair}, the sequence is isomorphic to
    	\[
    	0 \to \Omega^{p-m}_{D_{[m]}}[-p] \to \Omega^{p-m}_{D_{[m]}}[-p] \oplus \Omega^{p-m}_{Y \cap D_{[m-1]}}[-p] \to \Omega^{p-m}_{Y \cap D_{[m-1]}}[-p] \to 0,
    	\]
    	which is exact for all $m$ and $p$.
    \end{proof}
    As a consequence, (\ref{align_residue_sequence}) yields a distinguished triangle
    \[
    \Omega^\bullet_{X}(\log D) \stackrel{\iota}{\to} \Omega^\bullet_{X}(\log (D+Y)) \stackrel{{\rm Res}_{Y}}{\to} \Omega^\bullet_{Y}(\log (D \cap Y))(-1)[-1] \to \Omega^\bullet_{X}(\log D)[1]
    \]
    in $DFW(X)$. The boundary map induces the Gysin map
    \[
    \Omega^\bullet_{Y}(\log (D \cap Y))(-1)[-2] \to \Omega^\bullet_{X}(\log D)
    \]
    in $DFW(X)$, which is compatible with the topological Gysin map
    \[
    Rj_\ast\bQ_{Y^o}(-1)[-2] \to Rj_\ast\bQ_{X^o}.
    \]
    \subsection{Borel-Moore cohomology}\label{section_BM_coh}
    Let $X^o$ be an algebraic variety of pure dimension $n$. Its Borel-Moore cohomology is defined by
    \[
    H^i_{\rm BM}(X^o, \bQ) := H^{2n-i}_c(X^o, \bQ)^\ast, \quad \forall i,
    \]
    where $H^\ast_c$ denotes the compactly supported cohomology. Now assume that $X^o = X \backslash D$, where $(X, D)$ is a simple normal crossing projective pair. We are particularly interested in its Borel-Moore cohomology because $H^0(X, \sO_X(K_X + D))$ embeds into $H^i_{\rm BM}(X^o, \bQ) \otimes \bC$ as the lowest Hodge piece, as observed by Fujino-Fujisawa \cite{Fujino2014}. We will elaborate on this observation, as it provides insight into the construction we employ in the Viehweg-Zuo Higgs sheaves associated with a simple normal crossing family.
    
    For the remainder of this section, we assume that $(X, D)$ is a simple normal crossing projective pair, where $X$ is of pure dimension $n$. Let $X^o := X \backslash D$, and let $j: X^o \to X$ be the open immersion. Denote the irreducible decompositions as follows:
    \[
    X = \bigcup_{i \in I} X_i \quad \text{and} \quad D = \bigcup_{\lambda \in \Lambda} D_\lambda,
    \]
    where $I$ and $\Lambda$ are finite sets. For every non-empty subset $J \subset I$ and integer $k > 0$, define
    \[
    X_J := \bigcap_{i \in J} X_i \quad \text{and} \quad X_{[k]} := \coprod_{\substack{J \subset I \\ |J| = k}} X_J.
    \]
    Let $X^o_J := X_J \backslash D$ for $\emptyset \neq J \subset I$, and set $X^o_{[k]} := \coprod_{\substack{J \subset I \\ |J| = k}} X^o_J$. Denote by $i^o_{[k]}: X^o_{[k]} \to X^o$ the natural morphism. We then have the following quasi-isomorphism:
    \begin{align}\label{align_const_resolution}
    \bQ_{X^o} \simeq_{\rm qis} i^o_{[1]\ast}\bQ_{X^o_{[1]}} \to i^o_{[2]\ast}\bQ_{X^o_{[2]}} \to \cdots,
    \end{align}
    where the map $\bQ_{X^o_{[k]}} \to \bQ_{X^o_{[k+1]}}$ is the alternating sum of restriction maps for each $k>0$. To simplify notation, we still use $j$ to denote the immersions $X^o_{[k]} \to X_{[k]}$. Applying $Rj_!$ on both sides of (\ref{align_const_resolution}) yields
    \begin{align}\label{align_cpt_complex}
    	Rj_!\bQ_{X^o}\simeq_{\rm qis}{\rm Tot}\left(Rj_!i^o_{[1]\ast}\bQ_{X^o_{[1]}}\to Rj_!i^o_{[2]\ast}\bQ_{X^o_{[2]}}\to\cdots\right).
    \end{align}
    Note that
    \[
    H^i_{c}(X^o, \bQ) \simeq \bH^i(X, Rj_!\bQ_{X^o}).
    \]
    By Verdier duality, we have
    \[
    H^i_{\rm BM}(X^o, \bQ) \simeq \bH^i(X, \bD_X(Rj_!\bQ_{X^o}[2n])),
    \]
    where $\bD_X$ denotes the Verdier duality operator on $D(X)$. Applying Verdier duality to (\ref{align_cpt_complex}), we obtain
    \begin{align}\label{align_BM_complex_resolve}
    	\bD_X(Rj_!\bQ_{X^o}[2n]) &\simeq_{\rm qis} {\rm Tot}\left(\cdots \to \bD_X Rj_! i^o_{[2]\ast} \bQ_{X^o_{[2]}} \to \bD_X Rj_!i^o_{[1]\ast}  \bQ_{X^o_{[1]}}\right)[-2n] \\\nonumber 
    	&\simeq_{\rm qis} {\rm Tot}\left(\cdots \to Rj_\ast i^o_{[2]\ast} \bD_{X^o_{[2]}} \bQ_{X^o_{[2]}} \to Rj_\ast i^o_{[1]\ast} \bD_{X^o_{[1]}} \bQ_{X^o_{[1]}}\right)[-2n] \\\nonumber 
    	&\simeq_{\rm qis} {\rm Tot}\left(\cdots \to i_{[2]\ast} Rj_\ast \bQ_{X^o_{[2]}}(-1)[-2] \to i_{[1]\ast} Rj_\ast \bQ_{X^o_{[1]}}\right),
    \end{align}
    where $i_{[k]}: X_{[k]} \to X$ ($k > 0$) is the natural morphism, and $\bQ_{X^o_{[k+1]}}[-2k] \to \bQ_{X^o_{[k]}}[-2k+2]$ ($k > 0$) is the alternating sum of Gysin maps. Note that $i_{[1]\ast} Rj_\ast \bQ_{X^o_{[1]}}$ lies in the 0th term of the complex on the right-hand side of (\ref{align_BM_complex_resolve}).
    \subsection{Mixed Hodge structure on the Borel-Moore cohomology}\label{section_BM_absolute}
    (\ref{align_BM_complex_resolve}) suggests the canonical mixed Hodge structure on $H^i_{\rm BM}(X^o,\bQ)$. In fact,  the complex
    \begin{align}\label{align_BM_complex}
    \cdots\to i_{[3]\ast}Rj_\ast\bQ_{X^o_{[3]}}(-2)[-4]\to i_{[2]\ast}Rj_\ast\bQ_{X^o_{[2]}}(-1)[-2]\to i_{[1]\ast} Rj_\ast\bQ_{X^o_{[1]}}
    \end{align}
    (where $i_{[1]\ast} Rj_\ast\bQ_{X^o_{[1]}}$ is placed in the 0th term) underlies a complex of mixed Hodge complexes (or a complex of mixed Hodge modules).
    \subsubsection{The rational structure and the weight filtration}
    For each $k \geq 0$, let $\tau_{[k]}$ denote the canonical filtration of $Rj_\ast\bQ_{X^o_{[k]}}$, defined by $\tau_{[k],i} := \tau^{\leq i}(Rj_\ast\bQ_{X^o_{[k]}})$. These filtrations induce an increasing filtration (the rational weight filtration) $\tau_\bullet$ on $\bD_X(Rj_!\bQ_{X^o})$, via (\ref{align_BM_complex_resolve}), by defining
    \begin{align}\label{align_BM_weight_filtation}
    \tau_m&:={\rm Tot}\left(\cdots\to i_{[2]\ast}\tau^{\leq m-1 }(Rj_\ast\bQ_{X^o_{[2]}}(-1)[-2])\to i_{[1]\ast}\tau^{\leq m}(Rj_\ast\bQ_{X^o_{[1]}})\right)
    \end{align}
    \subsubsection{The Hodge filtration}
    Let $\emptyset \neq K \subset J \subset I$ such that $|K| = |J| - 1$, and let $D_K := X_K \backslash X^o_K$. Note that $X_J$ is a smooth divisor on $X_K$, and both $D_K$ and $X_J + D_K$ are simple normal crossing divisors on $X_K$. Consider the Gysin map
    \begin{align*}
    \Omega^\bullet_{X_J}(\log X_J\cap D_K)(-1)[-2]\to \Omega^\bullet_{X_K}(\log D_K)
    \end{align*}
    in $DFW(X)$, constructed in \S \ref{section_Gysin_map}. Denote $$D_{[k]}:=X_{[k]}\backslash X^o_{[k]}=\coprod_{\substack{J\subset I,\\|J|=k}}D_J,$$ a simple normal crossing divisor on $X_{[k]}$.   
    The alternative sum of the Gysin maps provides a complex
    \begin{align*}
    \Omega^\bullet_{(X,D),{\rm BM}}:={\rm Tot}\left(\cdots\to i_{[2]\ast}\Omega^\bullet_{X_{[2]}}(\log D_{[2]})[-2]\to i_{[1]\ast}\Omega^\bullet_{X_{[1]}}(\log D_{[1]})\right).
    \end{align*}
    Define the weight filtration $W_{\bullet}$ by
    \begin{align*}
    W_{m}&:={\rm Tot}\left(\cdots\to i_{[2]\ast}W_{m-1}(\Omega^\bullet_{X_{[2]}}(\log D_{[2]})(-1)[-2])\to i_{[1]\ast}W_m\Omega^\bullet_{X_{[1]}}(\log D_{[1]})\right)\\\nonumber
    &={\rm Tot}\left(\cdots\to i_{[2]\ast}W_{m-1}\Omega^\bullet_{X_{[2]}}(\log D_{[2]})[-2]\to i_{[1]\ast}W_m\Omega^\bullet_{X_{[1]}}(\log D_{[1]})\right).
    \end{align*}
    Define its Hodge filtration $F^\bullet$ by
    \begin{align}\label{align_Hodge_filtration}
    F^p&:={\rm Tot}\left(\cdots\to i_{[2]\ast}F^p(\Omega^{\bullet}_{X_{[2]}}(\log D_{[2]})(-1)[-2])\to i_{[1]\ast}F^p(\Omega^{\bullet}_{X_{[1]}}(\log D_{[1]}))\right)\\\nonumber
    &={\rm Tot}\left(\cdots\to i_{[2]\ast}\Omega^{\geq p-1}_{X_{[2]}}(\log D_{[2]})[-2]\to i_{[1]\ast}\Omega^{\geq p}_{X_{[1]}}(\log D_{[1]})\right).
    \end{align}
    Let $D_{[k]} = \bigcup_{i \in I_k} D_{[k],i}$ be the irreducible decomposition. For every non-empty subset $J \subset I$ and integer $l > 0$, denote
    $$D_{[k],J}=\bigcap_{i\in J}D_{[k],i}\quad\textrm{and}\quad D_{[k],[l]}=\coprod_{\substack{J\subset I,\\|J|=l}}D_{[k],J}.$$
    We also set $D_{[k],[0]} := X_{[k]}$.
    According to Lemma \ref{lem_mHC_logsmoothpair}, one observes that
    \begin{align}\label{align_BM_GrW}
    {\rm Gr}^W_m\Omega^\bullet_{(X,D),{\rm BM}}&\simeq_{\rm qis}\bigoplus_{i\geq0}{\rm Gr}^W_{m-i}(\Omega^{\bullet}_{X_{[i+1]}}(\log D_{[i+1]})(-i)[-2i])[i]\\\nonumber
    &\simeq_{\rm qis}\bigoplus_{i\geq0}{\rm Gr}^W_{m-i}(\Omega^{\bullet}_{X_{[i+1]}}(\log D_{[i+1]}))(-i)[-i]\\\nonumber
    &\simeq_{\rm qis}\bigoplus_{i\geq0}\Omega^\bullet_{D_{[i+1],[m-i]}}(-m)[-m].
    \end{align}
    This is a Hodge complex of sheaves with weight $m$, whose induced Hodge filtration is given by
    \begin{align}\label{align_BM_GrW_F}
    F^p({\rm Gr}^W_m\Omega^\bullet_{(X,D),{\rm BM}})&\simeq_{\rm qis}\bigoplus_{i\geq0}F^p{\rm Gr}^W_{m-i}(\Omega^{\bullet}_{X_{[i+1]}}(\log D_{[i+1]})(-i)[-2i])[i]\\\nonumber
    &\simeq_{\rm qis}\bigoplus_{i\geq0}F^p(\Omega^\bullet_{D_{[i+1],[m-i]}}(-m)[-m])\\\nonumber
    &\simeq_{\rm qis}\bigoplus_{i\geq0}\Omega^{\geq p-m}_{D_{[i+1],[m-i]}}[-m].
    \end{align}
    Therefore
    \begin{align*}
    {\rm Gr}^p_F{\rm Gr}^W_m\Omega^\bullet_{(X,D),{\rm BM}}\simeq_{\rm qis}\bigoplus_{i\geq0}\Omega^{p-m}_{D_{[i+1],[m-i]}}[-p].
    \end{align*}
    \begin{lem}\label{lem_mHC_logsmooth}
    	$\left((\bD_X(Rj_!\bQ_{X^o}),\tau_\bullet),(\Omega^\bullet_{(X,D),{\rm BM}},W_{\bullet},F^\bullet)\right)$ is a mixed Hodge complex of sheaves.
    \end{lem}
    \begin{proof}
    	According to Proposition \ref{prop_logsmooth_mHC}, $(Rj_\ast\bQ_{X^o_{[k]}}, \tau_{[k],\bullet}) \otimes_{\bQ} \bC$ is canonically filtered quasi-isomorphic to $(\Omega^\bullet_{X_{[k]}}(\log D_{[k]}), W_\bullet \Omega^\bullet_{X_{[k]}}(\log D_{[k]}))$. Consequently, 
    	\[
    	(\bD_X(Rj_!\bQ_{X^o}), \tau_\bullet) \otimes_{\bQ} \bC
    	\]
    	is filtered quasi-isomorphic to 
    	\[
    	(\Omega^\bullet_{(X,D),{\rm BM}}, W_\bullet).
    	\]
    	By (\ref{align_BM_GrW}) and (\ref{align_BM_GrW_F}), 
    	\[
    	({\rm Gr}^W_m \Omega^\bullet_{(X,D),{\rm BM}}, F^\bullet({\rm Gr}^W_m \Omega^\bullet_{(X,D),{\rm BM}}))
    	\]
    	is a Hodge complex of sheaves of weight $m$ for any $m \in \bZ$. It remains to show that the spectral sequence
    	\begin{align}\label{align_BM_ss_weightm}
    	E^{p,q}_1:=\bH^{p+q}(X,{\rm Gr}^p_F{\rm Gr}^W_m\Omega^\bullet_{(X,D),{\rm BM}})\Rightarrow \bH^{p,q}(X, {\rm Gr}^W_m\Omega^\bullet_{(X,D),{\rm BM}})
    	\end{align}
    	degenerates at the $E_1$-page for every $m$.
    	
    	By (\ref{align_BM_GrW}) and (\ref{align_BM_GrW_F}), the spectral sequence (\ref{align_BM_ss_weightm}) decomposes into a direct sum of spectral sequences
    	\[
    	E_1^{p,q}:=H^{q-p}(X, \Omega^{p-m}_{D_{[i+1],[m-i]}}) \Rightarrow H^{q-m}(X, \Omega^\bullet_{D_{[i+1],[m-i]}})
    	\]
    	for $0 \leq i \leq m$. Each of these spectral sequences degenerates at the $E_1$-page because each $D_{[i+1],[m-i]}$ is a smooth projective variety. Therefore, the spectral sequence (\ref{align_BM_ss_weightm}) degenerates at the $E_1$-page, completing the proof of the lemma.
    \end{proof}
    \subsubsection{Top Hodge piece}
    According to the construction of the Hodge filtration (\ref{align_Hodge_filtration}), the top Hodge piece is given by
    \begin{align}\label{align_top_Hodge_ab}
    F^n_{(X,D),{\rm BM}}:=\cdots\to  i_{[2]\ast}\omega_{X_{[2]}}(D_{[2]})[-1-n]\to i_{[1]\ast}\omega_{X_{[1]}}(D_{[1]})[-n],
    \end{align}
    \begin{lem}\label{lem_top_Hodge_BM}
    	$F^n_{(X,D),{\rm BM}}\simeq_{\rm qis}\sO_X(K_X+D)[-n]$.
    \end{lem}
    \begin{proof}
    	Denote by $\bD^{\rm coh}_X(-)=Rhom_{\sO_X}(-,\omega^\bullet_X)$ the duality operator on the derived category of coherent sheaves on $X$, where $\omega^\bullet_X\simeq_{\rm qis}\sO_X(K_X)[n]$ represents the dualizing complex of $X$. Then the Grothendieck-Serre duality theorem implies that
    	\begin{align*}
    	F^n_{(X,D),{\rm BM}}\simeq_{\rm qis}&\cdots\to i_{[2]\ast}\omega_{X_{[2]}}(D_{[2]})[-1-n]\to i_{[1]\ast}\omega_{X_{[1]}}(D_{[1]})[-n]\\\nonumber
    	\simeq_{\rm qis}&{\rm Tot}\left(\cdots\to i_{[2]\ast}\bD^{\rm coh}_{X_{[2]}}(\sO_{X_{[2]}}(-D_{[2]}))[-2n]\to i_{[1]\ast}\bD^{\rm coh}_{X_{[1]}}(\sO_{X_{[1]}}(-D_{[1]}))[-2n]\right)\\\nonumber
    	\simeq_{\rm qis}&{\rm Tot}\left(\cdots\to\bD^{\rm coh}_{X}(i_{[2]\ast}\sO_{X_{[2]}}(-D_{[2]}))\to \bD^{\rm coh}_{X}(i_{[1]\ast}\sO_{X_{[1]}}(-D_{[1]}))\right)[-2n]\\\nonumber
    	\simeq_{\rm qis}&\bD^{\rm coh}_{X}\left(i_{[1]\ast}\sO_{X_{[1]}}(-D_{[1]})\to i_{[2]\ast}\sO_{X_{[2]}}(-D_{[2]})\to\cdots\right)[-2n]\\\nonumber
    	\simeq_{\rm qis}&\bD^{\rm coh}_{X}\left(\sO_X(-D)\otimes^L_{\sO_X}(i_{[1]\ast}\sO_{X_{[1]}}\to i_{[2]\ast}\sO_{X_{[2]}}\to\cdots)\right)[-2n]\\\nonumber
    	\simeq_{\rm qis}&\bD^{\rm coh}_{X}\left(\sO_X(-D)\otimes^L_{\sO_X}\sO_X\right)[-2n]\\\nonumber
    	\simeq_{\rm qis}&RHom_{\sO_X}(\sO_X(-D),\omega^\bullet_X[-n])[-n]\\\nonumber
    	\simeq_{\rm qis}&\sO_X(K_X+D)[-n].
    	\end{align*}
    \end{proof}
    This relation is the primary motivation for considering Borel-Moore cohomology in this paper. It generalizes Fujino's observations \cite{Fujino2014,Fujino2018} to the sheaf-theoretic level.
    \subsection{Variation of Borel-Moore cohomology}\label{section_VMHS_BM}
    Let $f:(X,D) \to S$ be a projective simple normal crossing family over a smooth base variety $S$. Suppose $\dim S = d$ and $\dim X = n + d$. In this section, we establish that the local system determined by the Borel-Moore cohomology groups $H^i_{\rm BM}(X^o_y, \mathbb{Q})$, where $X^o_y := f^{-1}(\{y\}) \cap X^o$, carries a variation of mixed Hodge structures. The detailed constructions will be presented in the subsequent subsections.
    \subsubsection{Underlying complex of constructible sheaves}
    For simplicity, we denote $$\bV^{\rm BM}_{(X,D)}:=\bD_X(Rj_!\bQ_{X^o}[2n+2d])\quad\textrm{and}\quad \bV^k_{f,{\rm BM}}:=R^kf_\ast(\bV^{\rm BM}_{(X,D)}).$$
    According to (\ref{align_BM_complex_resolve}), one obtains
    \begin{align}\label{align_BM_local_system}
    \bV^k_{f,{\rm BM}}&\simeq_{\rm qis}R^kf_\ast\left({\rm Tot}\left(\cdots\to i_{[2]\ast}Rj_\ast\bQ_{X^o_{[2]}}(-1)[-2]\to i_{[1]\ast} Rj_\ast\bQ_{X^o_{[1]}}\right)\right)\\\nonumber
    &\simeq_{\rm qis}h^k{\rm Tot}\left(\cdots\to Rf_\ast Rj_\ast\bQ_{X^o_{[2]}}(-1)[-2]\to Rf_\ast Rj_\ast\bQ_{X^o_{[1]}}\right)
    \end{align}
    where $h^k$ denotes the $k$-th cohomology sheaf. 
    \begin{lem}
    	The sheaf $\bV^k_{f,{\rm BM}}$ is a local system for every $k$.
    \end{lem}
    \begin{proof}
    	Define the stupid filtration $L^\bullet$ of
    	$${\rm Tot}\left(\cdots\to i_{[2]\ast}Rj_\ast\bQ_{X^o_{[2]}}(-1)[-2]\to i_{[1]\ast} Rj_\ast\bQ_{X^o_{[1]}}\right)$$
    	by
    	$$L^p={\rm Tot}\left(\cdots\to i_{[p+2]\ast}Rj_\ast\bQ_{X^o_{[p+2]}}(-p-1)[-2p-2]\to i_{[p+1]\ast} Rj_\ast\bQ_{X^o_{[p+1]}}(-p)[-2p]\to0\to\cdots\to 0\right).$$ 
    	It gives rise to a spectral sequence
    	\begin{align}\label{align_Lerayss_BM}
    		E_1^{p,q}=R^{p+q}f_\ast({\rm Gr}^p_L(\bV^{\rm BM}_{(X,D)}))=R^qf_\ast Rj_\ast\bQ_{X^o_{[p+1]}}\Rightarrow \bV^{p+q}_{f,{\rm BM}}.
    	\end{align}  
    	Since $(X_{[k]},D_{[k]})$ is log smooth over $S$, the sheaf $R^qf_\ast Rj_\ast\bQ_{X^o_{[k]}}$ is a local system for every $k$ and $q$. Consequently, by employing (\ref{align_Lerayss_BM}), we obtain the lemma.
    \end{proof}
    Given $k \geq 0$, the local system $\bV^k_{f,{\rm BM}}$ induces a flat connection $(\cV^k_{f,{\rm BM}}, \nabla)$, where $\cV^k_{f,{\rm BM}} := \sO_S \otimes_{\bQ} \bV^k_{f,{\rm BM}}$ and $\nabla := d \otimes {\rm Id}$. Consider the complex $f^{-1}\Omega_S^\bullet \otimes_{\bQ} \bV^{\rm BM}_{(X,D)}$, equipped with a filtration $\{L^p\}_{p \geq 0}$, where $L^p := f^{-1}\Omega_S^{\geq p} \otimes_{\bQ} \bV^{\rm BM}_{(X,D)}$. Notice that
    \begin{align}\label{align_GrL}
    	L^p/L^{p+1} \simeq f^{-1}\Omega_S^p \otimes_{\bQ} \bV^{\rm BM}_{(X,D)}[-p].
    \end{align}
    Applying the functor $Rf_\ast$ to this filtered complex yields a spectral sequence
    \begin{align*}
    	E^{p,q}_1 := \Omega_S^p \otimes_{\bQ} R^qf_\ast(\bV^{\rm BM}_{(X,D)}) \Rightarrow R^{p+q}f_\ast(f^{-1}\Omega_S^\bullet \otimes_{\bQ}\bV^{\rm BM}_{(X,D)}).
    \end{align*}
    The morphisms of the $E_1$-terms give rise to
    $$\nabla:=d\otimes{\rm Id}:\Omega_S^p\otimes_{\bQ}R^qf_\ast(\bV^{\rm BM}_{(X,D)})\to \Omega_S^{p+1}\otimes_{\bQ}R^qf_\ast(\bV^{\rm BM}_{(X,D)}).$$
    In the special case when $p=0$, the morphism $E^{0,q} \to E^{1,q}$ is given by
    $$\nabla:=d\otimes{\rm Id}:\sO_S\otimes_{\bQ}R^qf_\ast(\bV^{\rm BM}_{(X,D)})\to \Omega_S\otimes_{\bQ}R^qf_\ast(\bV^{\rm BM}_{(X,D)}),$$
    which is the flat bundle $(\cV^k_{f,{\rm BM}}, \nabla)$ associated with $\bV^q_{f,{\rm BM}}=R^qf_\ast(\bV^{\rm BM}_{(X,D)})$ via the Riemann-Hilbert correspondence.
    \subsubsection{Flat connection}\label{section_flat_connection}
    We introduce the complex
    $$\Omega^\bullet_{(X,D),{\rm BM}}:={\rm Tot}\left(\cdots\to i_{[2]\ast}\Omega^\bullet_{X_{[2]}}(\log D_{[2]})[-2]\to i_{[1]\ast}\Omega^\bullet_{X_{[1]}}(\log D_{[1]})\right),$$
    and define the Koszul filtration $K^\bullet$ on it by 
    $$K^p={\rm Tot}\left(\cdots\to f^\ast\Omega^p_S\wedge i_{[2]\ast}\Omega^{\bullet-p}_{X_{[2]}}(\log D_{[2]})[-2]\to f^\ast\Omega^p_S\wedge i_{[1]\ast}\Omega^{\bullet-p}_{X_{[1]}}(\log D_{[1]})\right).$$
    We define the complex of $f^{-1}\sO_S$-modules
    $$\Omega^\bullet_{f,{\rm BM}}={\rm Tot}\left(\cdots\to i_{[2]\ast}\Omega^\bullet_{X_{[2]}/S}(\log D_{[2]})(-1)[-2]\to i_{[1]\ast}\Omega^\bullet_{X_{[1]}/S}(\log D_{[1]})\right).$$
    Then, we have
    \begin{align}\label{align_Gr_K}
    	K^p/K^{p+1}\simeq   f^{-1}\Omega_S^p\otimes_{f^{-1}\sO_S}\Omega^\bullet_{f,{\rm BM}}[-p].
    \end{align}
    Applying the functor $Rf_\ast$ to this filtered complex yields the spectral sequence
    \begin{align*}
    	E^{p,q}_1:=\Omega_S^p\otimes_{\sO_S}R^qf_\ast(\Omega^\bullet_{f,{\rm BM}})\Rightarrow R^{p+q}f_\ast(\Omega^\bullet_{(X,D),{\rm BM}}).
    \end{align*}
    The morphisms of the $E_1$-terms give rise to the Gauss-Manin connections
    $$\nabla_{\rm GM}:\Omega_S^p\otimes_{\sO_S}R^qf_\ast(\Omega^\bullet_{f,{\rm BM}})\to \Omega_S^{p+1}\otimes_{\sO_S}R^qf_\ast(\Omega^\bullet_{f,{\rm BM}}),$$
    which are defined as the connecting morphism $\delta$ in the long exact sequences
    $$\cdots\to R^{p+q}f_\ast(K^p/K^{p+2})\to\Omega_S^p\otimes_{\sO_S}R^qf_\ast(\Omega^\bullet_{f,{\rm BM}})\stackrel{\delta}{\to} \Omega_S^{p+1}\otimes_{\sO_S}R^qf_\ast(\Omega^\bullet_{f,{\rm BM}})\to\cdots$$
    associated with the distinguished triangle
    \begin{align}\label{align_nablaBM_def}
    \small\xymatrix{
    	K^{p+1}/K^{p+2}\ar[d]^{\simeq} \ar[r]&K^p/K^{p+2}\ar[r] &K^p/K^{p+1}\ar[d]^{\simeq}\ar[r]&\\
    	f^{-1}\Omega_S^{p+1}\otimes_{f^{-1}\sO_S}\Omega^\bullet_{f,{\rm BM}}[-p-1] & &f^{-1}\Omega_S^{p}\otimes_{f^{-1}\sO_S}\Omega^\bullet_{f,{\rm BM}}[-p] &
    }.
    \end{align}
    In the special case when $p=0$, the morphism $E^{0,q}\to E^{1,q}$ gives us the flat bundle
    $$\nabla_{\rm GM}:R^qf_\ast(\Omega^\bullet_{f,{\rm BM}})\to \Omega_S\otimes_{\sO_S} R^qf_\ast(\Omega^\bullet_{f,{\rm BM}}).$$
    \begin{lem}\label{lem_BM_rel_Poincare}
    	There exists a canonical filtered quasi-isomorphism
    	$$\left(f^{-1}\Omega_S^\bullet\otimes_{\bQ}\bV^{\rm BM}_{(X,D)},L^\bullet\right)\simeq_{\rm qis}\left(\Omega^\bullet_{(X,D),{\rm BM}},K^\bullet\right).$$
    	In particular, there is an isomorphism between the flat connections
    	$$(\sO_S\otimes_{\bQ} \bV^k_{f,{\rm BM}},d\otimes{\rm Id})\simeq (R^kf_\ast(\Omega^\bullet_{f,{\rm BM}}),\nabla_{\rm GM})$$
    	for every $k$.
    \end{lem}
    \begin{proof}
    	By (\ref{align_BM_complex_resolve}) one may replace $\bV^{\rm BM}_{(X,D)}$ with $${\rm Tot}\left(\cdots\to i_{[2]\ast}Rj_\ast\bQ_{X^o_{[2]}}(-1)[-2]\to i_{[1]\ast}Rj_\ast\bQ_{X^o_{[1]}}\right).$$
    	Consequently, there exists a morphism
    	$$\bV^{\rm BM}_{(X,D)}\to \Omega^\bullet_{(X,D),{\rm BM}}$$
    	induced by canonical morphisms
    	$Rj_\ast\bQ_{X^o_{[i]}}\to \Omega^\bullet_{X_{[i]}}(\log D_{[i]})$ ($i\geq0$). This extends to a filtered morphism
    	$$\varphi:f^{-1}\Omega_S^\bullet\otimes_{\bQ}\bV^{\rm BM}_{(X,D)}\to\Omega^\bullet_{(X,D),{\rm BM}}$$
    	such that $\varphi(L^k) \subset K^k$ for all $k$.
    	Since $(X_{[i]}, D_{[i]}) \to S$ is a log smooth family, there exists a quasi-isomorphism
    	\begin{align*}
    		f^{-1}\sO_S \otimes_{\bQ} Rj_\ast\bQ_{X^o_{[i]}} \simeq \Omega^\bullet_{X_{[i]}/S}(\log D_{[i]})
    	\end{align*}
    	for each $i$. According to (\ref{align_GrL}) and (\ref{align_Gr_K}), this demonstrates that $\varphi$ is a filtered isomorphism. This establishes the first claim of the lemma.
    	
    	Furthermore, since the left (resp. right) flat connection arises from the morphism $E^{0,k}_1 \to E^{1,k}_1$ of the spectral sequence associated with the left (resp. right) filtered complex, the second claim follows.
    \end{proof}  
    \subsubsection{Hodge and Weight filtrations}\label{section_Hodge_weight_BM}
    According to Lemma \ref{lem_BM_rel_Poincare}, there exists a canonical quasi-isomorphism
    \begin{align}\label{align_relPoincare_BM}
    	f^{-1}\sO_S \otimes_{\bQ} \bV^{\rm BM}_{(X,D)} \simeq_{\rm qis} \Omega^\bullet_{f,{\rm BM}}
    \end{align}
    induced by the quasi-isomorphisms
    \begin{align}\label{align_relPoincare_logsmooth}
    	(fi_{[k]})^{-1}(\sO_S) \otimes_{\bQ} Rj_\ast\bQ_{X^o_{[k]}} \simeq_{\rm qis} \Omega^\bullet_{X_{[k]}/S}(\log D_{[k]})
    \end{align}
    for all $k \geq 0$.
    
    For every $k \geq 0$ and for all $m, p \in \bZ$, we define
    \begin{align*}
    	W_m\Omega^p_{X_{[k]}/S}(\log D_{[k]}) = 
    	\begin{cases}
    		0, & m < 0 \\
    		\Omega^p_{X_{[k]}/S}(\log D_{[k]}), & m \geq p \\
    		\Omega^{p-m}_{X_{[k]}/S} \wedge \Omega^m_{X_{[k]}/S}(\log D_{[k]}), & 0 \leq m \leq p
    	\end{cases}
    \end{align*}
    and
    \[
    F^p\Omega^\bullet_{X_{[k]}/S}(\log D_{[k]}) := \Omega^{\geq p}_{X_{[k]}/S}(\log D_{[k]}).
    \]
    We define the weight filtration $W_{\bullet}$ of $\Omega_{f,{\rm BM}}^\bullet$ by
    \begin{align*}
    W_m\Omega_{f,{\rm BM}}^\bullet&:={\rm Tot}\left(\cdots\to i_{[2]\ast}W_{m-1}(\Omega^\bullet_{X_{[2]}/S}(\log D_{[2]})(-1)[-2])\to i_{[1]\ast}W_m\Omega^\bullet_{X_{[1]}/S}(\log D_{[1]})\right)\\\nonumber
    &={\rm Tot}\left(\cdots\to i_{[2]\ast}W_{m-1}\Omega^\bullet_{X_{[2]}/S}(\log D_{[2]})[-2]\to i_{[1]\ast}W_m\Omega^\bullet_{X_{[1]}/S}(\log D_{[1]})\right)
    \end{align*}
    and the Hodge filtration $F^\bullet$ by
    \begin{align*}
    F^p\Omega_{f,{\rm BM}}^\bullet&:=\sigma^{\geq p}\Omega_{f,{\rm BM}}^\bullet\\\nonumber
    &={\rm Tot}\left(\cdots\to i_{[2]\ast}F^p(\Omega^{\bullet}_{X_{[2]}/S}(\log D_{[2]})(-1)[-2])\to i_{[1]\ast}F^p(\Omega^{\bullet}_{X_{[1]}/S}(\log D_{[1]}))\right)\\\nonumber
    &={\rm Tot}\left(\cdots\to i_{[2]\ast}\Omega^{\geq p-1}_{X_{[2]}/S}(\log D_{[2]})[-2]\to i_{[1]\ast}\Omega^{\geq p}_{X_{[1]}/S}(\log D_{[1]})\right).
    \end{align*}
The weight filtration $\tau_\bullet$ defined in (\ref{align_BM_weight_filtation}) on $\bV^{\rm BM}_{(X,D)}$ induces the corresponding weight filtration $f^{-1}\sO_S \otimes_{\bQ} \tau_\bullet$ on $f^{-1}\sO_S \otimes_{\bQ} \bV^{\rm BM}_{(X,D)}$.
    \begin{lem}\label{lem_filtered_rePoincare_BM}
    	(\ref{align_relPoincare_BM}) provides a filtered quasi-isomorphism
    	$$\left(f^{-1}\sO_S\otimes_{\bQ}\bV^{\rm BM}_{(X,D)},f^{-1}\sO_S\otimes_{\bQ}\tau_\bullet\right)\simeq_{\rm qis}\left(\Omega^\bullet_{f,{\rm BM}},W_\bullet\right).$$
    \end{lem}
\begin{proof}
	The relative version of Lemma \ref{lem_mHC_logsmooth} demonstrates that (\ref{align_relPoincare_logsmooth}) is a filtered quasi-isomorphism that preserves the weight filtration $(fi_{[k]})^{-1}(\sO_S) \otimes_{\bQ} \tau_{[k],\bullet}$ on the left-hand side and the weight filtration $W_\bullet\Omega^\bullet_{X_{[k]}/S}(\log D_{[k]})$ on the right-hand side. Consequently, this leads to the lemma.
\end{proof}
    As the relative version of (\ref{align_BM_GrW}), one observes that
    \begin{align}\label{align_BM_GrW_rel}
    	{\rm Gr}^W_m\Omega^\bullet_{f,{\rm BM}} &\simeq_{\rm qis} \bigoplus_{i \geq 0} {\rm Gr}^W_{m-i}(\Omega^{\bullet}_{X_{[i+1]}/S}(\log D_{[i+1]})(-i)[-2i])[i] \\\nonumber
    	&\simeq_{\rm qis} \bigoplus_{i \geq 0} {\rm Gr}^W_{m-i}(\Omega^{\bullet}_{X_{[i+1]}/S}(\log D_{[i+1]}))(-i)[-i] \\\nonumber
    	&\simeq_{\rm qis} \bigoplus_{i \geq 0} \Omega^\bullet_{D_{[i+1],[m-i]}/S}(-m)[-m],
    \end{align}
    with the induced Hodge filtration
    \begin{align}\label{align_BM_GrW_F_rel}
    	F^p({\rm Gr}^W_m\Omega^\bullet_{f,{\rm BM}}) &\simeq_{\rm qis} \bigoplus_{i \geq 0} F^p{\rm Gr}^W_{m-i}(\Omega^{\bullet}_{X_{[i+1]}/S}(\log D_{[i+1]})(-i)[-2i])[i] \\\nonumber
    	&\simeq_{\rm qis} \bigoplus_{i \geq 0} F^p(\Omega^\bullet_{D_{[i+1],[m-i]}/S}(-m)[-m]) \\\nonumber
    	&\simeq_{\rm qis} \bigoplus_{i \geq 0} \Omega^{\geq p-m}_{D_{[i+1],[m-i]}/S}[-m].
    \end{align}
    Consequently,
    \begin{align*}
    	{\rm Gr}^p_F{\rm Gr}^W_m\Omega^\bullet_{f,{\rm BM}} \simeq_{\rm qis} \bigoplus_{i \geq 0} \Omega^{p-m}_{D_{[i+1],[m-i]}/S}[-p].
    \end{align*}
    According to (\ref{align_BM_GrW_rel}), the weight filtration induces the spectral sequence
    \begin{align}\label{align_rel_ss_BM}
    E^{-m,k+m}_1(Rf_\ast(\Omega^\bullet_{f,{\rm BM}}),W)=R^kf_\ast({\rm Gr}_m^{W}\Omega^\bullet_{f,{\rm BM}})=\bigoplus_{i\geq0} R^{k-m}f_\ast(\Omega^\bullet_{D_{[i+1],[m-i]}/S})(-m)\Rightarrow R^{k}f_\ast(\Omega^\bullet_{f,{\rm BM}}).
    \end{align}
    Since each $D_{[i+1],[m-i]}$ is smooth and projective over $S$, it follows that each $R^{k-m} f_\ast(\Omega^\bullet_{D_{[i+1],[m-i]}/S})(-m)$ underlies a polarizable variation of Hodge structure of weight $k+m$, with the Hodge filtration given by $R^kf_\ast(F^\bullet(\operatorname{Gr}^W_m \Omega^\bullet_{f, \text{BM}}))$. In particular, the morphism $E^{-m,k+m}_1 \to E^{-m+1,k+m}_1$ on the $E_1$-page is strictly compatible with the Hodge filtrations. Consequently, the limit of this spectral sequence induces a weight filtration on $R^k f_\ast(\mathbb{V}^{\text{BM}}_{(X,D)})$ defined by
    \begin{align*}
    W_mR^kf_\ast(\bV^{\rm BM}_{(X,D)}):={\rm Im}\left(R^kf_\ast(\tau_{m-k}(\bV^{\rm BM}_{(X,D)}))\to R^kf_\ast(\bV^{\rm BM}_{(X,D)})\right),
    \end{align*}
    and a weight filtration on $R^kf_\ast(\Omega^\bullet_{f,{\rm BM}})$ defined by
    \begin{align*}
    W_mR^kf_\ast(\Omega^\bullet_{f,{\rm BM}}):={\rm Im}\left(R^kf_\ast(W_{m-k}\Omega^\bullet_{f,{\rm BM}})\to R^kf_\ast(\Omega^\bullet_{f,{\rm BM}})\right).
    \end{align*}
    As a consequence of Lemma (\ref{lem_filtered_rePoincare_BM}), we have the following result:   
    \begin{lem}
    	For every integer $k$, there exists a filtered quasi-isomorphism
    	\[ \sO_S \otimes_{\mathbb{Q}} (R^k f_\ast(\mathbb{V}^{{\rm BM}}_{(X,D)}), W_\bullet R^k f_\ast(\mathbb{V}^{{\rm BM}}_{(X,D)})) \simeq_{\rm qis} (R^k f_\ast(\Omega^\bullet_{f, {\rm BM}}), W_\bullet R^k f_\ast(\Omega^\bullet_{f, {\rm BM}})). \]
    \end{lem}
    By the lemma on two filtrations (\cite{Delign1974}), we obtain the following theorem.
    \begin{thm}\label{thm_BM_preVMHS}
    	\begin{enumerate}
    		\item The spectral sequence (\ref{align_rel_ss_BM}) degenerates at the $E_2$-page. 
    		\item The three filtrations $F_{\text{dir}}$, $F_{\text{dir}}^\ast$, and $F_{\text{ind}}$ coincide on each page of the spectral sequence (\ref{align_rel_ss_BM}).
    		\item The spectral sequence 
    		$${_F}E^{p,q}_1=R^{p+q}f_\ast({\rm Gr}_F^p\Omega^\bullet_{f,{\rm BM}})\Rightarrow R^{p+q}f_\ast(\Omega^\bullet_{f,{\rm BM}})$$
    		degenerates at $E_1$. In particular, the natual maps
    		$$R^kf_\ast(F^p\Omega^\bullet_{f,{\rm BM}})\to R^kf_\ast(\Omega^\bullet_{f,{\rm BM}})$$
    		are injective, and their images, denoted by $F^p R^k f_\ast(\Omega^\bullet_{f,{\rm BM}})$, coincide with $F_{\text{dir}} = F_{\text{dir}}^\ast = F_{\text{ind}}$. Moreover,
    		$${\rm Gr}_F^pR^kf_\ast(F^p\Omega^\bullet_{f,{\rm BM}})\simeq R^kf_\ast({\rm Gr}_F^p\Omega^\bullet_{f,{\rm BM}}).$$
    		\item For every integer $k$, the data $$\left((R^kf_\ast(\bV^{\rm BM}_{(X,D)}),W_\bullet R^kf_\ast(\bV^{\rm BM}_{(X,D)})),(R^kf_\ast(\Omega^\bullet_{f,{\rm BM}}),W_\bullet R^kf_\ast(\Omega^\bullet_{f,{\rm BM}}),F^\bullet R^kf_\ast(\Omega^\bullet_{f,{\rm BM}}))\right)$$ forms a graded polarizable pre-variation of mixed Hodge structures.
    		\item The spectral sequence for the filtered complex
    		$({\rm Gr}_F^pR^kf_\ast(\Omega^\bullet_{f,{\rm BM}}),W)$ degenerates at $E_2$.
    	\end{enumerate}
    \end{thm}
    \begin{proof}
	    The proof follows the same argument as in \cite[Theorem 3.18]{Steenbrink2008}.
    \end{proof}
    \begin{thm}\label{thm_VMHS_BM}
    	For every integer $k$, the data $$\left((R^kf_\ast(\bV^{\rm BM}_{(X,D)}),W_\bullet R^kf_\ast(\bV^{\rm BM}_{(X,D)})),(R^kf_\ast(\Omega^\bullet_{f,{\rm BM}}),W_\bullet R^kf_\ast(\Omega^\bullet_{f,{\rm BM}}),F^\bullet R^kf_\ast(\Omega^\bullet_{f,{\rm BM}}))\right)$$ form a graded polarizable variation of mixed Hodge structures.
    \end{thm}
\begin{rmk}\label{rmk_BM_VMHS}
	By a slight abuse of notation, the variation of mixed Hodge structures described in Theorem \ref{thm_VMHS_BM} is denoted by $\bV^k_{f,{\rm BM}}$. This is referred to as the \emph{Borel-Moore variation of mixed Hodge structures associated with the morphism $f:(X,D)\to S$}.
\end{rmk}
    \begin{proof}
    	According to Theorem \ref{thm_BM_preVMHS}-(4), it suffices to show the Griffiths transversality
    	$$\nabla(F^pR^kf_\ast(\Omega^\bullet_{f,{\rm BM}}))\subset F^{p-1}R^kf_\ast(\Omega^\bullet_{f,{\rm BM}})\otimes_{\sO_S}\Omega_S,\quad\forall p.$$
    	By Lemma \ref{lem_BM_rel_Poincare}, it is equivalent to show
    	\begin{align}\label{align_BM_GT}
    	\nabla_{\rm GM}(F^pR^kf_\ast(\Omega^\bullet_{f,{\rm BM}}))\subset F^{p-1}R^kf_\ast(\Omega^\bullet_{f,{\rm BM}})\otimes_{\sO_S}\Omega_S,\quad\forall p.
    	\end{align}
    	Recall that $\nabla_{\text{GM}}$ is the connecting map of the long exact sequence associated with (\ref{align_nablaBM_def}). By considering the long exact sequences arising from the commutative diagram of distinguished triangles
    	$$
    	\xymatrix{
    		f^{-1}\Omega_S\otimes_{f^{-1}\sO_S}F^{p-1}\Omega^\bullet_{f,{\rm BM}}[-1]\ar[d] \ar[r]&\sigma^{\geq p}(K^0/K^2)\ar[r]\ar[d] &F^p\Omega^\bullet_{f,{\rm BM}}\ar[d]\ar[r]&\\
    		f^{-1}\Omega_S\otimes_{f^{-1}\sO_S}\Omega^\bullet_{f,{\rm BM}}[-1]\ar[r] &K^0/K^{2}\ar[r] &\Omega^\bullet_{f,{\rm BM}}\ar[r] & 
    	}
    	$$
    	we obtain a commutative diagram:
    	$$
    	\xymatrix{
    		R^kf_\ast(F^p\Omega^\bullet_{f,{\rm BM}})\ar[r]\ar[d] & \Omega_S\otimes_{\sO_S}R^kf_\ast(F^{p-1}\Omega^\bullet_{f,{\rm BM}})\ar[d]\\
    		R^kf_\ast(\Omega^\bullet_{f,{\rm BM}})\ar[r]^-{\nabla_{\rm GM}} & \Omega_S\otimes_{\sO_S}R^kf_\ast(\Omega^\bullet_{f,{\rm BM}})\\
    	}.
    	$$
    Combining this with Theorem \ref{thm_BM_preVMHS} (3), we conclude that (\ref{align_BM_GT}) holds.
    \end{proof}
\section{Logarithmic extension of $\bV^k_{f,{\rm BM}}$}\label{section_log_extension_VMHS_BM}
In this section, we consider the logarithmic version of the Borel-Moore variation of mixed Hodge structures $\mathbb{V}^k_{f,{\rm BM}}$ for a semistable family of simple normal crossing pairs. This version plays a crucial role in establishing the admissibility of $\mathbb{V}^k_{f, \text{BM}}$ and in the geometric construction of Deligne's lower canonical extension of $\mathbb{V}^k_{f, \text{BM}}$ (see Theorem \ref{thm_VMHS_BM_admissible} below).

Let $\overline{f}: (\overline{X}, \overline{D}) \to (\overline{S}, D_{\overline{S}})$ be a projective semistable morphism (Definition \ref{defn_semistable_cod1}) from a simple normal crossing pair $(\overline{X}, \overline{D})$ to a log smooth pair $(\overline{S}, D_{\overline{S}})$. Suppose $\dim S = d$ and $\dim X = n + d$. Define $S := \overline{S} \setminus D_{\overline{S}}$ and $(X, D) := (\overline{X}, \overline{D}) \cap f^{-1}(S)$. Consequently, the restriction $f := \overline{f}|_X: (X, D) \to S$ is a simple normal crossing projective family. For every integer $k$, this yields the Borel-Moore variation of mixed Hodge structures $\mathbb{V}^k_{f, \text{BM}}$ (Remark \ref{rmk_BM_VMHS}).
\subsection{Local monodromy}
\begin{lem}\label{lem_quasi-unipotency}
	For each $i \geq 1$ and $q \geq 0$, the cohomology sheaf $h^q(Rf_\ast Rj_\ast \mathbb{Z}_{X^o_{[i]}})$ has quasi-unipotent monodromy around every component of $D_{\overline{S}}$. The underlying local system of $\mathbb{V}^k_{f, {\rm BM}}$ admits a $\mathbb{Z}$-structure, and the local monodromy of $\mathbb{V}^k_{f, {\rm BM}}$ around every component of $D_{\overline{S}}$ is quasi-unipotent.
\end{lem}
\begin{proof}
According to (\ref{align_BM_local_system}), the local system $\mathbb{V}^k_{f, \text{BM}}$ has the integral structure
\[
\mathbb{V}^k_{f, \text{BM}, \mathbb{Z}} := h^k \operatorname{Tot}\left(\cdots \to Rf_\ast Rj_\ast \mathbb{Z}_{X^o_{[2]}}(-1)[-2] \to Rf_\ast Rj_\ast \mathbb{Z}_{X^o_{[1]}}\right).
\]
The quasi-unipotency of the local monodromies follows from the existence of this $\mathbb{Z}$-structure (cf. Borel's theorem \cite[Lemma 4.5]{Schmid1973}).
\end{proof}
\subsection{Logarithmic extension of the flat connection}\label{section_log_GMconnection_VBM}
We adopt the notations from \S \ref{section_BM_coh}. Specifically, we set $X^o := \overline{X} \setminus \overline{D} = X \setminus D$ and let $j: X^o \to \overline{X}$ denote the open immersion. Let $\overline{X} = \bigcup_{i \in I} \overline{X}_i$ and $\overline{D} = \bigcup_{\lambda \in \Lambda} \overline{D}_\lambda$ be the irreducible decompositions, where $I$ and $\Lambda$ are finite index sets. For every non-empty subset $J \subset I$ and integer $k > 0$, we define
\[
\overline{X}_J := \bigcap_{i \in J} \overline{X}_i \quad \text{and} \quad \overline{X}_{[k]} := \coprod_{\substack{J \subset I,\\ |J| = k}} \overline{X}_J.
\]
Let $i_{[k]}: \overline{X}_{[k]} \to \overline{X}$ be the natural map. Define $X^o_J := \overline{X}_J \setminus \overline{D}$ for $\emptyset \neq J \subset I$, and set 
\[
X^o_{[k]} := \coprod_{\substack{J \subset I,\\ |J| = k}} X^o_J,
\]
with the natural morphism denoted by $i^o_{[k]}: X^o_{[k]} \to X^o$. Let $\overline{D}_{[k]} := \overline{X}_{[k]} \setminus X^o$. Then, for each $k$, $(\overline{X}_{[k]}, \overline{D}_{[k]})$ is a log smooth pair, and the morphism $(\overline{X}_{[k]}, \overline{D}_{[k]}) \to (\overline{S}, D_{\overline{S}})$ is semistable with log smooth generic fibers. We denote by $$\Omega_{\overline{X}_{[k]}/\overline{S}}(\log \overline{D}_{[k]}):=\Omega_{\overline{X}_{[k]}}(\log \overline{D}_{[k]})/f^\ast\Omega_{\overline{S}}(\log D_{\overline{S}})$$ the relative logarithmic cotangent sheaf and by $\Omega^\bullet_{\overline{X}_{[k]}/\overline{S}}(\log \overline{D}_{[k]})$ the associated relative logarithmic de Rham complex.

As in \S \ref{section_flat_connection},
we introduce the Koszul filtration $K^\bullet_{\overline{f}}$ on 
$$\Omega^\bullet_{(\overline{X},\overline{D}),{\rm BM}}:={\rm Tot}\left(\cdots\to i_{[2]\ast}\Omega^\bullet_{\overline{X}_{[2]}}(\log \overline{D}_{[2]})[-2]\to i_{[1]\ast}\Omega^\bullet_{\overline{X}_{[1]}}(\log \overline{D}_{[1]})\right),$$
defined by 
$$K^p_{\overline{f}}={\rm Tot}\left(\cdots\to \overline{f}^\ast\Omega^p_{\overline{S}}(\log D_{\overline{S}})\wedge i_{[2]\ast}\Omega^{\bullet-p}_{\overline{X}_{[2]}}(\log \overline{D}_{[2]})[-2]\to \overline{f}^\ast\Omega^p_{\overline{S}}(\log D_{\overline{S}})\wedge i_{[1]\ast}\Omega^{\bullet-p}_{\overline{X}_{[1]}}(\log \overline{D}_{[1]})\right).$$
We define the complex of $\overline{f}^{-1}\sO_{\overline{S}}$-modules
$$\Omega^\bullet_{\overline{f},{\rm BM}}={\rm Tot}\left(\cdots\to i_{[2]\ast}\Omega^\bullet_{\overline{X}_{[2]}/\overline{S}}(\log \overline{D}_{[2]})[-2]\to i_{[i]\ast}\Omega^\bullet_{\overline{X}_{[1]}/\overline{S}}(\log \overline{D}_{[1]})\right).$$
Then we have
\begin{align*}
	K^p_{\overline{f}}/K^{p+1}_{\overline{f}}\simeq \overline{f}^{-1}\Omega_{\overline{S}}^p(\log D_{\overline{S}})\otimes_{\overline{f}^{-1}\sO_{\overline{S}}}\Omega^\bullet_{\overline{f},{\rm BM}}[-p].
\end{align*}
Applying the functor $R\overline{f}_\ast$ to this filtered complex yields the spectral sequence
\begin{align*}
	E^{p,q}_1:=\Omega_{\overline{S}}^p(\log D_{\overline{S}})\otimes_{\sO_{\overline{S}}}R^q\overline{f}_\ast(\Omega^\bullet_{\overline{f},{\rm BM}})\Rightarrow R^{p+q}\overline{f}_\ast(\Omega^\bullet_{(\overline{X},\overline{D}),{\rm BM}}).
\end{align*}
The morphisms of the $E_1$-terms give rise to the logarithmic Gauss-Manin connections
$$\nabla_{\rm GM}:\Omega_{\overline{S}}^p(\log D_{\overline{S}})\otimes_{\sO_{\overline{S}}}R^q\overline{f}_\ast(\Omega^\bullet_{\overline{f},{\rm BM}})\to \Omega_{\overline{S}}^{p+1}(\log D_{\overline{S}})\otimes_{\sO_{\overline{S}}}R^q\overline{f}_\ast(\Omega^\bullet_{\overline{f},{\rm BM}}).$$
This connection is defined as the connecting morphism $\delta$ in the long exact sequence
$$\cdots\to R^{p+q}\overline{f}_\ast(K^p_{\overline{f}}/K^{p+2}_{\overline{f}})\to\Omega_{\overline{S}}^p(\log D_{\overline{S}})\otimes_{\sO_{\overline{S}}}R^q\overline{f}_\ast(\Omega^\bullet_{\overline{f},{\rm BM}})\stackrel{\delta}{\to} \Omega_{\overline{S}}^{p+1}(\log D_{\overline{S}})\otimes_{\sO_{\overline{S}}}R^qf_\ast(\Omega^\bullet_{\overline{f},{\rm BM}})\to\cdots$$
which is associated with the distinguished triangle
\begin{align*}
	{\small\xymatrix{
		 K^{p+1}_{\overline{f}}/K^{p+2}_{\overline{f}}\ar[d]^{\simeq} \ar[r]&K^p_{\overline{f}}/K^{p+2}_{\overline{f}}\ar[r] &K^p_{\overline{f}}/K^{p+1}_{\overline{f}}\ar[d]^{\simeq}\ar[r]&\\
		\overline{f}^{-1}\Omega_{\overline{S}}^{p+1}(\log D_{\overline{S}})\otimes_{\overline{f}^{-1}\sO_{\overline{S}}}\Omega^\bullet_{\overline{f},{\rm BM}}[-p-1] & &\overline{f}^{-1}\Omega_{\overline{S}}^{p}(\log D_{\overline{S}})\otimes_{\overline{f}^{-1}\sO_{\overline{S}}}\Omega^\bullet_{\overline{f},{\rm BM}}[-p] &
	}}.
\end{align*}
In the special case when $p=0$, the morphism $E^{0,q}\to E^{1,q}$ yields the logarithmic flat connection
$$\nabla_{\rm GM}:R^q\overline{f}_\ast(\Omega^\bullet_{\overline{f},{\rm BM}})\to \Omega_{\overline{S}}(\log D_{\overline{S}})\otimes_{\sO_{\overline{S}}} R^q\overline{f}_\ast(\Omega^\bullet_{\overline{f},{\rm BM}}).$$
We will demonstrate in Proposition \ref{prop_weightss_E2_BM} that $R^q \overline{f}_\ast (\Omega^\bullet_{\overline{f}, \text{BM}})$ is locally free for every $q$.
\subsection{Logarithmic Hodge and Weight filtrations}
As in \S \ref{section_Hodge_weight_BM}, we formally define the logarithmic weight filtration and Hodge filtration on $\Omega^\bullet_{\overline{f},{\rm BM}}$. Let $E := \overline{f}^{-1}(D_{\overline{S}})$ denote the preimage of $D_{\overline{S}}$ under the map $\overline{f}$. Then $E_{[k]} := i_{[k]}^{-1}(E)$ is a simple normal crossing divisor on $\overline{X}_{[k]}$. For every non-negative integer $k$ and integers $m, p \in \mathbb{Z}$, we establish the following definitions:
	\begin{align*}
		W_m\Omega^p_{\overline{X}_{[k]}/\overline{S}}(\log \overline{D}_{[k]})=\begin{cases}
			0, &m<0\\
			\Omega^p_{\overline{X}_{[k]}/\overline{S}}(\log E_{[k]}), &m\geq p\\
			\Omega_{\overline{X}_{[k]}/\overline{S}}^{p-m}(\log E_{[k]})\wedge\Omega^m_{\overline{X}_{[k]}/\overline{S}}(\log \overline{D}_{[k]}), &0\leq m\leq p
		\end{cases}
	\end{align*}
	and set $$F^p\Omega^\bullet_{\overline{X}_{[k]}/\overline{S}}(\log \overline{D}_{[k]}):=\Omega^{\geq p}_{\overline{X}_{[k]}/\overline{S}}(\log \overline{D}_{[k]}).$$
	We define the weight filtration $W_{\bullet}$ of $\Omega_{\overline{f},{\rm BM}}^\bullet$ as follows:
	\begin{align}\label{align_weight_filtration_BM}
		W_m\Omega_{\overline{f},{\rm BM}}^\bullet&:={\rm Tot}\left(\cdots\to W_{m-1}\Omega^\bullet_{\overline{X}_{[2]}/\overline{S}}(\log \overline{D}_{[2]})[-2]\to W_m\Omega^\bullet_{\overline{X}_{[1]}/\overline{S}}(\log \overline{D}_{[1]})\right).
	\end{align}
	Additionally, we define its logarithmic Hodge filtration $F^\bullet$ by
	\begin{align}\label{align_log_Hodge_fil}
		F^p\Omega_{\overline{f},{\rm BM}}^\bullet&:=\sigma^{\geq p}\Omega_{\overline{f},{\rm BM}}^\bullet={\rm Tot}\left(\cdots\to \Omega^{\geq p-1}_{\overline{X}_{[2]}/\overline{S}}(\log \overline{D}_{[2]})[-2]\to \Omega^{\geq p}_{\overline{X}_{[1]}/\overline{S}}(\log \overline{D}_{[1]})\right).
	\end{align}	
	In particular, the top Hodge complex is given by
	\begin{align}\label{align_top_Hodge_complex}
	F^{n}\Omega^\bullet_{\overline{f},{\rm BM}}&\simeq{\rm Tot}\left(\cdots\to \omega_{\overline{X}_{[2]}/\overline{S}}(\log \overline{D}_{[2]})[-1-n]\to \omega_{\overline{X}_{[1]}/\overline{S}}(\log \overline{D}_{[1]})[-n]\right)\\\nonumber
	&\simeq F^{n+d}\Omega^\bullet_{(\overline{X},\overline{D}),{\rm BM}}[d]\otimes_{\sO_{\overline{X}}}(\overline{f}^\ast\omega_{\overline{S}}(D_{\overline{S}}))^{-1}\quad (\ref{align_top_Hodge_ab})\\\nonumber
	&\simeq\sO_{\overline{X}}(K_{\overline{X}}+\overline{D})[-n]\otimes_{\sO_{\overline{X}}}(\overline{f}^\ast\omega_{\overline{S}}(D_{\overline{S}}))^{-1}\quad\textrm{(Lemma \ref{lem_top_Hodge_BM})}\\\nonumber
	&\simeq \sO_{\overline{X}}(K_{\overline{X}/\overline{S}}+\overline{D}-\overline{f}^\ast(D_{\overline{S}}))[-n].
	\end{align}
    These subcomplexes induce the (logarithmic) Hodge filtration
    \begin{align*}
    	F^pR^k\overline{f}_\ast(\Omega_{\overline{f},{\rm BM}}^\bullet):={\rm Im}\left(R^k\overline{f}_\ast(F^p\Omega_{\overline{f},{\rm BM}}^\bullet)\to R^k\overline{f}_\ast(\Omega_{\overline{f},{\rm BM}}^\bullet)\right),\quad p\geq 0
    \end{align*}
and the weight filtration
    \begin{align*}
    	W_mR^k\overline{f}_\ast(\Omega_{\overline{f},{\rm BM}}^\bullet):={\rm Im}\left(R^k\overline{f}_\ast(W_{m-k}\Omega_{\overline{f},{\rm BM}}^\bullet)\to R^k\overline{f}_\ast(\Omega_{\overline{f},{\rm BM}}^\bullet)\right),\quad m\in\bZ.
    \end{align*}
	Let $\overline{D}^h_{[k]}$ denote the horizontal part of $\overline{D}_{[k]}$, and let $\overline{D}^h_{[k]} = \bigcup_{i \in I_k} \overline{D}_{[k],i}$ be its irreducible decomposition. For every non-empty subset $\emptyset \neq J \subset I_k$ and positive integer $l > 0$, we define
	$$
	\overline{D}_{[k],J} := \bigcap_{i \in J} \overline{D}_{[k],i} \quad \text{and} \quad \overline{D}_{[k],[l]} := \coprod_{\substack{J \subset I_k, \\ |J| = l}} \overline{D}_{[k],J}.
	$$
	Additionally, we set $\overline{D}_{[k],[0]} := \overline{X}_{[k]}$. Let $E_{[k],[l]} := \overline{D}_{[k],[l]} \cap E$. Then, each morphism $(\overline{D}_{[k],[l]}, E_{[k],[l]}) \to (\overline{S}, D_{\overline{S}})$ forms a semistable family with smooth general fibers.
	Similar to (\ref{align_BM_GrW_rel}) and (\ref{align_BM_GrW_F_rel}), we observe that
	\begin{align}\label{align_BM_GrW_rel_ss}
		{\rm Gr}^W_m\Omega^\bullet_{\overline{f},{\rm BM}}&\simeq_{\rm qis}\bigoplus_{i\geq0}{\rm Gr}^W_{m-i}(\Omega^{\bullet}_{\overline{X}_{[i+1]}/\overline{S}}(\log \overline{D}_{[i+1]})(-i)[-2i])[i]\\\nonumber
		&\simeq_{\rm qis}\bigoplus_{i\geq0}{\rm Gr}^W_{m-i}(\Omega^{\bullet}_{\overline{X}_{[i+1]}/\overline{S}}(\log \overline{D}_{[i+1]}))(-i)[-i]\\\nonumber
		&\simeq_{\rm qis}\bigoplus_{i\geq0}\Omega^\bullet_{\overline{D}_{[i+1],[m-i]}/\overline{S}}(\log E_{[i+1],[m-i]})(-m)[-m],
	\end{align}
	with the induced Hodge filtration given by
	\begin{align*}
		F^p({\rm Gr}^W_m\Omega^\bullet_{\overline{f},{\rm BM}})&\simeq_{\rm qis}\bigoplus_{i\geq0}F^p{\rm Gr}^W_{m-i}(\Omega^{\bullet}_{\overline{X}_{[i+1]}/\overline{S}}(\log \overline{D}_{[i+1]})(-i)[-2i])[i]\\\nonumber
		&\simeq_{\rm qis}\bigoplus_{i\geq0}F^p(\Omega^\bullet_{\overline{D}_{[i+1],[m-i]}/\overline{S}}(\log E_{[i+1],[m-i]})(-m)[-m])\\\nonumber
		&\simeq_{\rm qis}\bigoplus_{i\geq0}\Omega^{\geq p-m}_{\overline{D}_{[i+1],[m-i]}/\overline{S}}(\log E_{[i+1],[m-i]})[-m].
	\end{align*}
	Therefore
	\begin{align}\label{align_GRFW_logBM}
		{\rm Gr}^p_F{\rm Gr}^W_m\Omega^\bullet_{\overline{f},{\rm BM}}\simeq_{\rm qis}\bigoplus_{i\geq0}\Omega^{p-m}_{\overline{D}_{[i+1],[m-i]}/\overline{S}}(\log E_{[i+1],[m-i]})[-p].
	\end{align}
\subsection{The logarithmic weight spectral sequence}
	Recall that for a flat connection $(\mathcal{V}, \nabla)$ on $S$, the lower canonical extension of $(\mathcal{V}, \nabla)$ to $\overline{S}$ is a logarithmic connection
	\[
	\nabla: \widetilde{\mathcal{V}}\to \widetilde{\mathcal{V}}\otimes_{\sO_{\overline{S}}} \Omega_{\overline{S}}(\log D_{\overline{S}})
	\]
	that extends $(\mathcal{V}, \nabla)$ such that the real parts of the eigenvalues of the residue maps along $D_{\overline{S}}$ lie in the interval $[0,1)$. According to \cite[Proposition 5.4]{Deligne1970}, the lower canonical extension always exists and is unique up to isomorphism.
	
	According to (\ref{align_BM_GrW_rel_ss}), the weight filtration $W_\bullet$ (\ref{align_weight_filtration_BM}) on $\Omega^\bullet_{\overline{f},{\rm BM}}$ induces a spectral sequence 
	\begin{align}\label{align_rel_ss_BM_ss}
	E^{-m,k+m}_1(R\overline{f}_\ast(\Omega^\bullet_{\overline{f},{\rm BM}}),W):=\bigoplus_{i\geq0} R^{k-m}\overline{f}_\ast(\Omega^\bullet_{\overline{D}_{[i+1],[m-i]}/\overline{S}}(\log E_{[i+1],[m-i]}))(-m)\Rightarrow R^{k}\overline{f}_\ast(\Omega^\bullet_{\overline{f},{\rm BM}}).
	\end{align}
\begin{prop}\label{prop_weightss_E2_BM}
	The Gauss-Manin connection 
	$$\nabla_{\rm GM}:R^q\overline{f}_\ast(\Omega^\bullet_{\overline{f},{\rm BM}})\to \Omega_{\overline{S}}(\log D_{\overline{S}})\otimes_{\sO_{\overline{S}}}R^q\overline{f}_\ast(\Omega^\bullet_{\overline{f},{\rm BM}}).$$ is isomorphic to Deligne's lower canonical extension of $$\nabla_{\rm GM}:R^q f_\ast(\Omega^\bullet_{f,{\rm BM}})\to \Omega_S\otimes_{\sO_S}R^qf_\ast(\Omega^\bullet_{f,{\rm BM}}).$$
	In particular, each $R^q\overline{f}_\ast(\Omega^\bullet_{\overline{f},{\rm BM}})$ is locally free. Moreover, the weight spectral sequence (\ref{align_rel_ss_BM_ss}) degenerates at the $E_2$-page. Consequently, one has
	$$E^{-m,k+m}_2(R\overline{f}_\ast(\Omega^\bullet_{\overline{f},{\rm BM}}),W)\simeq {\rm Gr}_{k+m}^WR^{k}\overline{f}_\ast(\Omega^\bullet_{\overline{f},{\rm BM}})$$ and ${\rm Gr}_{m}^W R^{k}\overline{f}_\ast(\Omega^\bullet_{\overline{f},{\rm BM}})$, $W_m R^{k}\overline{f}_\ast(\Omega^\bullet_{\overline{f},{\rm BM}})$ are locally free for every $k$ and $m$.
\end{prop}
\begin{proof}   
    For every $k > 0$ and $l \geq 0$, note that since $(\overline{D}_{[k],[l]}, E_{[k],[l]}) \to (\overline{S}, D_{\overline{S}})$ is a semistable family extending the smooth family $D_{[k],[l]} \to S$, it follows from \cite[Proposition 2.20]{Steenbrink1975} that $$(R^q \overline{f}_\ast (\Omega^\bullet_{\overline{D}_{[k],[l]}/\overline{S}}(\log E_{[k],[l]})), \nabla_{\text{GM}})$$ is the lower canonical extension of $$(R^q f_\ast (\Omega^\bullet_{D_{[k],[l]}/S}), \nabla_{\text{GM}})$$ for every $q$.
    Consequently, in the spectral sequence (\ref{align_rel_ss_BM_ss}),
	\[
	E^{-m,k+m}_1(R\overline{f}_\ast(\Omega^\bullet_{\overline{f},{\rm BM}}),W):=\bigoplus_{i\geq0} R^{k-m}\overline{f}_\ast(\Omega^\bullet_{\overline{D}_{[i+1],[m-i]}/\overline{S}}(\log E_{[i+1],[m-i]}))
	\]
	is the lower canonical extension of 
	\[
	E^{-m,k+m}_1(Rf_\ast(\Omega^\bullet_{f,{\rm BM}}),W):=\bigoplus_{i\geq0} R^{k-m}f_\ast(\Omega^\bullet_{D_{[i+1],[m-i]}/S}).
	\]
	Since taking the lower canonical extension is an exact functor \cite[Proposition 5.4(ii)]{Deligne1970}, it follows that 
	\[
	E_2^{-m,k+m}(R\overline{f}_\ast(\Omega^\bullet_{\overline{f},{\rm BM}}), W)
	\]
	is the lower canonical extension of 
	\[
	E^{-m,k+m}_2(Rf_\ast(\Omega^\bullet_{f,{\rm BM}}), W)
	\]
	for every $k$ and $m$.
	In particular, this implies that $E_2^{-m,k+m}(R\overline{f}_\ast(\Omega^\bullet_{\overline{f},{\rm BM}}), W)$ is locally free for every $k$ and $m$. Once the local freeness is established, the spectral sequence (\ref{align_rel_ss_BM_ss}) degenerates at the $E_2$-page because its restriction to $S$ degenerates at the $E_2$-page (Theorem \ref{thm_BM_preVMHS} (1)). Consequently,
	\[
	R^k\overline{f}_\ast(\Omega^\bullet_{\overline{f},{\rm BM}}) \simeq \bigoplus_{m} E_2^{-m,k+m}(R\overline{f}_\ast(\Omega^\bullet_{\overline{f},{\rm BM}}), W)
	\]
	is locally free and is the lower canonical extension of 
	\[
	R^k f_\ast(\Omega^\bullet_{f,{\rm BM}}) \simeq \bigoplus_{m} E_2^{-m,k+m}(Rf_\ast(\Omega^\bullet_{f,{\rm BM}}), W)
	\]
	for every $k$. The local freeness of ${\rm Gr}_{m}^W R^{k}\overline{f}_\ast(\Omega^\bullet_{\overline{f},{\rm BM}})$ follows from the $E_2$-degeneration and the local freeness of $R^k\overline{f}_\ast(\Omega^\bullet_{\overline{f},{\rm BM}})$.
	
	Thus, the proposition is proved.
\end{proof}
\subsection{The logarithmic Hodge spectral sequence}
The logarithmic Hodge filtration $F^\bullet$ (\ref{align_log_Hodge_fil}) on $\Omega^\bullet_{\overline{f},{\rm BM}}$ induces a spectral sequence
\begin{align}\label{align_ss_logHodge}
E^{p,q}_1:=R^{p+q}\overline{f}_\ast({\rm Gr}_{F}^p\Omega^\bullet_{\overline{f},{\rm BM}})\Rightarrow R^{p+q}\overline{f}_\ast(\Omega^\bullet_{\overline{f},{\rm BM}}).
\end{align}
The main result of this section is the $E_1$-degeneration of this spectral sequence (Proposition \ref{prop_log_Hodge_fil_BM}). Prior to presenting the proof of this main result, we shall review the theory of the degeneration of variations of Hodge structures in the geometric context, which was developed by Steenbrink \cite{Steenbrink1975,Steenbrink1977} and further extended by Qianyu Chen \cite{Chen2023}.

Let $Y \to \Delta$ be a projective morphism from a complex manifold $Y$ to the unit disc $\Delta$, with $0$ as its unique critical value. Suppose that the schematic fiber $Y_0 := f^{-1}\{0\}$ is a (possibly non-reduced) simple normal crossing divisor on $Y$. Denote $Y^\ast = Y \setminus Y_0$. Then, $R^k f_\ast(\Omega^\bullet_{Y^\ast/\Delta^\ast})$ underlies a variation of Hodge structure of weight $k$. Furthermore, $(R^k f_\ast(\Omega^\bullet_{Y/\Delta}(\log Y_0)), \nabla_{\rm GM})$ is isomorphic to the lower canonical extension of $(R^k f_\ast(\Omega^\bullet_{Y^\ast/\Delta^\ast}), \nabla_{\rm GM})$. The logarithmic Hodge spectral sequence  
\begin{align*}  
	E_1^{p,q} := R^q f_\ast(\Omega^{p}_{Y/\Delta}(\log Y_0)) \Rightarrow R^{p+q} f_\ast(\Omega^\bullet_{Y/\Delta}(\log Y_0))  
\end{align*}  
degenerates at $E_1$. Moreover, $R^q f_\ast(\Omega^{p}_{Y/\Delta}(\log Y_0))$ is locally free for all $p$ and $q$. The canonical morphism  
$$  
F^p:= R^k f_\ast(\Omega^{\geq p}_{Y/\Delta}(\log Y_0)) \to R^k f_\ast(\Omega^\bullet_{Y/\Delta}(\log Y_0))  
$$  
is injective for every $k$, thereby defining a filtration $F^\bullet$ on $R^k f_\ast(\Omega^\bullet_{Y/\Delta}(\log Y_0))$ consisting of subbundles with locally free subquotients.

Let $\cV := R^k f_\ast(\Omega^\bullet_{Y/\Delta}(\log Y_0))$ for simplicity. Denote by $R : \cV(0) \to \cV(0)$ the residue operator associated with $\nabla_{\rm GM}$. Let $R = P + N$ be the Jordan decomposition, where $P$ is a semisimple operator and $N$ is nilpotent. Since the monodromy of $\nabla_{\rm GM}$ is quasi-unipotent, the eigenvalues of $P$ lie in $[0,1) \cap \bQ$. The nilpotent operator $N$ defines the monodromy weight filtration $W_\bullet$ on $\cV(0)$, which satisfies the following two properties:  
\begin{itemize}  
	\item $N(W_i) \subset W_{i-2}$ for every $i$;  
	\item $N^l : {\rm Gr}^W_{k+l}\cV(0) \to {\rm Gr}^W_{k-l}\cV(0)$ is an isomorphism for every $l$.  
\end{itemize}  
Let $F^\bullet(0)$ denote the filtration on $\cV(0)$ induced by $F^\bullet$. When the monodromy operator of $\nabla_{\rm GM}$ is unipotent (i.e., $P = 0$), $(\cV(0), W_\bullet, F^\bullet(0))$ forms a mixed Hodge structure (\cite{Steenbrink1975}). However, in the general case (i.e., $P \neq 0$), the situation becomes more intricate.

Let $0\leq\lambda_1 < \lambda_2 < \cdots < \lambda_m<1$ be the eigenvalues of $P$, which induce a filtration of logarithmic sub-connections:  
$$
0 = \cV_0 \subset \cV_1 \subset \cdots \subset \cV_m = \cV.
$$  
Here, $\cV_i$ is the submodule of $\cV$ spanned by the generalized eigenspaces $\ker(R - \lambda)^\infty$ for $\lambda \geq \lambda_i$. Let $W_{\bullet,i}$ and $F^\bullet_i$ denote the induced filtrations on $\cV_i(0)/\cV_{i-1}(0)$, defined as follows:  
$$
W_{m,i} = \frac{W_m \cap \cV_i(0) + \cV_{i-1}(0)}{\cV_{i-1}(0)}, \quad \forall m,
$$  
$$
F^p_i = \frac{F^p \cap \cV_i(0) + \cV_{i-1}(0)}{\cV_{i-1}(0)}, \quad \forall p.
$$  
\begin{thm}\emph{\cite[Theorem A]{Chen2023}}\label{thm_LMHS}
	For every $i=1,\dots,m$, $(\cV_i(0)/\cV_{i-1}(0),W_{\bullet,i}, F^\bullet_i)$ forms a mixed Hodge structure.
\end{thm}
The following lemma will be used in proving Proposition \ref{prop_log_Hodge_fil_BM}.
\begin{lem}\label{lem_filstrict}
	Let $(V, V_\bullet, F^\bullet_V)$ and $(W, W_\bullet, F^\bullet_W)$ be two bi-filtered vector spaces, where $V_\bullet$ and $W_\bullet$ are regular increasing filtrations on $V$ and $W$, respectively, and $F^\bullet_V$ and $F^\bullet_W$ are regular decreasing filtrations on $V$ and $W$, respectively. Let $\varphi : V \to W$ be a linear operator that preserves both filtrations. For each $i$, let $F^\bullet(V_i / V_{i-1})$ denote the filtration induced on $V_i / V_{i-1}$, defined as  
	$$
	F^p(V_i / V_{i-1}) = \frac{F^p_V \cap V_i + V_{i-1}}{V_{i-1}}, \quad \forall p.
	$$  
	Similarly, define $F^\bullet(W_i / W_{i-1})$ for the filtration induced on $W_i / W_{i-1}$. Assume that the induced operator $$\varphi_i := \varphi\!\!\!\!\mod V_{i-1} : V_i / V_{i-1} \to W_i / W_{i-1}$$ is strictly compatible with $F^\bullet(V_i / V_{i-1})$ and $F^\bullet(W_i / W_{i-1})$ for all $i = 1, \dots, m$. Then, $\varphi$ is strictly compatible with $F^\bullet_V$ and $F^\bullet_W$.
\end{lem}
\begin{proof}
	Our goal is to demonstrate that $F^p_W \cap \mathrm{Im}(\varphi) \subset \varphi(F^p_V)$ for every $p$. Let $a < b$ be integers such that $V_b = V$, $V_a = 0$, $W_b = W$, and $W_a = 0$. By assumption, we have  
	$$
	F^p(W_i/W_{i-1}) \cap \mathrm{Im}(\varphi_i) \subset \varphi(F^p(V_i/V_{i-1}))
	$$  
	for every $i$. This implies that  
	\begin{align}\label{align_strict_graded}
		F^p_W \cap \mathrm{Im}(\varphi) \cap W_i \subset \varphi(F^p_V) + W_{i-1}, \quad \forall i = a, \dots, b.
	\end{align}  
	Let $x \in V$ such that $\varphi(x) \in F^p_W$. We claim that for each $a \leq m \leq b$, there exists $x_m \in F^p_V$ such that $\varphi(x - x_m) \in W_{m-1}$. We prove this claim by induction. The base case $m = b$ follows from (\ref{align_strict_graded}) by taking $i = b$. Now suppose that there exists $x_{m+1} \in F^p_V$ such that $\varphi(x - x_{m+1}) \in W_m$. Then $\varphi(x - x_{m+1}) \in F^p_W \cap \mathrm{Im}(\varphi) \cap W_m$. According to (\ref{align_strict_graded}) (taking $i = m$), there exists $y \in F^p_V$ such that $\varphi(x - x_{m+1}) - \varphi(y) \in W_{m-1}$. By letting $x_m = x_{m+1} + y$, the claim follows.  
	
	As a consequence, there exists $x_{a+1} \in F^p_V$ such that $\varphi(x) - \varphi(x_{a+1}) \in W_a = 0$. Therefore, $\varphi(x) \in \varphi(F^p_V)$. This completes the proof of the lemma.
\end{proof}
\begin{prop}\label{prop_log_Hodge_fil_BM}
	The spectral sequence (\ref{align_ss_logHodge}) degenerates at the $E_1$-page. Moreover, for every $p$, $m$, and $k$, both 
	\[
	{\rm Gr}^k_F {\rm Gr}^W_{-m} R^k \overline{f}_\ast(\Omega^\bullet_{\overline{f},{\rm BM}})
	\]
	and 
	\[
	R^k \overline{f}_\ast(F^p \Omega^\bullet_{\overline{f},{\rm BM}}) \simeq F^p R^k \overline{f}_\ast(\Omega^\bullet_{\overline{f},{\rm BM}})
	\]
	are locally free.
\end{prop}
\begin{proof}
    Consider the morphism of the $E_1$-page of the weight spectral sequence (\ref{align_rel_ss_BM_ss}):
    \begin{align}\label{align_d1_E_1}
    \xymatrix{
    E^{-m,k+m}_1(R\overline{f}_\ast(\Omega^\bullet_{\overline{f},{\rm BM}}), W)\ar[d]^{d_1}\ar[r]^-{\simeq} &  \bigoplus_{i\geq0} R^{k-m}\overline{f}_\ast(\Omega^\bullet_{\overline{D}_{[i+1],[m-i]}/\overline{S}}(\log E_{[i+1],[m-i]}))(-m)\\
    E^{-m+1,k+m}_1(R\overline{f}_\ast(\Omega^\bullet_{\overline{f},{\rm BM}}), W) \ar[r]^-{\simeq}& 
    \bigoplus_{i\geq0} R^{k-m+2}\overline{f}_\ast(\Omega^\bullet_{\overline{D}_{[i+1],[m-i-1]}/\overline{S}}(\log E_{[i+1],[m-i-1]}))(-m+1)
    }.
    \end{align}
    By the lemma on two filtrations (\cite{Deligne1974}, see also \cite[Theorem 3.12]{Steenbrink2008}), $d_1$ preserves the filtration $F$, as $F_{\rm rec} = F_d = F_{d^\ast}$ on the $E_1$-terms in general, and $F_{\rm rec}$ on $E^{-m,k+m}_1(R\overline{f}_\ast(\Omega^\bullet_{\overline{f},{\rm BM}}), W)$ coincides with the logarithmic Hodge filtrations on 
    \[
    \bigoplus_{i\geq 0} R^{k-m}\overline{f}_\ast(\Omega^\bullet_{\overline{D}_{[i+1],[m-i]}/\overline{S}}(\log E_{[i+1],[m-i]}))(-m).
    \]
    We aim to demonstrate that the restriction of (\ref{align_d1_E_1}) to each point $s \in \overline{S}$,  
    \begin{align}\label{align_log_Hodge_fil_BM_d1}
    	d_1(s):& \bigoplus_{i\geq0}R^{k-m}\overline{f}_\ast(\Omega^\bullet_{\overline{D}_{[i+1],[m-i]}/\overline{S}}(\log E_{[i+1],[m-i]})) \otimes \mathbb{C}(s)\\\nonumber
    	\to & \bigoplus_{i\geq0} R^{k-m+2}\overline{f}_\ast(\Omega^\bullet_{\overline{D}_{[i+1],[m-i-1]}/\overline{S}}(\log E_{[i+1],[m-i-1]})) \otimes \mathbb{C}(s),
    \end{align}  
    is strictly compatible with the logarithmic Hodge filtrations.
    Since $d_1|_S$ is a morphism between variations of Hodge structures, it follows that $d_1(s)$ is strictly compatible with the logarithmic Hodge filtration for any $s \in S$.
    
    Now, let us consider a point $s \in \overline{S} \setminus S$. Consider those $k$ and $l$ such that $s\in \overline{f}(E_{[k],[l]})$. Then $(R^q\overline{f}_\ast(\Omega^\bullet_{\overline{D}_{[k],[l]}/\overline{S}}(\log E_{[k],[l]})),\nabla_{\rm GM})$ is a logarithmic connection whose eigenvalues of the residue map $R_s$ at $s$ lie in $[0,1)\cap\bQ$. Denote by $$\cV^q_{[k],[l]}:=R^q\overline{f}_\ast(\Omega^\bullet_{\overline{D}_{[k],[l]}/\overline{S}}(\log E_{[k],[l]}))$$ and $\cV^q_{[k],[l]}(s)=\cV^q_{[k],[l]}\otimes\bC(s)$. Let $$\{F^p:=R^q\overline{f}_\ast(\Omega^{\geq p}_{\overline{D}_{[k],[l]}/\overline{S}}(\log E_{[k],[l]}))\}$$ be the logarithmic Hodge filtration on $\cV^q_{[k],[l]}$, which induces a filtration $\{F^p(s):=F^p\otimes\bC(s)\}$ on $\cV^q_{[k],[l]}(s)$. Let $W_\bullet$ be the monodromy weight filtration on $\cV^q_{[k],[l]}(s)$ associated with the nilpotent part of $R_s$.
    
    Let $$0\leq\lambda_1 < \lambda_2 < \cdots < \lambda_m<1$$ be the eigenvalues of $R_s:\cV^q_{[k],[l]}(s)\to \cV^q_{[k],[l]}(s)$. They induce a filtration of logarithmic sub-connections:  
    $$
    0 = \cV_0 \subset \cV_1 \subset \cdots \subset \cV_m = \cV^q_{[k],[l]}(s).
    $$  
    Here, $\cV_i$ is the linear subspace of $\cV^q_{[k],[l]}(s)$ spanned by the generalized eigenspaces $\ker(R_s-\lambda)^\infty$ for $\lambda \geq \lambda_i$. Let $W_{\bullet}(\cV_i/\cV_{i-1})$ and $F^\bullet(\cV_i/\cV_{i-1})$ denote the induced filtrations on $\cV_i/\cV_{i-1}$, defined as follows:  
    $$
    W_m(\cV_i/\cV_{i-1}) = \frac{W_m \cap \cV_i + \cV_{i-1}}{\cV_{i-1}}, \quad \forall m,
    $$  
    $$
    F^p(\cV_i/\cV_{i-1}) = \frac{F^p(s) \cap \cV_i + \cV_{i-1}}{\cV_{i-1}}, \quad \forall p.
    $$  
    According to Theorem \ref{thm_LMHS}, the triple
    $$
    \left(\cV_i/\cV_{i-1}, W_\bullet(\cV_i/\cV_{i-1}), F^\bullet(\cV_i/\cV_{i-1}) \right)
    $$  
    forms a mixed Hodge structure for every $i=1,\dots,m$.
    
    Now let us analyze (\ref{align_log_Hodge_fil_BM_d1}). Since $d_1$ admits a rational structure: 
    \begin{align*}
    	d_1: \bigoplus_{i\geq0}R^{k-m}\overline{f}_\ast(\psi_{[i+1],[m-i]}) \to \bigoplus_{i\geq0} R^{k-m+2}\overline{f}_\ast(\psi_{[i+1],[m-i-1]}),
    \end{align*}  
    where $\psi_{[i+1],[m-i]}$ denotes the nearby cycle complex associated with the semistable morphism $(\overline{D}_{[i+1],[m-i]},E_{[i+1],[m-i]}) \to (\overline{S},D_{\overline{S}})$ at $s$, it follows that $d_1$ commutes with the monodromy operators. Consequently, $d_1(s)$ preserves both the generalized eigenspaces of the residue maps at $s$ and the monodromy weight filtrations. Let  
    $$
    0 \leq \lambda_1 < \lambda_2 < \cdots < \lambda_m < 1
    $$  
    be the union of eigenvalues of the residue maps (on the source and target of $d_1(s)$). Define  
    $$
    0 = \cV_0 \subset \cV_1 \subset \cdots \subset \cV_m = \bigoplus_{i\geq0}R^{k-m}\overline{f}_\ast(\Omega^\bullet_{\overline{D}_{[i+1],[m-i]}/\overline{S}}(\log E_{[i+1],[m-i]}))\otimes\bC(s),
    $$  
    where $\cV_i$ is the submodule spanned by the generalized eigenspaces $\ker(R-\lambda)^\infty$ for $\lambda \geq \lambda_i$, where $R$ is the direct sum of residue maps. Similarly, one obtains the filtration  
    $$
    0 = \cV'_0 \subset \cV'_1 \subset \cdots \subset \cV'_m = \bigoplus_{i\geq0} R^{k-m+2}\overline{f}_\ast(\Omega^\bullet_{\overline{D}_{[i+1],[m-i-1]}/\overline{S}}(\log E_{[i+1],[m-i-1]}))\otimes\bC(s),
    $$  
    where $\cV'_i$ is the submodule spanned by the generalized eigenspaces $\ker(R-\lambda)^\infty$ for $\lambda \geq \lambda_i$. Then  
    $$
    d_1(s)(\cV_i) \subset \cV'_i, \quad i=1,\dots,m,
    $$  
    and  
    $$
    d_1(s)\!\!\!\!\mod\cV_{i-1}: \frac{\cV_i}{\cV_{i-1}} \to \frac{\cV'_i}{\cV'_{i-1}}
    $$  
    preserves the monodromy weight filtrations and the logarithmic Hodge filtrations. Hence, $d_1(s)\!\!\!\mod\cV_{i-1}$ defines a morphism between mixed Hodge structures. Therefore, $d_1(s)\!\!\!\mod\cV_{i-1}$ is strictly compatible with the logarithmic Hodge filtrations for every $i=1,\dots,m$. By Lemma \ref{lem_filstrict}, we conclude that $d_1(s)$ is strictly compatible with the logarithmic Hodge filtrations.

    This demonstrates that $d_1(s)$ is strictly compatible with the logarithmic Hodge filtration for every $s \in \overline{S}$. By applying \cite[Lemma 3.4(iv)]{Fujino2014} to the filtered complex  
    \[
    \left(E^{\bullet,k+m}_1(R\overline{f}_\ast(\Omega^\bullet_{\overline{f},{\rm BM}}), W), F_{\rm rec}\right),
    \]  
    it follows that  
    \[
    {\rm Gr}^p_{F_{\rm rec}} E_2^{-m,k+m}(R\overline{f}_\ast(\Omega^\bullet_{\overline{f},{\rm BM}}), W)
    \]  
    is locally free for all $p$, $k$, and $m$, and that $F_{\rm rec} = F_d = F_{d^\ast}$ holds on $E_2^{-m,k+m}(R\overline{f}_\ast(\Omega^\bullet_{\overline{f},{\rm BM}}), W)$ (see \cite[Theorem 3.12]{Steenbrink2008}). Furthermore, the morphism (\ref{align_d1_E_1}) is strictly compatible with $F_{\rm rec}$ (see \cite[Lemma 3.4(iv)]{Fujino2014}).  
    Therefore, by Proposition \ref{prop_weightss_E2_BM}, it follows that 
    \[
    {\rm Gr}^p_{F_{\rm rec}} E_2^{-k,k+m}(R\overline{f}_\ast(\Omega^\bullet_{\overline{f},{\rm BM}}), W) \simeq {\rm Gr}^p_{F_{\rm rec}} {\rm Gr}^W_{k+m} R^{k} \overline{f}_\ast(\Omega^\bullet_{\overline{f},{\rm BM}})
    \]
    is locally free for all $p$, $q$, and $k$. 
      
    Since the spectral sequence (\ref{align_rel_ss_BM_ss}) degenerates at the $E_2$-page (Proposition \ref{prop_weightss_E2_BM}) and the morphism of the $E_1$-page is strictly compatible with $F_{\rm rec}$, it follows from the lemma on two filtrations that the spectral sequence (\ref{align_ss_logHodge}) degenerates at the $E_1$-page. Consequently,  
    \[
    R^k \overline{f}_\ast(F^p \Omega^\bullet_{\overline{f},{\rm BM}}) \simeq F^p R^k \overline{f}_\ast(\Omega^\bullet_{\overline{f},{\rm BM}})
    \]  
    is locally free for all $p$ and $k$. Furthermore, ${\rm Gr}^p_FR^k \overline{f}_\ast(\Omega^\bullet_{\overline{f},{\rm BM}})$ is locally free for all $p$ and $k$, as a result of the $E_1$-degeneration and the local freeness of $R^k \overline{f}_\ast(\Omega^\bullet_{\overline{f},{\rm BM}})$ (Proposition \ref{prop_weightss_E2_BM}).
\end{proof}
\begin{lem}\label{lem_Griff_trans_logBM}
	The logarithmic Hodge filtration $F^\bullet R^k\overline{f}_\ast(\Omega^\bullet_{\overline{f},\rm BM})$ satisfies the Griffiths tranversality:
	$$\nabla_{\rm GM}(F^p R^k\overline{f}_\ast(\Omega^\bullet_{\overline{f},\rm BM}))\subset F^{p-1} R^k\overline{f}_\ast(\Omega^\bullet_{\overline{f},\rm BM})\otimes_{\sO_{\overline{S}}} \Omega_{\overline{S}}(\log D_{\overline{S}}),\quad\forall p.$$
\end{lem}
\begin{proof}
    By considering the long exact sequences arising from the commutative diagram
	$$
	\xymatrix{
		\overline{f}^{-1}\Omega_{\overline{S}}(\log D_{\overline{S}})\otimes_{\overline{f}^{-1}\sO_{\overline{S}}}F^{p-1}\Omega^\bullet_{\overline{f},{\rm BM}}[-1]\ar[d] \ar[r]&\sigma^{\geq p}(K^0_{\overline{f}}/K^2_{\overline{f}})\ar[r]\ar[d] &F^p\Omega^\bullet_{\overline{f},{\rm BM}}\ar[d]\ar[r]&\\
		\overline{f}^{-1}\Omega_{\overline{S}}(\log D_{\overline{S}})\otimes_{\overline{f}^{-1}\sO_{\overline{S}}}\Omega^\bullet_{\overline{f},{\rm BM}}[-1]\ar[r] &K^0_{\overline{f}}/K^2_{\overline{f}}\ar[r] &\Omega^\bullet_{\overline{f},{\rm BM}}\ar[r] &
	}
	$$
	we obtain a commutative diagram:
	$$
	\xymatrix{
		R^k\overline{f}_\ast(F^p\Omega^\bullet_{\overline{f},{\rm BM}})\ar[r]\ar[d] & \Omega_{\overline{S}}(\log D_{\overline{S}})\otimes_{\sO_{\overline{S}}}R^k\overline{f}_\ast(F^{p-1}\Omega^\bullet_{\overline{f},{\rm BM}})\ar[d]\\
		R^k\overline{f}_\ast(\Omega^\bullet_{\overline{f},{\rm BM}})\ar[r]^-{\nabla_{\rm GM}} & \Omega_{\overline{S}}(\log D_{\overline{S}})\otimes_{\sO_{\overline{S}}}R^k\overline{f}_\ast(\Omega^\bullet_{\overline{f},{\rm BM}})\\
	}.
	$$
	By Proposition \ref{prop_log_Hodge_fil_BM}, we establish the assertion of the lemma. 
\end{proof}
\subsection{Relative monodromy weight filtration}
In this section, we assume that $\overline{S} = \Delta$ is the unit disc with a coordinate $t$, and $D_{\overline{S}} = \{t = 0\}$. By Lemma \ref{lem_quasi-unipotency}, one can always achieve unipotent monodromy by performing a base change via a finite map $\Delta \to \Delta$, $z \mapsto z^m$ for some $m \in \bZ_{>0}$. Specifically, the following condition holds:  
\begin{framed}
	{\bf (UM):} For each $k\geq 1$, $l\geq 0$ and $q \geq 0$, the local system $h^q(Rf_\ast Rj_\ast \mathbb{Z}_{D_{[k],[l]}})$ has unipotent monodromy around $0$.
\end{framed}
Henceforth, we assume that condition {\bf (UM)} holds throughout this section. The main result of this section is the existence of the relative monodromy weight filtration (Theorem \ref{thm_relMonfil_BM}).
\subsubsection{The residue map}
Recall that $E := \overline{f}^{-1}(0)$. The canonical morphism  
$$
\Omega^\bullet_{\overline{f},{\rm BM}} \to \Omega^\bullet_{\overline{f},{\rm BM}} \otimes_{\sO_{\overline{X}}} \sO_{E_{\rm red}}
$$  
induces a morphism:  
\begin{align*}
	\tau_0 : R\overline{f}_\ast(\Omega^\bullet_{\overline{f},{\rm BM}}) \otimes\bC(0) \to R\Gamma(E_{\rm red}, \Omega^\bullet_{\overline{f},{\rm BM}} \otimes_{\sO_{\overline{X}}} \sO_{E_{\rm red}}).
\end{align*}  
\begin{lem}\label{lem_log_ext_fiber}
	$\tau_{0}$ is a quasi-isomorphism.
\end{lem}
\begin{proof} 
	The weight filtration $W_\bullet$ (\ref{align_weight_filtration_BM}) on $\Omega^\bullet_{\overline{f},{\rm BM}}$ induces filtrations (still denoted by $W_\bullet$) on $\Omega^\bullet_{\overline{f},{\rm BM}} \otimes_{\sO_{\overline{X}}} \sO_{E}$ and $\Omega^\bullet_{\overline{f},{\rm BM}} \otimes_{\sO_{\overline{X}}} \sO_{E_{\rm red}}$, respectively. The restriction gives a filtered morphism:  
	\begin{align}  
	\tau : (\Omega^\bullet_{\overline{f},{\rm BM}} \otimes_{\sO_{\overline{X}}} \sO_{E}, W_\bullet) \to (\Omega^\bullet_{\overline{f},{\rm BM}} \otimes_{\sO_{\overline{X}}} \sO_{E_{\rm red}}, W_\bullet).  
	\end{align}  
	
	By (\ref{align_BM_GrW_rel_ss}), we have  
	\begin{align*}  
	{\rm Gr}^W_m(\Omega^\bullet_{\overline{f},{\rm BM}} \otimes_{\sO_{\overline{X}}} \sO_{E}) = \bigoplus_{i \geq 0} R^{k-m}\overline{f}_\ast(\Omega^\bullet_{\overline{D}_{[i+1],[m-i]}/\Delta}(\log E_{[i+1],[m-i]})) \otimes \sO_{E_{[i+1],[m-i]}},  
	\end{align*}  
	and  
	\begin{align*}  
	{\rm Gr}^W_m(\Omega^\bullet_{\overline{f},{\rm BM}} \otimes_{\sO_{\overline{X}}} \sO_{E_{\rm red}}) = \bigoplus_{i \geq 0} R^{k-m}\overline{f}_\ast(\Omega^\bullet_{\overline{D}_{[i+1],[m-i]}/\Delta}(\log E_{[i+1],[m-i]})) \otimes \sO_{E_{[i+1],[m-i],{\rm red}}}.  
	\end{align*}  
	
	According to \cite[\S 4.14]{Steenbrink1975} and the assumption {\bf (UM)}, the induced morphism  
	\begin{align*}  
		{\rm Gr}^W_m(\tau) : {\rm Gr}^W_m(\Omega^\bullet_{\overline{f},{\rm BM}} \otimes_{\sO_{\overline{X}}}\sO_{E}) \to {\rm Gr}^W_m(\Omega^\bullet_{\overline{f},{\rm BM}} \otimes_{\sO_{\overline{X}}}\sO_{E_{\rm red}})  
	\end{align*}  
	is a quasi-isomorphism for every $m$. Consequently, $\tau$ is a quasi-isomorphism. As a result, the morphism  
	\begin{align}\label{align_limit_fiber1}
		R\Gamma(E, \Omega^\bullet_{\overline{f},{\rm BM}} \otimes_{\sO_{\overline{X}}} \sO_{E}) \to R\Gamma(E_{\rm red}, \Omega^\bullet_{\overline{f},{\rm BM}} \otimes_{\sO_{\overline{X}}} \sO_{E_{\rm red}})
	\end{align}  
	is also a quasi-isomorphism. By Theorem \ref{prop_weightss_E2_BM}, $R^k\overline{f}_\ast(\Omega^\bullet_{\overline{f},{\rm BM}})$ is locally free and commutes with arbitrary base changes for every $k$. Therefore, there exists a canonical isomorphism  
	\begin{align}\label{align_limit_fiber2}
		R\overline{f}_\ast(\Omega^\bullet_{\overline{f},{\rm BM}}) \otimes \bC(0) \simeq R\Gamma(E, \Omega^\bullet_{\overline{f},{\rm BM}} \otimes_{\sO_{\overline{X}}} \sO_{E}).
	\end{align}  
	Combining (\ref{align_limit_fiber1}) and (\ref{align_limit_fiber2}), we conclude that $\tau_0$ is a quasi-isomorphism.
\end{proof}
The residue sequences  
\begin{align*}  
0 \to & \Omega^\bullet_{\overline{X}_{[k]}/\Delta}(\log\overline{D}_{[k]}) \otimes_{\sO_{\overline{X}_{[k]}}} \sO_{E_{[k],\rm red}}[-1] \stackrel{\wedge\frac{dt}{t}}{\longrightarrow} \Omega^\bullet_{\overline{X}_{[k]}}(\log\overline{D}_{[k]}) \otimes_{\sO_{\overline{X}_{[k]}}} \sO_{E_{[k],\rm red}} \\  
\to & \Omega^\bullet_{\overline{X}_{[k]}/\Delta}(\log\overline{D}_{[k]}) \otimes_{\sO_{\overline{X}_{[k]}}} \sO_{E_{[k],\rm red}} \to 0  
\end{align*}  
for all $k \geq 0$ induce the residue distinguished triangle:  
\begin{align}\label{align_residue_triangle}  
\Omega^\bullet_{\overline{f},{\rm BM}} \otimes_{\sO_{\overline{X}}} \sO_{E_{\rm red}}[-1] \stackrel{\wedge\frac{dt}{t}}{\longrightarrow} \Omega^\bullet_{(\overline{X},\overline{D}),{\rm BM}} \otimes_{\sO_{\overline{X}}} \sO_{E_{\rm red}} \to \Omega^\bullet_{\overline{f},{\rm BM}} \otimes_{\sO_{\overline{X}}} \sO_{E_{\rm red}} \to.  
\end{align}  
The connecting map of this triangle yields a natural morphism  
\begin{align*}  
{\rm Res}^q_0 : \bH^q(E_{\rm red}, \Omega^\bullet_{\overline{f},{\rm BM}} \otimes_{\sO_{\overline{X}}} \sO_{E_{\rm red}}) \to \bH^q(E_{\rm red}, \Omega^\bullet_{\overline{f},{\rm BM}} \otimes_{\sO_{\overline{X}}} \sO_{E_{\rm red}})  
\end{align*}  
for every $q$.
\begin{lem}\label{lem_residue_BM}
	The map ${\rm Res}^q_{0}$ is the residue map associated with  
	$$
	\nabla_{\rm GM}: R^q\overline{f}_\ast(\Omega^\bullet_{\overline{f},{\rm BM}}) \to \Omega_{\Delta}(\log 0) \otimes_{\sO_{\Delta}} R^q\overline{f}_\ast(\Omega^\bullet_{\overline{f},{\rm BM}})
	$$  
	via the identification of fibers given by Lemma \ref{lem_log_ext_fiber}:  
	$$
	\tau^q_{0}: R^q\overline{f}_\ast(\Omega^\bullet_{\overline{f},{\rm BM}}) \otimes \bC(0) \simeq \bH^q(E_{\rm red}, \Omega^\bullet_{\overline{f},{\rm BM}} \otimes_{\sO_{\overline{X}}} \sO_{E_{\rm red}}).
	$$  
\end{lem}
\begin{proof}
	Consider the commutative diagram of distinguished triangles:
	\begin{align*}
	\xymatrix{
			\overline{f}^\ast\Omega_{\Delta}(\log 0)\otimes_{\overline{f}^{-1}\sO_{\Delta}}\Omega^\bullet_{\overline{f},{\rm BM}}[-1] \ar[r]\ar[d]^{\overline{f}^{-1}({\rm Res}_0)}&\Omega^\bullet_{(\overline{X},\overline{D}),{\rm BM}}\ar[r]\ar[d] &\Omega^\bullet_{\overline{f},{\rm BM}}\ar[r]\ar[d]& \\
			\Omega^\bullet_{\overline{f},{\rm BM}} \otimes_{\sO_{\overline{X}}} \sO_{E_{\rm red}}[-1]\ar[r]^{\wedge\frac{dt}{t}}&\Omega^\bullet_{(\overline{X},\overline{D}),{\rm BM}} \otimes_{\sO_{\overline{X}}} \sO_{E_{\rm red}}\ar[r] &\Omega^\bullet_{\overline{f},{\rm BM}} \otimes_{\sO_{\overline{X}}} \sO_{E_{\rm red}}\ar[r]&
	}.
	\end{align*}
	where ${\rm Res}_0:\Omega_{\Delta}(\log 0)\to\bC(0)$ is the residue map.
	The connecting morphisms yields a commutative diagram
	\begin{align*}
	\xymatrix{
		R^q\overline{f}_\ast(\Omega^\bullet_{\overline{f},{\rm BM}}) \ar[r]^-{\nabla_{\rm GM}}\ar[d]^{\rm Lemma\  \ref{lem_log_ext_fiber}} & \Omega_{\Delta}(\log 0) \otimes_{\sO_{\Delta}} R^q\overline{f}_\ast(\Omega^\bullet_{\overline{f},{\rm BM}}) \ar[d]^{\rm Res_0} \\
		\bH^q(E_{\rm red}, \Omega^\bullet_{\overline{f},{\rm BM}} \otimes_{\sO_{\overline{X}}} \sO_{E_{\rm red}}) \ar[r]^{{\rm Res}^q_0} & \bH^q(E_{\rm red}, \Omega^\bullet_{\overline{f},{\rm BM}} \otimes_{\sO_{\overline{X}}} \sO_{E_{\rm red}}) 
	},
	\end{align*}
	Therefore, the lemma is established.
\end{proof}
\subsubsection{Bi-complex resolution}
Motivated by \cite{Steenbrink1975}, we aim to introduce a bi-complex resolution of $\Omega^\bullet_{\overline{f},{\rm BM}} \otimes_{\sO_{\overline{X}}} \sO_{E_{\rm red}}$, where the operator ${\rm Res}_0(\nabla_{\rm GM})$ is defined at the level of complexes.

Let $k \geq 1$ be given. Observe that $(\overline{X}_{[k]}, \overline{D}_{[k]})$ forms a log smooth pair, and the morphism $(\overline{X}_{[k]}, \overline{D}_{[k]}) \to (\Delta, 0)$ is semistable with log smooth general fibers. Define $E_{[k]}:=i^{-1}_{[k]}(E)$ as the pullback of $E$ via the natural morphism $\overline{X}_{[k]} \to \overline{X}$, and let $\overline{D}^h_{[k]}$ denote the horizontal component of $\overline{D}_{[k]}$.

Let $W^{E_{[k]}}_\bullet \Omega^{\bullet}_{\overline{X}_{[k]}}(\log \overline{D}_{[k]})$ represent the weight filtration on $\Omega^{\bullet}_{\overline{X}_{[k]}}(\log \overline{D}_{[k]})$ with respect to the divisor $E_{[k]}$, which is explicitly defined as follows:  
\begin{align*}  
	W^{E_{[k]}}_m \Omega^{p}_{\overline{X}_{[k]}}(\log \overline{D}_{[k]}) = \begin{cases}  
		0, & m < 0 \\  
		\Omega^p_{\overline{X}_{[k]}}(\log \overline{D}_{[k]}), & m \geq p \\  
		\Omega^{p-m}_{\overline{X}_{[k]}}(\log \overline{D}^h_{[k]}) \wedge \Omega^m_X(\log \overline{D}_{[k]}), & 0 \leq m \leq p  
	\end{cases}.  
\end{align*}  
Let $A^{\bullet,\bullet}_{(\overline{X}_{[k]},\overline{D}_{[k]})/ \Delta}$ denote a bi-complex of sheaves on $E_{[k]}$, where  
$$
A^{p,q}_{(\overline{X}_{[k]},\overline{D}_{[k]})/ \Delta} = \Omega^{p+q+1}_{\overline{X}_{[k]}}(\log \overline{D}_{[k]}) / W^{E_{[k]}}_p \Omega^{p+q+1}_{\overline{X}_{[k]}}(\log \overline{D}_{[k]}).
$$  
The differentials are defined as follows:  
$$
d' = \frac{dt}{t} \wedge : A^{p,q}_{(\overline{X}_{[k]},\overline{D}_{[k]})/ \Delta} \to A^{p+1,q}_{(\overline{X}_{[k]},\overline{D}_{[k]})/ \Delta},
$$  
and  
$$
d'' = d : A^{p,q}_{(\overline{X}_{[k]},\overline{D}_{[k]})/\Delta} \to A^{p,q+1}_{(\overline{X}_{[k]},\overline{D}_{[k]})/\Delta}.
$$  
The bi-complex $A^{\bullet,\bullet}_{\overline{X}_{[k]}/ \Delta}$ is equipped with a weight filtration given by  
$$
W^{\overline{f}}_{m}(A^{p,q}_{(\overline{X}_{[k]},\overline{D}_{[k]})/\Delta}) := W_{m}^{\overline{D}^h_{[k]}} \Omega^{p+q+1}_{\overline{X}_{[k]}}(\log \overline{D}_{[k]}) / W^{E_{[k]}}_p \Omega^{p+q+1}_{\overline{X}_{[k]}}(\log \overline{D}_{[k]}),
$$  
where $W^{\overline{D}^h_{[k]}}_\bullet \Omega^{\bullet}_{\overline{X}_{[k]}}(\log \overline{D}_{[k]})$ denotes the weight filtration on $\Omega^{\bullet}_{\overline{X}_{[k]}}(\log \overline{D}_{[k]})$ with respect to the divisor $\overline{D}^h_{[k]}$. Explicitly,  
$$
W^{\overline{D}^h_{[k]}}_m \Omega^{p}_{\overline{X}_{[k]}}(\log \overline{D}_{[k]}) = 
\begin{cases}  
	0, & m < 0 \\  
	\Omega^p_{\overline{X}_{[k]}}(\log \overline{D}_{[k]}), & m \geq p \\  
	\Omega^{p-m}_{\overline{X}_{[k]}}(\log E_{[k]}) \wedge \Omega^m_X(\log \overline{D}_{[k]}), & 0 \leq m \leq p.  
\end{cases}
$$  
Additionally, $A^{\bullet,\bullet}_{(\overline{X}_{[k]},\overline{D}_{[k]})/\Delta}$ is endowed with a monodromy weight filtration defined as  
$$
W_{m}(A^{p,q}_{(\overline{X}_{[k]},\overline{D}_{[k]})/\Delta}) := W^{\overline{D}_{[k]}}_{m+2p+1} \Omega^{p+q+1}_{\overline{X}_{[k]}}(\log \overline{D}_{[k]}) / W^{E_{[k]}}_p \Omega^{p+q+1}_{\overline{X}_{[k]}}(\log \overline{D}_{[k]}), \quad m \in \mathbb{Z}.
$$  
We define the Hodge filtration as:  
$$
F^r(A^{\bullet,\bullet}_{(\overline{X}_{[k]},\overline{D}_{[k]})/\Delta}) = \bigoplus_{p} \bigoplus_{q \geq r} A^{p,q}_{(\overline{X}_{[k]},\overline{D}_{[k]})/\Delta}, \quad r \in \mathbb{Z}.
$$  
Let $A^{\bullet}_{(\overline{X}_{[k]},\overline{D}_{[k]})/ \Delta}$ denote the total complex of $A^{\bullet,\bullet}_{(\overline{X}_{[k]},\overline{D}_{[k]})/ \Delta}$. The induced monodromy weight filtration, weight filtration, and Hodge filtration on $A^{\bullet}_{(\overline{X}_{[k]},\overline{D}_{[k]})/ \Delta}$ are denoted by $W_\bullet(A^{\bullet}_{(\overline{X}_{[k]},\overline{D}_{[k]})/ \Delta})$, $W^{\overline{f}}_{[k],\bullet}(A^{\bullet}_{(\overline{X}_{[k]},\overline{D}_{[k]})/ \Delta})$, and $F^\bullet(A^{\bullet}_{(\overline{X}_{[k]},\overline{D}_{[k]})/ \Delta})$, respectively.  

The morphism  
$$
\Omega^q_{\overline{X}_{[k]}}(\log \overline{D}_{[k]}) \to A^{0,q}_{(\overline{X}_{[k]},\overline{D}_{[k]})/ \Delta}, \quad \omega \mapsto (-1)^q \frac{dt}{t} \wedge \omega\!\!\!\!\mod W^{E_{[k]}}_0
$$  
induces a morphism  
\begin{align}\label{align_mapthetak}
	\theta_{[k]} : \Omega^\bullet_{\overline{X}_{[k]}/\Delta}(\log \overline{D}_{[k]}) \otimes_{\sO_{\overline{X}_{[k]}}}\sO_{E_{\rm red}} \to A^{\bullet}_{(\overline{X}_{[k]},\overline{D}_{[k]})/\Delta}.
\end{align} 
On the other hand, the projection map  
$$
A^{p,q}_{(\overline{X}_{[k]},\overline{D}_{[k]})/\Delta} \to A^{p+1,q-1}_{(\overline{X}_{[k]},\overline{D}_{[k]})/ \Delta}, \quad \omega \mapsto \omega\!\!\!\!\mod W^{E_{[k]}}_{p+1}
$$  
commutes with $d'$ and $d''$. Consequently, it induces an endomorphism  
\begin{align}\label{align_mapnuk}
	\nu_{[k]} : A^{\bullet}_{(\overline{X}_{[k]},\overline{D}_{[k]})/ \Delta} \to A^{\bullet}_{(\overline{X}_{[k]},\overline{D}_{[k]})/ \Delta}.
\end{align}  
Let  
\begin{align*}  
	A^\bullet_{\overline{f},{\rm BM}} := {\rm Tot}\left(\cdots \to A^{\bullet}_{(\overline{X}_{[2]},\overline{D}_{[2]})/ \Delta}[-2] \to A^{\bullet}_{(\overline{X}_{[1]},\overline{D}_{[1]})/ \Delta}\right).  
\end{align*}  
The complex $A^\bullet_{\overline{f},{\rm BM}}$ is endowed with a weight filtration  
\begin{align*}  
	W^{\overline{f}}_m := {\rm Tot}\left(\cdots \to W^{\overline{f}}_{m-1}A^{\bullet}_{(\overline{X}_{[2]},\overline{D}_{[2]})/ \Delta}[-2] \to W^{\overline{f}}_mA^{\bullet}_{(\overline{X}_{[1]},\overline{D}_{[1]})/ \Delta}\right), \quad m \in \bZ,  
\end{align*}  
and a monodromy weight filtration  
\begin{align*}  
	W_m := {\rm Tot}\left(\cdots \to W_{m-1}A^{\bullet}_{(\overline{X}_{[2]},\overline{D}_{[2]})/ \Delta}[-2] \to W_mA^{\bullet}_{(\overline{X}_{[1]},\overline{D}_{[1]})/ \Delta}\right), \quad m \in \bZ.  
\end{align*}  
We define the Hodge filtration as  
\begin{align*}  
	F^p := {\rm Tot}\left(\cdots \to F^{p-1}A^{\bullet}_{(\overline{X}_{[2]},\overline{D}_{[2]})/ \Delta}[-2] \to F^pA^{\bullet}_{(\overline{X}_{[1]},\overline{D}_{[1]})/ \Delta}\right), \quad p \in \bZ.  
\end{align*}  
Observe that the weight filtration $W_\bullet$ and the Hodge filtration $F^\bullet$ on $\Omega^\bullet_{\overline{f},{\rm BM}}$ naturally induce corresponding filtrations on $\Omega^\bullet_{\overline{f},{\rm BM}} \otimes_{\sO_{\overline{X}}} \sO_{E_{\rm red}}$, which are still denoted by $W_\bullet$ and $F^\bullet$.  
The maps (\ref{align_mapthetak}) induce a morphism  
\begin{align*}
	\theta: \Omega^\bullet_{\overline{f},{\rm BM}} \otimes_{\sO_{\overline{X}}} \sO_{E_{\rm red}} \to A^\bullet_{\overline{f},{\rm BM}},  
\end{align*}  
which preserves both the weight filtration and the Hodge filtration on both sides.
\begin{lem}\label{lem_theta_qis}
	The map $$\theta: (\Omega^\bullet_{\overline{f},{\rm BM}} \otimes_{\sO_{\overline{X}}} \sO_{E_{\rm red}}, W_\bullet, F^\bullet) \to (A^\bullet_{\overline{f},{\rm BM}}, W^{\overline{f}}_\bullet, F^\bullet)$$ is a bi-filtered quasi-isomorphism. In particular, $\theta$ induces an isomorphism  
	$$
	\bH^q(E_{\rm red}, \Omega^\bullet_{\overline{f},{\rm BM}} \otimes_{\sO_{\overline{X}}} \sO_{E_{\rm red}}) \to \bH^q(E_{\rm red}, A^\bullet_{\overline{f},{\rm BM}})
	$$  
	for every $q$.
\end{lem}
\begin{proof}
	According to (\ref{align_GRFW_logBM}), one has  
	\begin{align*}  
		\text{Gr}^p_F \text{Gr}^W_m(\Omega^\bullet_{\overline{f},\text{BM}} \otimes_{\sO_{\overline{X}}} \sO_{E_{\text{red}}}) \simeq_{\text{qis}} \bigoplus_{i \geq 0} \Omega^{p-m}_{\overline{D}_{[i+1],[m-i]}/\Delta}(\log E_{[i+1],[m-i]})\otimes_{\sO_{\overline{D}_{[i+1],[m-i]}}} \sO_{E_{[i+1],[m-i],\text{red}}}[-p].  
	\end{align*}  
	On the other hand, one has  
	\begin{align}\label{align_Gr_Wf_BM}  
		\text{Gr}^{W^{\overline{f}}}_m A^\bullet_{\overline{f},\text{BM}} &\simeq_{\text{qis}} \bigoplus_{i \geq 0} \text{Gr}^{W^{\overline{f}}}_{m-i}(A^\bullet_{(\overline{X}_{[i+1]}, \overline{D}_{[i+1]})/\Delta})[-2i][i] \\ \nonumber  
		&\simeq_{\text{qis}} \bigoplus_{i \geq 0} A^\bullet_{(\overline{D}_{[i+1],[m-i]}, E_{[i+1],[m-i]})/\Delta}[-m],  
	\end{align}  
	with the induced Hodge filtration  
	\begin{align*}  
		F^p(\text{Gr}^{W^{\overline{f}}}_m A^\bullet_{\overline{f},\text{BM}}) &\simeq_{\text{qis}} \bigoplus_{i \geq 0} F^p \text{Gr}^{W^{\overline{f}}}_{m-i}(A^\bullet_{(\overline{X}_{[i+1]}, \overline{D}_{[i+1]})/\Delta}(-i)[-2i])[i] \\ \nonumber  
		&\simeq_{\text{qis}} \bigoplus_{i \geq 0} F^p(A^\bullet_{(\overline{D}_{[i+1],[m-i]}, E_{[i+1],[m-i]})/\Delta}(-m)[-m]) \\ \nonumber  
		&\simeq_{\text{qis}} \bigoplus_{i \geq 0} F^{p-m}A^\bullet_{(\overline{D}_{[i+1],[m-i]}, E_{[i+1],[m-i]})/\Delta}[-m].  
	\end{align*}  
	Therefore,  
	\begin{align*}  
		\text{Gr}^p_F \text{Gr}^W_m A^\bullet_{\overline{f},\text{BM}} \simeq_{\text{qis}} \bigoplus_{i \geq 0} \text{Gr}^{p-m}_F A^\bullet_{(\overline{D}_{[i+1],[m-i]}, E_{[i+1],[m-i]})/\Delta}[-m].  
	\end{align*}  
	Since each $(\overline{D}_{[k],[l]}, E_{[k],[l]}) \to (\Delta, 0)$ is a semistable family with smooth general fibers, by condition {\bf (UM)} and \cite[Lemma 4.15]{Steenbrink1975}, one has  
	$$  
	\Omega^{p-m}_{\overline{D}_{[i+1],[m-i]}/\Delta} \otimes_{\sO_{\overline{D}_{[i+1],[m-i]}}} \sO_{E_{[i+1],[m-i],\text{red}}} \simeq \text{Gr}^{p-m}_F A^\bullet_{(\overline{D}_{[i+1],[m-i]}, E_{[i+1],[m-i]})/\Delta}[p-m]  
	$$  
	for every $p$ and $m$. Thus, $\theta$ induces an isomorphism  
	\begin{align*}  
		\text{Gr}^p_F \text{Gr}^W_m(\Omega^\bullet_{\overline{f},\text{BM}} \otimes_{\sO_{\overline{X}}} \sO_{E_{\text{red}}}) \simeq_{\text{qis}} \text{Gr}^p_F \text{Gr}^{W^{\overline{f}}}_m(A^\bullet_{\overline{f},\text{BM}})
	\end{align*}  
for every $p$ and $m$.
	The lemma is proved.
\end{proof}
\subsubsection{}
The endomorphisms (\ref{align_mapnuk}) induce an endomorphism
\begin{align*}
\nu:A^\bullet_{\overline{f},{\rm BM}}\to A^\bullet_{\overline{f},{\rm BM}}
\end{align*}
which satisfies the following conditions:
\begin{align*}
	\nu(W_m)\subset W_{m-2},\quad \nu(W^{\overline{f}}_m)\subset W^{\overline{f}}_m,\quad\text{and}\quad \nu(F^p)\subset F^{p-1},\quad\forall m, p.
\end{align*}
\begin{lem}\label{lem_VMHS_BM_residue}
	The following diagram is commutative:  
	$$\xymatrix{
	\bH^q(E_{\rm red},\Omega^\bullet_{\overline{f},{\rm BM}}\otimes_{\sO_{\overline{X}}}\sO_{E_{\rm red}})\ar[r]^-{\theta}\ar[d]^{{\rm Res}_0^q}& \bH^q(E_{\rm red},A^\bullet_{\overline{f},{\rm BM}})\ar[d]^{-\nu}\\
	\bH^q(E_{\rm red},\Omega^\bullet_{\overline{f},{\rm BM}}\otimes_{\sO_{\overline{X}}}\sO_{E_{\rm red}})\ar[r]^-{\theta}&\bH^q(E_{\rm red},A^\bullet_{\overline{f},{\rm BM}})
    }.$$
\end{lem}
\begin{proof}
	Let  
	$$
	B^\bullet := \text{cone}(\nu : A^\bullet_{\overline{f},\text{BM}} \to A^\bullet_{\overline{f},\text{BM}})[-1].
	$$  
	Then $B^\bullet$ is the total complex of a double complex $B^{\bullet,\bullet}$, where $B^{p,q} = A^{p,q}_{\overline{f},\text{BM}} \oplus A^{p,q-1}_{\overline{f},\text{BM}}$. Define maps  
	$$
	\mu_{[k]}:\Omega^q_{\overline{X}_{[k]}}(\log \overline{D}_{[k]})\otimes_{\sO_{\overline{X}_{[k]}}} \sO_{E_{[k],\text{red}}} \to A^{0,q}_{(\overline{X}_{[k]},\overline{D}_{[k]})/ \Delta} \oplus A^{0,q-1}_{(\overline{X}_{[k]},\overline{D}_{[k]})/ \Delta},
	$$  
	given explicitly by  
	$$
	\mu_{[k]}(\omega) = \left(\omega \wedge \frac{dt}{t}\!\!\!\!\mod W_0^{E_{[k]}}, (-1)^{q-1} \omega\!\!\!\!\mod W_0^{E_{[k]}}\right).
	$$  
	These maps induce a morphism
	$$
	\mu : \Omega^\bullet_{\overline{f},\text{BM}} \otimes_{\sO_{\overline{X}}} \sO_{E_{\text{red}}} \to B^{\bullet}.
	$$  
	We obtain the following commutative diagram with exact columns, in which the left column is the residue distinguished triangle (\ref{align_residue_triangle}), and the horizontal maps are quasi-isomorphisms (Lemma \ref{lem_theta_qis}):  
	$$
	\xymatrix{
		\Omega^\bullet_{\overline{f},\text{BM}} \otimes_{\sO_{\overline{X}}} \sO_{E_{\text{red}}}[-1] \ar[d]^{\wedge\frac{dt}{t}} \ar[r]^-{\theta[-1]} & A^\bullet_{\overline{f},\text{BM}}[-1] \ar[d] \\
		\Omega^\bullet_{(\overline{X},\overline{D}),\text{BM}} \otimes_{\sO_{\overline{X}}} \sO_{E_{\text{red}}} \ar[r]^-{\mu} \ar[d] & B^\bullet \ar[d] \\
		\Omega^\bullet_{\overline{f},\text{BM}} \otimes_{\sO_{\overline{X}}} \sO_{E_{\text{red}}} \ar[d] \ar[r]^-{\theta} & A^\bullet_{\overline{f},\text{BM}} \ar[d] \\
		 & 
	}.
	$$  
	The lemma follows from Lemma \ref{lem_residue_BM} and the observation that the connecting homomorphism in the long exact hypercohomology sequence associated with the right-hand column is induced by $-\nu$.
\end{proof}
\begin{prop}\label{prop_LMHC_BM}
$(A^\bullet_{\overline{f},{\rm BM}}, W^{\overline{f}}_\bullet, W_\bullet, F^\bullet)$ is a cohomological filtered mixed Hodge complex in the sense of Zein \cite{Zein1991}, i.e., the following conditions are satisfied:  
\begin{enumerate}  
	\item For every $a < b$, $(W^{\overline{f}}_b/W^{\overline{f}}_a, W_\bullet, F^\bullet)$ forms a mixed Hodge complex of sheaves;  
	\item The spectral sequence associated with the hypercohomology of $(A^\bullet_{\overline{f},{\rm BM}}, W^{\overline{f}}_\bullet)$ degenerates at the $E_2$-page.  
\end{enumerate}  
\end{prop}
\begin{proof}
	Let $a<b$. According to \cite[(3.3.2)]{Zein1986} there are quasi-isomorphisms 
	\begin{align*}  
		\text{Gr}^W_m(W^{\overline{f}}_b/W^{\overline{f}}_a) 
		\simeq &\bigoplus_{i \geq 0} \text{Gr}^W_{m-i}(W^{\overline{f}}_{b-i}/W^{\overline{f}}_{a-i}(A^\bullet_{(\overline{X}_{[i+1]}, \overline{D}_{[i+1]})/\Delta}))(-i)[-i] \\  
		\simeq &\bigoplus_{\substack{i \geq 0,p \geq 0, j \leq m+p-i, \\ j \in [a-i, b-i]}} \Omega^\bullet_{E_{[i+1],[j],[m+2p-i-j+1]}}(-m-p)[-m-2p].  
	\end{align*}
    with the induced Hodge filtration given by
    \begin{align*}
    	F^{\bullet}(\text{Gr}^W_m(W^{\overline{f}}_b/W^{\overline{f}}_a))&\simeq_{\rm qis}\bigoplus_{i\geq0}F^{\bullet}\text{Gr}^W_{m-i}(W^{\overline{f}}_{b-i}/W^{\overline{f}}_{a-i}(A^\bullet_{(\overline{X}_{[i+1]}, \overline{D}_{[i+1]})/\Delta})(-i)[-2i])[i]\\\nonumber
    	&\simeq \bigoplus_{\substack{i \geq 0,p \geq 0, j \leq m+p-i, \\ j \in [a-i, b-i]}} F^{\bullet}(\Omega^\bullet_{E_{[i+1],[j],[m+2p-i-j+1]}}(-m-p)[-m-2p]).
    \end{align*} 
Here $E_{[i],[j],[k]}$ denotes the union of strata of codimension $k$ of $E_{[i],[j]}$.
	This shows that $$\left(\text{Gr}^W_m(W^{\overline{f}}_b/W^{\overline{f}}_a), F^\bullet(\text{Gr}^W_m(W^{\overline{f}}_b/W^{\overline{f}}_a))\right)$$ is a Hodge complex of sheaves of weight $m$. This proves claim (1). Statement (2) follows from Theorem \ref{prop_weightss_E2_BM}, Lemma \ref{lem_log_ext_fiber} and Lemma \ref{lem_theta_qis}.
\end{proof}
\begin{thm}\label{thm_relMonfil_BM}
	For every $k$, the induced filtration $W_\bullet$ on $\bH^k(E_{\rm red},A^\bullet_{\overline{f},{\rm BM}})$ coincides with the relative monodromy weight filtration of $(\bH^k(E_{\rm red},A^\bullet_{\overline{f},{\rm BM}}),W^{\overline{f}}_\bullet,\nu)$. Consequently, the triple  
	$$
	\left(R^k\overline{f}_\ast(\Omega^\bullet_{\overline{f},{\rm BM}}) \otimes\bC(0),W_\bullet R^k\overline{f}_\ast(\Omega^\bullet_{\overline{f},{\rm BM}}) \otimes\bC(0),{\rm Res}^k_0\right)
	$$  
	admits a relative monodromy weight filtration.
\end{thm}
\begin{proof}
Consider the hypercohomology spectral sequence of $(A^\bullet_{\overline{f},\text{BM}}, W^{\overline{f}}_\bullet)$. According to (\ref{align_Gr_Wf_BM}), one has  
\begin{align}  
	E_1^{-m,k+m}(A^\bullet_{\overline{f},\text{BM}}, W^{\overline{f}}_\bullet) \simeq_{\text{qis}} \bigoplus_{i \geq 0} \bH^{k-m}(E_{[i+1],[m-i]}, A^\bullet_{(\overline{D}_{[i+1],[m-i]}, E_{[i+1],[m-i]})/\Delta}).  
\end{align}  
The morphism and filtration on $$\bH^{k-m}(E_{[i+1],[m-i]}, A^\bullet_{(\overline{D}_{[i+1],[m-i]}, E_{[i+1],[m-i]})/\Delta})$$ induced by $\nu$ and $W_\bullet[-m]$ coincide with the residue map and the monodromy weight filtration on the limit mixed Hodge structure of the semistable family $$(\overline{D}_{[i+1],[m-i]}, E_{[i+1],[m-i]}) \to (\Delta,0)$$ (\cite[\S 4.17]{Steenbrink1975}). Therefore, by \cite[Proposition 4.23]{Steenbrink1975}, it follows that the induced filtration $W_\bullet[-m]$ on $$E_1^{-m,k+m}(A^\bullet_{\overline{f},\text{BM}}, W^{\overline{f}}_\bullet)$$ is the relative monodromy weight filtration of the endomorphism $H^k(\text{Gr}^{W^{\overline{f}}}_m\nu)$ with respect to the filtration $W^{\overline{f}}$. Since $(A^\bullet_{\overline{f},{\rm BM}}, W^{\overline{f}}_\bullet, W_\bullet, F^\bullet)$ forms a cohomological filtered mixed Hodge complex and the spectral sequence $E_r^{-m,k+m}(A^\bullet_{\overline{f},\text{BM}}, W^{\overline{f}}_\bullet)$ degenerates at the $E_2$-page (Proposition \ref{prop_LMHC_BM}), the first assertion follows directly from \cite[Lemma 3.17]{Fujino2014}. The second assertion is then a consequence of the first assertion, combined with Lemma \ref{lem_theta_qis} and Lemma \ref{lem_VMHS_BM_residue}.
\end{proof}
\subsection{Admissibility}\label{section_admissibility}
Let $f:(X,D)\to S$ be a projective simple normal crossing family with relative dimension $\dim X-\dim S=n$. Let $\widetilde{S}$ be a smooth variety containing $S$ as a Zariski open subset such that $D_{\widetilde{S}} := \widetilde{S} \setminus S$ is a simple normal crossing divisor. Denote $\cV^k_{f,{\rm BM}} := R^k f_\ast(\Omega^\bullet_{f,{\rm BM}})$. Then, as discussed in \S\ref{section_VMHS_BM}, there exists a variation of mixed Hodge structures  
$$
\bV^k_{f,{\rm BM}} = (\cV^k_{f,{\rm BM}}, \nabla_{\rm GM}, W_\bullet(\cV^k_{f,{\rm BM}}), F^\bullet(\cV^k_{f,{\rm BM}}))
$$  
for every $k \geq 0$.
\begin{thm}\label{thm_VMHS_BM_admissible}
	For every $k$, $\bV^k_{f,{\rm BM}}$ is a graded polarizable variation of mixed Hodge structures which is admissible with respect to $\widetilde{S}$. 
\end{thm}
\begin{rmk}
	The admissibility condition for variations of mixed Hodge structures was first introduced by Steenbrink-Zucker \cite{Steenbrink-Zucker1985} over one-dimensional bases and later generalized by Kashiwara \cite{Kashiwara1986} to general bases. In this paper, we adopt Kashiwara's definition.
\end{rmk}
\begin{proof}
	By \cite[Lemma 1.9.1]{Kashiwara1986} and Lemma \ref{lem_quasi-unipotency}, it suffices to assume that $S$ is a punctured disc and that condition {\bf (UM)} holds for the family. The extension of the Hodge filtration then follows from Proposition \ref{prop_log_Hodge_fil_BM}, while the existence of the relative monodromy weight filtration is guaranteed by Theorem \ref{thm_relMonfil_BM}.
\end{proof}
Denote by $\widetilde{\cV}^k_{f,{\rm BM}}$ the lower canonical extension of $\cV^k_{f,{\rm BM}}$, and let $W_m(\widetilde{\cV}^k_{f,{\rm BM}})$ be the lower canonical extension of $W_m(\cV^k_{f,{\rm BM}})$.  
\begin{cor}\emph{(Kashiwara \cite[Proposition 1.11.3]{Kashiwara1986})}\label{cor_Kashiwara_extendHodge}
	The Hodge filtration $\{F^p(\cV^k_{f,{\rm BM}})\}$ admits an extension $\{F^p(\widetilde{\cV}^k_{f,{\rm BM}})\}$ on $\widetilde{S}$ such that:  
	\begin{enumerate}  
		\item Each $F^p(\widetilde{\cV}^k_{f,{\rm BM}})$ is a subbundle of $\widetilde{\cV}^k_{f,{\rm BM}}$.  
		\item Each ${\rm Gr}_F^p {\rm Gr}_m^W(\widetilde{\cV}^k_{f,{\rm BM}})$ is locally free.  
	\end{enumerate}  
\end{cor}  
Next we give a geometric discription of Kashiwara's extension of Hodge filtrations.
As mentioned in Remark \ref{rmk_semi_log_resolution}, by employing a semi-log resolution, the simple normal crossing family $f:(X,D)\to S$ can be extended to a projective morphism $\widetilde{f}:(\widetilde{X},\widetilde{D})\to(\widetilde{S},D_{\widetilde{S}})$, where $(\widetilde{X},\widetilde{D})$ is a simple normal crossing pair and $(\widetilde{S},D_{\widetilde{S}})$ is a log smooth pair. Furthermore, it is possible to assume the existence of a closed algebraic subset $Z \subset \widetilde{S}$ of codimension $\geq 2$, such that  
$$
(\overline{S},D_{\overline{S}}) := (\widetilde{S},D_{\widetilde{S}}) \setminus Z
$$  
is a log smooth pair with a smooth boundary divisor $D_{\overline{S}}$, and  
$$
(\overline{X},\overline{D}) := (\widetilde{X},\widetilde{D}) \cap \widetilde{f}^{-1}(\overline{S})
$$  
is a simple normal crossing pair. Additionally, the restricted morphism  
$$
\overline{f} := \widetilde{f}|_{\overline{X}} : (\overline{X},\overline{D}) \to (\overline{S},D_{\overline{S}})
$$  
forms a semistable family.
\begin{thm}\label{thm_log_Higgs_geo}
	There is a natural bi-filtered isomorphism between
	$$\left(R^k\overline{f}_\ast(\Omega^\bullet_{\overline{f},{\rm BM}}),W_\bullet(R^k\overline{f}_\ast(\Omega^\bullet_{\overline{f},{\rm BM}})),F^\bullet(R^k\overline{f}_\ast(\Omega^\bullet_{\overline{f},{\rm BM}}))\right)$$
	and 
	$$\left(\widetilde{\cV}^k_{f,{\rm BM}},W_\bullet(\widetilde{\cV}^k_{f,{\rm BM}}),F^\bullet(\widetilde{\cV}^k_{f,{\rm BM}})\right)|_{\overline{S}}.$$
\end{thm}
\begin{proof}
	By 
	Proposition \ref{prop_weightss_E2_BM} there is a filtered isomorphism
	$$\tau:\left(R^k\overline{f}_\ast(\Omega^\bullet_{\overline{f},{\rm BM}}),W_\bullet(R^k\overline{f}_\ast(\Omega^\bullet_{\overline{f},{\rm BM}}))\right)\to \left(\widetilde{\cV}^k_{f,{\rm BM}},W_\bullet(\widetilde{\cV}^k_{f,{\rm BM}})\right)|_{\overline{S}}.$$
	By Proposition\ref{prop_log_Hodge_fil_BM},  $\{\tau(R^k\overline{f}_\ast(F^p\Omega^\bullet_{\overline{f},{\rm BM}}))\}_{p\geq0}$ is an extension of of the Hodge filtrations $\{F^p(\cV^k_{f,{\rm BM}})\}_{p\geq0}$ with locally free subquotients. According to the uniqueness of such extensions (\cite[Corollary 5.2]{Fujino2014}) one concludes that $\tau$ gives an isomorphism
	\begin{align*}
	\tau:F^pR^k\overline{f}_\ast(\Omega^\bullet_{\overline{f},{\rm BM}})\simeq R^k\overline{f}_\ast(F^p\Omega^\bullet_{\overline{f},{\rm BM}})\to \widetilde{F}^p(\cV^k_{f,{\rm BM}})|_{\overline{S}}.
	\end{align*}
    This completes the proof of the lemma.
\end{proof}
\subsection{The associated logarithmic Higgs bundle}\label{section_lHB_BM}
We adopt the notations introduced in \S \ref{section_admissibility}. Let $j:\overline{S}\to\widetilde{S}$ denote the open immersion. According to Theorem \ref{thm_log_Higgs_geo}, it follows that  
$$
F^p(\widetilde{\cV}^k_{f,{\rm BM}}) \simeq j_\ast\left(R^k\overline{f}_\ast(F^p\Omega^\bullet_{\overline{f},{\rm BM}})\right)
$$  
is the reflexive hull of $R^k\overline{f}_\ast(F^p\Omega^\bullet_{\overline{f},{\rm BM}})$ for every $p$. Consequently, by Lemma \ref{lem_Griff_trans_logBM}, $F^p(\widetilde{\cV}^k_{f,{\rm BM}})$ satisfies Griffiths transversality:  
\begin{align}\label{align_Griff_tran_Kashiwara_extension}
	\nabla_{\rm GM}(F^p(\widetilde{\cV}^k_{f,{\rm BM}})) \subset F^p(\widetilde{\cV}^k_{f,{\rm BM}}) \otimes_{\sO_{\widetilde{S}}} \Omega_{\widetilde{S}}(\log D_{\widetilde{S}}).
\end{align}
We observe that $F^{n+1}\widetilde{\cV}^k_{f,{\rm BM}}=0$.
Let  
$$
\widetilde{H}^{p,k-p}_{f,{\rm BM}} := {\rm Gr}_F^p(\widetilde{\cV}^k_{f,{\rm BM}}),\quad p=0,\dots,n.
$$  
Then (\ref{align_Griff_tran_Kashiwara_extension}) implies that $\nabla_{\rm GM}$ induces a morphism  
\begin{align}\label{align_logtheta_BM}
	\theta: \widetilde{H}^{p,k-p}_{f,{\rm BM}} \to \widetilde{H}^{p-1,k-p+1}_{f,{\rm BM}} \otimes_{\sO_{\widetilde{S}}} \Omega_{\widetilde{S}}(\log D_{\widetilde{S}}),\quad p=0,\dots,n.
\end{align}  
Define $\widetilde{H}^{k}_{f,{\rm BM}} := \bigoplus_{p=0}^n \widetilde{H}^{p,k-p}_{f,{\rm BM}}$, and let  
$$
\theta: \widetilde{H}^{k}_{f,{\rm BM}} \to \widetilde{H}^{k}_{f,{\rm BM}} \otimes_{\sO_{\widetilde{S}}} \Omega_{\widetilde{S}}(\log D_{\widetilde{S}})
$$  
denote the induced map. The flatness of $\nabla_{\rm GM}$ (i.e., $\nabla^2_{\rm GM} = 0$) ensures that $\theta^2 = 0$. Consequently, $(\widetilde{H}^{k}_{f,{\rm BM}}, \theta)$ defines a logarithmic Higgs bundle.
\begin{defn}\label{defn_LHB_ass_VHSBM}
	We refer to $(\widetilde{H}^{k}_{f,{\rm BM}}, \theta)$ as the lower canonical logarithmic Higgs bundle associated with the admissible variation of mixed Hodge structures $\bV^k_{f,{\rm BM}}$.
\end{defn}
This logarithmic Higgs bundle is equipped with a weight filtration of Higgs subbundles. For each $k$ and $p$, let  
$$
\{W_m(\widetilde{H}^{p,k-p}_{f,{\rm BM}}) = W_m({\rm Gr}_F^p(\widetilde{\cV}^k_{f,{\rm BM}}))\}_{m \in \bZ}
$$  
denote the induced weight filtration, and define  
$$
W_m(\widetilde{H}^{k}_{f,{\rm BM}}) := \bigoplus_{p=0}^n W_m(\widetilde{H}^{p,k-p}_{f,{\rm BM}}).
$$  
Then $(W_m(\widetilde{H}^{k}_{f,{\rm BM}}), \theta)$ forms a logarithmic Higgs subbundle, i.e.,  
$$
\theta(W_m(\widetilde{H}^{k}_{f,{\rm BM}})) \subset W_m(\widetilde{H}^{k}_{f,{\rm BM}}) \otimes_{\sO_{\widetilde{S}}} \Omega_{\widetilde{S}}(\log D_{\widetilde{S}}).
$$  
This construction yields the sub-quotient logarithmic Higgs bundles  
\begin{align}
	\theta: {\rm Gr}^{W}_m(\widetilde{H}^{k}_{f,{\rm BM}}) \to {\rm Gr}^{W}_m(\widetilde{H}^{k}_{f,{\rm BM}}) \otimes_{\sO_{\widetilde{S}}} \Omega_{\widetilde{S}}(\log D_{\widetilde{S}}).
\end{align}  
This defines the lower canonical logarithmic Higgs bundle associated with the variation of pure Hodge structure ${\rm Gr}^W_m(\bV^k_{f,{\rm BM}})$.
\begin{prop}\label{prop_negative_Higgs_kernel}
	Suppose that $\widetilde{S}$ is projective. Let $K \subset \widetilde{H}^{k}_{f,{\rm BM}}$ be a coherent subsheaf on $\widetilde{S}$ satisfying $\theta(K) = 0$. Then, the class $-c_1(K)$ is pseudo-effective.
\end{prop}
\begin{proof}
	Define  
	\[
	{\rm Gr}^{W}_{m}(K) := \frac{K \cap W_m(\widetilde{H}^{k}_{f,{\rm BM}})}{K \cap W_{m-1}(\widetilde{H}^{k}_{f,{\rm BM}})}.
	\]  
	Then, we have  
	\begin{align}\label{align_Higgs_ker_gr}
		-c_1(K) = \sum_{m} -c_1({\rm Gr}^{W}_{m}(K)).
	\end{align}  
	Note that $\bV^k_{f,{\rm BM}}$ is a graded polarized variation of mixed Hodge structures (Theorem \ref{thm_VMHS_BM}). Consequently,  
	\[
	({\rm Gr}_m^{W}(\cV^k_{f,{\rm BM}}), F^\bullet {\rm Gr}_m^{W}(\cV^k_{f,{\rm BM}}))
	\]  
	is a polarized variation of Hodge structure on $S$, and  
	\[
	\theta: {\rm Gr}^{W}_m(\widetilde{H}^{k}_{f,{\rm BM}}) \to {\rm Gr}^{W}_m(\widetilde{H}^{k}_{f,{\rm BM}}) \otimes \Omega_{\widetilde{S}}(\log D_{\widetilde{S}})
	\]  
	represents the associated lower canonical logarithmic Higgs bundle. Moreover, ${\rm Gr}^{W}_{m}(K)$ is a coherent subsheaf of ${\rm Gr}^{W}_m(\widetilde{H}^{k}_{f,{\rm BM}})$ satisfying $\theta({\rm Gr}^{W}_{m}(K)) = 0$. By \cite[Theorem 1.4]{Brunebarbe2017} (see also \cite[Proposition 1.1]{Zuo2000} for the unipotent version), $-c_1({\rm Gr}^{W}_{m}(K))$ is pseudo-effective for every $m$. Therefore, $-c_1(K)$ is pseudo-effective as well, thanks to (\ref{align_Higgs_ker_gr}).
\end{proof}
\subsubsection{Top Hodge bundle}
Recall that $n = \dim X - \dim S$. Regarding the extension of the top-indexed Hodge bundle, there exists the following inclusion.
\begin{lem}\label{lem_Fujino}
	There exists an inclusion
	$$R^q\widetilde{f}_\ast(\sO_{\widetilde{X}}(K_{\widetilde{X}/\widetilde{S}}+\widetilde{D}^h))\subset \widetilde{H}^{n,q}_{f,{\rm BM}}\otimes_{\sO_{\widetilde{S}}}\sO_{\widetilde{S}}(D_{\widetilde{S}})$$ where $\widetilde{D}^h$ denotes the horizontal part of $\widetilde{D}$.  
\end{lem}
\begin{proof}
	According to equation (\ref{align_top_Hodge_complex}) and Theorem \ref{thm_log_Higgs_geo}, we derive the following:  
	\begin{align*}  
		R^q\widetilde{f}_\ast(\sO_{\widetilde{X}}(K_{\widetilde{X}/\widetilde{S}}+\widetilde{D}^h))|_{\overline{S}} &\simeq R^q\overline{f}_\ast(\sO_{\overline{X}}(K_{\overline{X}/\overline{S}}+\overline{D}^h)) \\  
		&\subset R^q\overline{f}_\ast(\sO_{\overline{X}}(K_{\overline{X}/\overline{S}}+\overline{D}-\overline{f}^\ast D_{\overline{S}}))\otimes_{\sO_{\overline{S}}}\sO_{\overline{S}}(D_{\overline{S}}) \\  
		&\simeq F^n(\widetilde{\cV}^{n+q}_{f,{\rm BM}})|_{\overline{S}}\otimes_{\sO_{\overline{S}}}\sO_{\overline{S}}(D_{\overline{S}})\\
		&\simeq\widetilde{H}^{n,q}_{f,{\rm BM}}|_{\overline{S}}\otimes_{\sO_{\overline{S}}}\sO_{\overline{S}}(D_{\overline{S}}).  
	\end{align*}  
	Given that $R^q\widetilde{f}_\ast(\sO_{\widetilde{X}}(K_{\widetilde{X}/\widetilde{S}}+\widetilde{D}^h))$ is locally free as stated in \cite[Theorem 7.3,(b)]{Fujino2014}, taking reflexive hulls on both sides completes the proof of the lemma.
\end{proof}
	\section{The Viehweg-Zuo Higgs sheaves}\label{section_VZ_Higgs}
	\subsection{Setting}\label{section_setting}	
	Throughout this section, we fix a projective morphism $f : X \to S$ from a simple normal crossing variety to a smooth variety, with relative dimension $n = \dim X - \dim S$. Let $\Delta$ be a $\bQ$-divisor on $X$ whose coefficients lie in $[0,1]$, such that $(X, \Delta)$ is a simple normal crossing pair. Assume that there exists a simple normal crossing divisor $D_f \subset S$ such that the restricted morphism $f^o := f|_{X^o} : (X^o, \Delta^o) \to S^o$ is a simple normal crossing family, where $S^o := S \setminus D_f$, $X^o := f^{-1}(S^o)$, and $\Delta^o := \Delta \cap X^o$. Furthermore, assume that ${\rm supp}(\Delta)$ does not contain any component of $f^{-1}(D_f)$.  
	
	We fix a line bundle $L$ on $S$, and a nonzero morphism  
	\begin{align}\label{align_s}  
	s_L : L^{\otimes k} \to f_\ast\left(\sO_X(kK_{X/S} + k\Delta)\right)  
	\end{align}  
	for some $k \geq 1$ such that $k\Delta$ is integral.
	\subsection{The main result: Viehweg-Zuo Higgs sheaves}\label{section_VZ_sheaf}	
	Define  
	$$ B = \omega_{X/S}(\lceil \Delta \rceil) \otimes_{\sO_X} f^{\ast}(L)^{-1}, $$  
	a line bundle on $X$. The map $s_L$ determines a non-zero section  
	$ s \in H^0(X, B^{\otimes k})$.	
	Let $\varpi : X_k \to X$ be the $k:1$ cyclic covering map branched along $\{s = 0\}$. Let $\mu : Z \to X_k$ be a resolution of singularities such that $Z$ is smooth (not necessarily connected) and $\Delta_Z := (\varpi \mu)^{-1}(\Delta \cup X_{\rm sing})$ is a simple normal crossing divisor. Define $g := f \varpi \mu : Z \to S$, and let $D_g \subset \widetilde{X}$ be a reduced closed algebraic subset containing $D_f$, such that $g : (Z, \Delta_Z) \to S$ is a log smooth family over $U := S \setminus D_g$. Let $Z^o := g^{-1}(U)$, and let $g^o := g|_{Z^o} : (Z^o, Z^o \cap \Delta_Z) \to U$ denote the restricted log smooth family.
	
	Let $V \subset S$ be a Zariski open subset such that ${\rm codim}_S(S \setminus V) \geq 2$ and $V \cap D_g$ is a simple normal crossing divisor on $V$. Then $R^n g^{o}_\ast(\bV_{(Z^o, Z^o \cap \Delta_Z), {\rm BM}})$ underlies a graded polarized variation of mixed Hodge structures $\bV^n_{g^o, {\rm BM}} = (\cV^n_{g^o, {\rm BM}}, \nabla, W_{\bullet}(\cV^n_{g^o, {\rm BM}}),F^\bullet_{{\rm BM}}(\cV^n_{g^o, {\rm BM}}))$ on $S'^o$ (Theorem \ref{thm_VMHS_BM}). By Theorem \ref{thm_VMHS_BM_admissible}, $\bV^n_{g^o, {\rm BM}}$ is admissible on $V$. Therefore, according to \S \ref{section_lHB_BM}, one has the lower canonical logarithmic Higgs bundle
	$$
	\left(\widetilde{H}^{k}_{g^o, {\rm BM}} = \bigoplus_{p=0}^n \widetilde{H}^{p, k-p}_{g^o, {\rm BM}}, \theta\right)
	$$  
	on $V$, associated with $\bV^n_{g^o, {\rm BM}}$ (Definition \ref{defn_LHB_ass_VHSBM}). 
\begin{thm}\label{thm_VZ_construction}  
Notations and assumptions as in \S \ref{section_setting} and \S \ref{section_VZ_sheaf}. Then the following statements hold:  
\begin{enumerate}  
	\item There exists a natural inclusion $L \otimes_{\sO_S} \sO_S(-D_f)|_V \subset \widetilde{H}^{n,0}_{g^o, {\rm BM}}$.  
	
	\item Let $$(\bigoplus_{p=0}^n L^p, \theta) \subset (\widetilde{H}^n_{g^o, {\rm BM}}, \theta)$$ be the meromorphic Higgs subsheaf generated by $L^0 := L \otimes_{\sO_S} \sO_S(-D_f)|_V$, where $L^p \subset \widetilde{H}^{n-p,p}_{g^o, {\rm BM}}$. Then the Higgs field  
	$$  
	\theta : L^p \to L^{p+1} \otimes_{\sO_V} \Omega_V(\log (D_g \cap V))  
	$$  
	is holomorphic over $V \setminus D_f$ and has at most logarithmic poles along $D_f$ for each $0 \leq p<n$, i.e.,  
	$$  
	\theta(L^p) \subset L^{p+1} \otimes_{\sO_V} \Omega_V(\log (D_f \cap V)),\quad 0 \leq p< n.  
	$$ 
\end{enumerate}
\end{thm}
	The proof will occupy the remainder of this section. It will be completed by constructing a logarithmic Higgs subsheaf $\bigoplus_{p+q=n} G^{p,q}$ of $\widetilde{H}^n_{g^o, {\rm BM}}$ that contains $L \otimes_{\sO_S} \sO_S(-D_f)|_V$, such that the Higgs field is holomorphic on $V \setminus D_f$.
	\subsection{Construction:}
	\subsubsection{Preliminary} We recall some facts and make some assumptions that will be used in the construction of the Viehweg-Zuo Higgs sheaf.
	\begin{enumerate}
		\item By taking a suitable semi-log resolution of $(X, \Delta + f^\ast(D_f))$ and $(Z, \Delta_Z + g^\ast(D_g))$, one can assume that $(X, \Delta + f^\ast(D_f))$ and $(Z, \Delta_Z + g^\ast(D_g))$ are semistable in codimension 1 over $S$. More precisely, there exists a Zariski open subset $V' \subset V$ such that ${\rm codim}_S(S \setminus V') \geq 2$, $D'_g:=V' \cap D_g$ and $D'_f:=V' \cap D_f$ are (not necessarily connected) smooth divisors on $V'$, and the restricted families  
		$$  
		f_1 = f|_{X_1} : (X_1 := f^{-1}(V'), \Delta_1 := \lceil \Delta + f^\ast(D_f) \rceil \cap X_1) \to (V', D'_f)  
		$$  
		and  
		$$  
		g_1 = g|_{Z_1} : (Z_1 := g^{-1}(V'), \Delta_{Z_1} := \lceil \Delta_Z + g^\ast(D_g) \rceil \cap Z_1) \to (V', D'_g)  
		$$  
		are semistable.  
		\item According to \S\ref{section_log_extension_VMHS_BM}, we have a logarithmic connection  
		$$  
		\nabla_{\rm GM} : R^n g_{1\ast}(\Omega^\bullet_{g_1, {\rm BM}}) \to R^n g_{1\ast}(\Omega^\bullet_{g_1, {\rm BM}}) \otimes_{\sO_{V'}} \Omega_{V'}(\log D'_g)  
		$$  
		and a filtration of subbundles $\{F^p R^n g_{1\ast}(\Omega^\bullet_{g_1, {\rm BM}})\}_{p=0,\dots,n}$ satisfying the Griffiths transversality condition:  
		$$  
		\nabla_{\rm GM}(F^p R^n g_{1\ast}(\Omega^\bullet_{g_1, {\rm BM}})) \subset F^{p-1} R^n g_{1\ast}(\Omega^\bullet_{g_1, {\rm BM}}) \otimes_{\sO_{V'}} \Omega_{V'}(\log D'_g), \quad p=0,\dots,n.  
		$$  
		Moreover, one has isomorphisms  
		$$  
		{\rm Gr}^p_F R^n g_{1\ast}(\Omega^\bullet_{g_1, {\rm BM}}) \simeq R^n g_{1\ast}({\rm Gr}^p_F \Omega^\bullet_{g_1, {\rm BM}}), \quad p=0,\dots,n.  
		$$
		Let  
		$$  
		H^{p,n-p}_{g_1, {\rm BM}} := {\rm Gr}^p_F R^n g_{1\ast}(\Omega^\bullet_{g_1, {\rm BM}}), \quad p=0,\dots,n,  
		$$  
		let  
		$$  
		H^n_{g_1, {\rm BM}} := \bigoplus_{p=0}^n H^{p,n-p}_{g_1, {\rm BM}},  
		$$  
		and let  
		$$  
		\theta_{g_1}:= {\rm Gr}_F(\nabla) : H^n_{g_1, {\rm BM}} \to H^n_{g_1, {\rm BM}} \otimes_{\sO_{V'}} \Omega_{V'}(\log D'_g).  
		$$  	
		Then we obtain a logarithmic Higgs bundles $(H^n_{g_1, {\rm BM}}, \theta)$ satisfying  
		$$  
		\theta_{g_1}(H^{p,n-p}_{g_1, {\rm BM}}) \subset H^{p-1,n-p+1}_{g_1, {\rm BM}} \otimes_{\sO_{V'}} \Omega_{V'}(\log D'_g), \quad p=0,\dots,n.  
		$$
		According to Theorem \ref{thm_log_Higgs_geo}, there exists an isomorphism of logarithmic Higgs bundles 
		\begin{align}\label{align_g_1BM_to_extg0}
		\eta:(H^{k}_{g_1,{\rm BM}},\theta_{g_1})\to (\widetilde{H}^k_{g^o,{\rm BM}},\theta)|_{V'}.
		\end{align}
		which sends $H^{p,n-p}_{g_1, {\rm BM}}$ to $\widetilde{H}^{p,n-p}_{g^o, {\rm BM}}|_{V'}$ for every $p=0,\dots,n$. 
	\end{enumerate}
	\subsubsection{Construction}
    Let $\varphi_1 := \varpi\mu|_{Z_1} : Z_1 \to X_1$.  Since $X_k$ is embedded into the total space of the line bundle $B$, the pullback $\varpi^\ast(B)$ admits a tautological section induced by $s$. This implies the existence of a non-zero morphism  
    \begin{align}\label{align_B-1_to_O}  
    \times\varphi_1^{\ast}(s^{\frac{1}{k}}):\varphi_1^\ast(B^{-1}) \to \sO_{Z_1}.  
    \end{align}  
    Recall that both families  
    $$  
    f_1 : (X_1, \Delta_1) \to (V', D'_f) \quad \text{and} \quad g_1 : (Z_1, \Delta_{Z_1}) \to (V', D'_g)  
    $$  
    are semistable. By (\ref{align_log_Hodge_fil}), one has the Hodge filtrations (of complexes) $F^p \Omega^\bullet_{f_1, {\rm BM}}$ and $F^p \Omega^\bullet_{g_1, {\rm BM}}$, for $0 \leq p \leq n$, together with the natural morphisms  
    $$  
    \varphi_1^\ast(F^p \Omega^\bullet_{f_1, {\rm BM}}) \to F^p \Omega^\bullet_{g_1, {\rm BM}}, \quad 0 \leq p \leq n.  
    $$
    Combining these morphisms with (\ref{align_B-1_to_O}), we obtain natural morphisms  
    $$  
    \varphi_1^\ast(F^p \Omega^\bullet_{f_1, {\rm BM}} \otimes_{\sO_{X_1}} B^{-1}) \to F^p \Omega^\bullet_{g_1, {\rm BM}}, \quad 0 \leq p \leq n.  
    $$  
    These induce natural morphisms  
    \begin{align}\label{align_X_1_Z_1}
    F^p \Omega^\bullet_{f_1, {\rm BM}} \otimes_{\sO_{X_1}} B^{-1} \to R\varphi_{1\ast}(F^p \Omega^\bullet_{g_1, {\rm BM}}), \quad 0 \leq p \leq n.  
    \end{align}  
    This further induces a morphism  
    $$  
    \iota_{V'} : R^{p+q} f_{1\ast}({\rm Gr}_F^p \Omega^\bullet_{f_1, {\rm BM}} \otimes_{\sO_{X_1}} B^{-1}) \to R^{p+q} g_{1\ast}({\rm Gr}_F^p \Omega^\bullet_{g_1, {\rm BM}})  
    $$  
    for every $p, q$.
    
    As in \S \ref{section_log_GMconnection_VBM}, we have the Koszul filtrations $\{K^p_{f_1}\}_{p\geq 0}$ and $\{K^p_{g_1}\}_{p\geq 0}$ of $\Omega^\bullet_{(X_1, \Delta_1), {\rm BM}}$ and $\Omega^\bullet_{(Z_1, \Delta_{Z_1}), {\rm BM}}$, respectively.     
    The morphisms (\ref{align_X_1_Z_1}) induce a morphism between distinguished triangles:  
    $$  
    {\small\xymatrix{  
    	f_1^{-1} \Omega_{V'}(\log D'_f) \otimes_{f_1^{-1} \sO_{V'}} {\rm Gr}_F^{p-1} \Omega^\bullet_{f_1, {\rm BM}} \otimes_{\sO_{X_1}} B^{-1}[-1] \ar[r] \ar[d] &  R\varphi_{1\ast} \left(g_1^{-1} \Omega_{V'}(\log D'_g) \otimes_{g_1^{-1} \sO_{V'}} {\rm Gr}_F^{p-1} \Omega^\bullet_{g_1, {\rm BM}}[-1]\right) \ar[d] \\  
    	\frac{\sigma^{\geq p-1}(K^0_{f_1}/K^2_{f_1})}{\sigma^{\geq p+1}(K^0_{f_1}/K^2_{f_1})} \otimes_{\sO_{X_1}} B^{-1} \ar[r]\ar[d] & R\varphi_{1\ast} \left(\frac{\sigma^{\geq p-1}(K^0_{g_1}/K^2_{g_1})}{\sigma^{\geq p+1}(K^0_{g_1}/K^2_{g_1})}\right) \ar[d]\\
    	{\rm Gr}_F^{p} \Omega^\bullet_{f_1, {\rm BM}} \otimes_{\sO_{X_1}} B^{-1} \ar[r]\ar[d] & R\varphi_{1\ast}({\rm Gr}_F^{p} \Omega^\bullet_{g_1, {\rm BM}}) \ar[d]\\
    	&
    }}  
    $$     
    Taking the higher direct images, one obtains a commutative diagram:  
    $$  
    \xymatrix{  
    	R^{p+q} f_{1\ast}({\rm Gr}_F^{p} \Omega^\bullet_{f_1, {\rm BM}} \otimes_{\sO_{X_1}} B^{-1}) \ar[r]^-{\vartheta} \ar[d]^{\iota_{V'}} & R^{p+q} f_{1\ast}({\rm Gr}_F^{p-1} \Omega^\bullet_{f_1, {\rm BM}} \otimes_{\sO_{X_1}} B^{-1}) \otimes_{\sO_{V'}} \Omega_{V'}(\log D'_f) \ar[d]^{\iota_{V'} \otimes {\rm Id}} \\  
    	R^{p+q} g_{1\ast}({\rm Gr}_F^{p} \Omega^\bullet_{g_1, {\rm BM}}) \ar[r]^-{\theta_{g_1}} & R^{p+q} g_{1\ast}({\rm Gr}_F^{p-1} \Omega^\bullet_{g_1, {\rm BM}}) \otimes_{\sO_{V'}} \Omega_{V'}(\log D'_g),  
    }  
    $$  
    where $\vartheta$ is the coboundary map.
    Combining this commutative diagram with the isomorphism (\ref{align_g_1BM_to_extg0}), we obtain a commutative diagram: 
    \begin{align}\label{align_f1_to_extg0}
    \xymatrix{
    	R^{p+q}f_{1\ast}({\rm Gr}_F^{p}\Omega^\bullet_{f_1,{\rm BM}}\otimes_{\sO_{X_1}} B^{-1})\ar[r]^-{\vartheta}\ar[d]^{\eta\iota_{V'}} & R^{p+q}f_{1\ast}({\rm Gr}_F^{p-1}\Omega^\bullet_{f_1,{\rm BM}}\otimes_{\sO_{X_1}} B^{-1})\otimes_{\sO_{V'}}\Omega_{V'}(\log D'_f)\ar[d]^{\eta\iota_{V'}\otimes{\rm Id}}\\
    	\widetilde{H}^{p,q}_{g^o,{\rm BM}}\ar[r]^-{\theta} & \widetilde{H}^{p-1,q+1}_{g^o,{\rm BM}}\otimes_{\sO_{V'}}\Omega_{V'}(\log D_g\cap V')\\
    }.
    \end{align}
	For every $p,q$, we define 
	\begin{align*}
	G^{p,q}:={\rm Im}\left(\eta\iota_{V'}:R^{p+q}f_{1\ast}({\rm Gr}_F^{p}\Omega^\bullet_{f_1,{\rm BM}}\otimes_{\sO_{X_1}} B^{-1})\to \widetilde{H}^{p,q}_{g^o,{\rm BM}}\right).
	\end{align*}
\subsubsection{Proof of Theorem \ref{thm_VZ_construction}}
Theorem \ref{thm_VZ_construction} is the consequence of the following two lemmas.
	\begin{lem}\label{lem_thetaG}
		\begin{align}\label{align_ZV_is_log}
		\theta(G^{p,q})\subset G^{p-1,q+1}\otimes_{\sO_{V'}}\Omega_{V'}(\log D'_f).
		\end{align}
	\end{lem}
	\begin{proof}
		The inclusion (\ref{align_ZV_is_log}) holds over $V'\backslash D'_f$ due to the commutative diagram (\ref{align_f1_to_extg0}). Let $x \in D'_f$, and let $z_1, \dots, z_d$ denote holomorphic local coordinates at $x$ such that $D'_f = \{z_1 = 0\}$. Define  
		\begin{align*}  
			\xi_i := \begin{cases}  
				z_i \frac{\partial}{\partial z_i}, & i = 1, \\  
				\frac{\partial}{\partial z_i}, & i = 2, \dots, d.  
			\end{cases}  
		\end{align*}  
		Let $v \in R^{p+q} f_{1\ast}({\rm Gr}_F^{p}\Omega^\bullet_{f_1,{\rm BM}} \otimes_{\sO_{X_1}} B^{-1})$. From (\ref{align_f1_to_extg0}), it follows that  
		\begin{align*}  
			\theta(\xi_i)(\eta \iota_{V'}(v)) = \eta \iota_{V'}\left(\vartheta(\xi_i)(v)\right) \in {\rm Im}(\eta \iota_{V'}), \quad \forall i = 1, \dots, d.  
		\end{align*}  
		This demonstrates that  
		\begin{align*}  
			\theta(\xi_i)(G^{p,q}) \subset G^{p-1,q+1}, \quad \forall i = 1, \dots, d.  
		\end{align*}  
		Hence, the lemma is proved.
	\end{proof}
	\begin{lem}
		There is a natural inclusion $L\otimes_{\sO_{S}}\sO_{S}(-D_f)|_{V'}\subset G^{n,0}$.
	\end{lem}
	\begin{proof}
		According to (\ref{align_top_Hodge_complex}), we have  
		$$
		\omega_{X_1/V'}(\lceil\Delta|_{X_1}\rceil - f_1^\ast(D'_f)) \subset \omega_{X_1/V'}(\lceil\Delta_1 - f_1^\ast(D'_f)\rceil) \simeq F^n\Omega^\bullet_{f_1,{\rm BM}}[n].
		$$  
		Consider the natural maps:
		\begin{align*}
		\beta:L\otimes_{\sO_{S}}\sO_{S}(-D_f)|_{V'}\to f_{1\ast}(f_1^\ast(L|_{V'}\otimes_{\sO_{V'}}\sO_{V'}(-D'_f)))&\simeq f_{1\ast}(B^{-1}\otimes_{\sO_{X_1}}\omega_{X_1/V'}(\lceil\Delta|_{X_1}\rceil-f_1^\ast(D'_f)))\\\nonumber
		&\subset R^nf_{1\ast}(B^{-1}\otimes_{\sO_{X_1}} F^n\Omega^\bullet_{f_1,{\rm BM}}),
		\end{align*}
		and 
		\begin{align*}
		\alpha:L\otimes_{\sO_{S}}\sO_{S}(-D_f)|_{V'}\to f_{1\ast}(f_1^\ast(L|_{V'}\otimes_{\sO_{V'}}\sO_{V'}(-D'_f)))&\subset f_{1\ast}(B^{-1}\otimes_{\sO_{X_1}}\omega_{X_1/V'}(\lceil\Delta|_{X_1}\rceil))\\\nonumber
		&\to g_{1\ast}(\omega_{Z_1/V'}(\lceil\Delta_Z|_{Z_1}\rceil))\subset \widetilde{H}^{n,0}_{g^o,{\rm BM}}|_{V'},
		\end{align*}
		where the last inclusion follows from Lemma \ref{lem_Fujino}, along with the fact that $\Delta_Z$ is $g$-horizontal. By construction, it holds that $\iota_{V'} \circ \beta = \alpha$. Consequently, ${\rm Im}(\alpha) \subset G^{n,0}$.
		
		Now consider the composition of maps:  
		\begin{align*}  
			\alpha|_{S^o}: L|_{S^o} \to f^o_{\ast}(f^{o\ast} L|_{S^o}) \simeq f^o_{\ast}(B^{-1} \otimes_{\sO_{X^o}} \omega_{X^o/S^o}(\lceil \Delta|_{X^o}\rceil)) \stackrel{\iota_{V'}|_{S^o}}{\to} H^{n,0}_{g^o,{\rm BM}}.  
		\end{align*}  
		Observe that $\alpha$ is the multiplication by $s^{\frac{1}{k}}$, combined with the pullback of differential forms $\omega_{X/S} \to (\varpi\mu)_\ast(\omega_{Z/X})$. For each local generator $e$ of $L$, $\alpha(e)$ must be nonzero on some component of $Z$. Consequently, $\alpha|_{S^o}$ is a nonzero map.  
		Since $\widetilde{H}^{n,0}_{g^o,{\rm BM}}$ is locally free, it follows that $\alpha$ is injective. Therefore, we obtain the claimed injective morphism $L \otimes_{\sO_S} \sO_S(-D_f)|_{V'} \subset G^{n,0}$ as stated in the lemma.	
	\end{proof}
    Now we proceed to finalize the proof of Theorem \ref{thm_VZ_construction}. To summarize, we have constructed an injective morphism  
    $$
    \alpha: L \otimes_{\sO_S} \sO_S(-D_f)|_{V'} \to \widetilde{H}^{n,0}_{g^o,{\rm BM}}|_{V'}
    $$  
    such that the image ${\rm Im}(\alpha)$ is contained in a Higgs subsheaf $(\oplus_{p=0}^n G^{p,n-p}, \theta)$ of $(\widetilde{H}^n_{g^o,{\rm BM}}, \theta)|_{V'}$, satisfying  
    \begin{align*}  
    	\theta(G^{p,n-p}) \subset G^{p-1,n-p+1} \otimes \Omega_{V'}(\log D'_f)
    \end{align*}  
    for every $p=0,\cdots,n$.
    Since $\widetilde{H}^n_{g^o,{\rm BM}}$ is locally free and ${\rm codim}_V(V \setminus V') \geq 2$, the morphism $\alpha$ extends to an injective morphism  
    $$
    \alpha: L \otimes \sO_S(-D_f)|_V \to \widetilde{H}^{n,0}_{g^o,{\rm BM}}
    $$  
    that satisfies the assertion of the theorem.
	\section{Admissible families of stable minimal models}\label{section_boundedness}	
	\subsection{Stable minimal models and their moduli}\label{section_moduli}
We now review the main results from \cite{Birkar2022} that will be utilized in the subsequent sections. For the purposes of this article, all schemes are defined over ${\rm Spec}(\bC)$. A \emph{stable minimal model} is a triple $(X, B), A$, where $X$ is a reduced, connected, projective scheme of finite type over ${\rm Spec}(\bC)$, and $A, B \geq 0$ are $\bQ$-divisors satisfying the following conditions:  
\begin{itemize}  
	\item $(X, B)$ is a projective, connected slc (semi-log-canonical) pair,  
	\item $K_X + B$ is semi-ample,  
	\item $K_X + B + tA$ is ample for some $t > 0$, and  
	\item $(X, B + tA)$ is slc for some $t > 0$.  
\end{itemize}  
	Let  
	$$
	d \in \bN, \, c \in \bQ^{\geq 0}, \, \Gamma \subset \bQ^{>0} \text{ a finite set, and } \sigma \in \bQ[t].
	$$  
	A $(d, \Phi_c, \Gamma, \sigma)$-stable minimal model is a stable minimal model $(X, B), A$ satisfying the following conditions:  
	\begin{itemize}  
		\item $\dim X = d$,  
		\item the coefficients of $A$ and $B$ belong to $c \bZ^{\geq 0}$,  
		\item ${\rm vol}(A|_F) \in \Gamma$, where $F$ is any general fiber of the fibration $f: X \to Z$ determined by $K_X + B$, and  
		\item ${\rm vol}(K_X + B + tA) = \sigma(t)$ for $0 \leq t \ll 1$.  
	\end{itemize}  
Let $S$ be a reduced scheme over ${\rm Spec}(\bC)$. A family of $(d, \Phi_c, \Gamma, \sigma)$-stable minimal models over $S$ consists of a projective morphism $X \to S$ of schemes and $\bQ$-divisors $A$ and $B$ on $X$, satisfying the following conditions:  
\begin{itemize}  
	\item $(X, B + tA) \to S$ is a locally stable family (i.e., $K_{X/S} + B + tA$ is $\bQ$-Cartier) for every sufficiently small rational number $t \geq 0$,  
	\item $A = cN$, $B = cD$, where $N, D \geq 0$ are relative Mumford divisors, and  
	\item $(X_s, B_s), A_s$ is a $(d, \Phi_c, \Gamma, \sigma)$-stable minimal model for each point $s \in S$.  
\end{itemize} 
Let ${\rm Sch}_{\bC}^{\rm red}$ denote the category of reduced schemes defined over ${\rm Spec}(\bC)$. Define  
$$
\sM^{\rm red}_{\rm slc}(d, \Phi_c, \Gamma, \sigma): S \mapsto \{\text{families of } (d, \Phi_c, \Gamma, \sigma)\text{-stable minimal models over } S\},
$$  
a functor of groupoids over ${\rm Sch}_{\bC}^{\rm red}$.  

\begin{thm}[Birkar \cite{Birkar2022}] \label{thm_moduli_stable_var}  
	There exists a proper Deligne-Mumford stack $\sM_{\rm slc}(d, \Phi_c, \Gamma, \sigma)$ over ${\rm Spec}(\bC)$ such that the following properties hold:  
	\begin{itemize}  
		\item $\sM_{\rm slc}(d, \Phi_c, \Gamma, \sigma)|_{{\rm Sch}_{\bC}^{\rm red}} = \sM^{\rm red}_{\rm slc}(d, \Phi_c, \Gamma, \sigma)$ as functors of groupoids.  
		\item $\sM_{\rm slc}(d, \Phi_c, \Gamma, \sigma)$ admits a projective good coarse moduli space $M_{\rm slc}(d, \Phi_c, \Gamma, \sigma)$.  
	\end{itemize}  
\end{thm}  
\begin{proof}  
	See the proof of \cite[Theorem 1.14]{Birkar2022}. Using the notations in \cite[\S 10.7]{Birkar2022}, we have  
	$$
	\sM_{\rm slc}(d, \Phi_c, \Gamma, \sigma) = \left[M_{\rm slc}^e(d, \Phi_c, \Gamma, \sigma, a, r, \bP^n)/{\rm PGL}_{n+1}(\bC)\right],
	$$  
	where the right-hand side denotes the stacky quotient.  
\end{proof}
	\subsection{Polarization on $M_{\rm slc}(d,\Phi_c,\Gamma,\sigma)$}\label{section_polarization_moduli}
In this section, we consider certain natural ample $\bQ$-line bundles on $M_{\rm slc}(d, \Phi_c, \Gamma, \sigma)$. Their constructions are implicitly described in the proof of \cite[Theorem 1.14]{Birkar2022}, relying on Koll\'ar's ampleness criterion \cite{Kollar1990}.  

Fix the data $d, \Phi_c, \Gamma, \sigma$. Since $\sM_{\rm slc}(d, \Phi_c, \Gamma, \sigma)$ is of finite type, there exist constants  
$$
(a, r, j) \in \bQ^{\geq 0} \times (\bZ^{>0})^2,
$$  
depending only on $d, \Phi_c, \Gamma, \sigma$, such that every $(d, \Phi_c, \Gamma, \sigma)$-stable minimal model $(X, B), A$ satisfies the following conditions (cf. \cite[Lemma 10.2]{Birkar2022}):  
\begin{itemize}  
	\item $(X, B + aA)$ is an slc pair,  
	\item $r(K_X + B + aA)$ is a very ample integral Cartier divisor with  
	$$
	H^i(X, \sO_X(k r(K_X + B + aA))) = 0, \quad \forall i > 0, \forall k > 0,  
	$$  
	\item the embedding $X \hookrightarrow \bP(H^0(X, r(K_X + B + aA)))$ is defined by equations of degree $\leq j$, and  
	\item the multiplication map  
	$$
	S^j(H^0(X, \sO_X(r(K_X + B + aA)))) \to H^0(X, \sO_X(j r(K_X + B + aA)))
	$$  
	is surjective.  
\end{itemize}
	\begin{defn}
		A tuple $(a, r, j) \in \bQ^{\geq 0} \times (\bZ^{>0})^2$ that satisfies the conditions above is referred to as a \emph{$(d, \Phi_c, \Gamma, \sigma)$-polarization datum}.
	\end{defn}
	Let $(a, r, j) \in \bQ^{\geq 0} \times (\bZ^{>0})^2$ be a $(d, \Phi_c, \Gamma, \sigma)$-polarization datum. Let $(X, B), A \to S$ be a family of $(d, \Phi_c, \Gamma, \sigma)$-stable minimal models. Then $f_\ast(r(K_{X/S} + B + aA))$ is locally free and commutes with arbitrary base changes. Therefore, the assignment  
	$$
	f: (X, B), A \to S \in \sM_{\rm slc}(d, \Phi_c, \Gamma, \sigma)(S) \mapsto f_\ast(r(K_{X/S} + B + aA))
	$$  
	defines a locally free coherent sheaf on the stack $\sM_{\rm slc}(d, \Phi_c, \Gamma, \sigma)$, denoted by $\Lambda_{a,r}$. Let $\lambda_{a,r} := \det(\Lambda_{a,r})$. Since $\sM_{\rm slc}(d, \Phi_c, \Gamma, \sigma)$ is Deligne-Mumford, some power $\lambda_{a,r}^{\otimes k}$ descends to a line bundle on $M_{\rm slc}(d, \Phi_c, \Gamma, \sigma)$. For this reason, we regard $\lambda_{a,r}$ as a $\bQ$-line bundle on $M_{\rm slc}(d, \Phi_c, \Gamma, \sigma)$.
	\begin{prop}\label{prop_ample_line_bundle_moduli}  
		Let $(a, r, j) \in \bQ^{\geq 0} \times (\bZ^{>0})^2$ be a $(d, \Phi_c, \Gamma, \sigma)$-polarization datum. Then $\lambda_{a,r}$ is ample on $M_{\rm slc}(d, \Phi_c, \Gamma, \sigma)$.  
	\end{prop}  	
	\begin{proof}  
		By the same arguments as in \cite[\S 2.9]{Kollar1990}. It suffices to show that $f_\ast(r(K_{X/S} + B + aA))$ is nef when $S$ is a smooth projective curve. This was established by Fujino \cite{Fujino2018} and Kov\'acs-Patakfalvi \cite{Kovacs2017}.  
	\end{proof}  
	\subsection{Fiber product of stable families}\label{section_ssred}
	Following Koll\'ar \cite{Kollar2023}, a flat family of slc pairs $(X, \Delta) \to S$ over a smooth base $S$ is referred to as a \emph{locally stable family} if $K_{X/S} + \Delta$ is $\bQ$-Cartier.
	
	Let $S$ be a smooth variety, and let $f:(X, \Delta) \to S$ be a locally stable family of slc pairs. Denote by $X^{[r]}_S$ the $r$-fold fiber product $X \times_S X \times_S \cdots \times_S X$, and let $f^{[r]}: X^{[r]}_S \to S$ be the projection map. Define  
	$$
	\Delta^{[r]}_S := \sum_{i=1}^r p_i^\ast(\Delta),
	$$  
	where $p_i: X^{[r]}_S \to X$ is the projection onto the $i$-th component. Then $f^{[r]}:(X^{[r]}_S, \Delta^{[r]}_S) \to S$ is also a locally stable family of slc pairs (\cite[Corollary 4.3]{WeiWu2023}). According to \cite[Theorem 4.54]{Kollar2023}, $(X^{[r]}_S, \Delta^{[r]}_S)$ is an slc pair.
	
	Let $\tau: (X^{(r)}_S, \Delta^{(r)}_S) \to (X^{[r]}_S, \Delta^{[r]}_S)$ be a semi-log resolution, where $\Delta^{(r)}_S$ is a simple normal crossing divisor determined by  
	\begin{align}\label{align_slc_adjunction}  
		\tau^\ast(K_{X^{[r]}_S} + \Delta^{[r]}_S) = K_{X^{(r)}_S} + \Delta^{(r)}_S - E,  
	\end{align}  
	with $E$ being an effective $\tau$-exceptional divisor that shares no common components with $\Delta^{(r)}_S$. 
	\begin{lem}\label{lem_mild_pushforward}
		Let $k \geq 1$ such that $kK_{X/S} + k\Delta$ is integral. Then the following statements hold:  
		\begin{enumerate}
			\item $\tau_{\ast}\left(\sO_{X^{(r)}_S}(kK_{X^{(r)}_S/S}+k\Delta^{(r)}_S)\right)\simeq\sO_{X^{[r]}_S}(kK_{X^{[r]}_S/S}+k\Delta^{[r]}_S)$.
			\item $f^{[r]}_\ast\left(\sO_{X^{[r]}_S}(kK_{X^{[r]}_S/S}+k\Delta^{[r]}_S)\right)$ is a reflexive sheaf.
			\item $f^{[r]}_\ast\left(\sO_{X^{[r]}_S}(kK_{X^{[r]}_S/S}+k\Delta^{[r]}_S)\right)\simeq (f_\ast\left(\sO_X(kK_{X/S}+k\Delta)\right)^{\otimes r})^{\vee\vee}$.
		\end{enumerate}
	\end{lem}
	\begin{proof}
		Observe that $\tau_\ast(\sO_{X^{(r)}_S}(kE))\simeq \sO_{X^{[r]}_S}$, as $X^{[r]}_S$ satisfies the $S_2$ condition and $E$ is $\tau$-exceptional with $\mathrm{codim}_{X^{[r]}_S}(\tau(E)) \geq 2$. Consequently, the first assertion follows from (\ref{align_slc_adjunction}) and the projection formula.
		
		For the second assertion, it suffices to show that any section of $f^{[r]}_\ast\left(\sO_{X^{[r]}_S}(kK_{X^{[r]}_S/S} + k\Delta^{[r]}_S)\right)$ extends across an arbitrary locus of codimension $\geq 2$. Let $U \subset S$ be an open subset, and let $Z \subset U$ be a Zariski closed subset of codimension $\geq 2$. Let  
		$$
		s \in \Gamma\left(U \setminus Z, f^{[r]}_\ast\left(\sO_{X^{[r]}_S}(kK_{X^{[r]}_S/S} + k\Delta^{[r]}_S)\right)\right) = \Gamma\left((f^{[r]})^{-1}(U \setminus Z), \sO_{X^{[r]}_S}(kK_{X^{[r]}_S/S} + k\Delta^{[r]}_S)\right).
		$$  
		Since $f$ is flat, so is $f^{[r]}$. Hence $(f^{[r]})^{-1}(Z)$ has codimension $\geq 2$ in $(f^{[r]})^{-1}(U)$. Since $X^{[r]}_S$ satisfies the $S_2$ condition and $\sO_{X^{[r]}_S}(kK_{X^{[r]}_S/S} + k\Delta^{[r]}_S)$ is invertible, there exists  
		$$
		\widetilde{s} \in \Gamma\left(U, f^{[r]}_\ast\left(\sO_{X^{[r]}_S}(kK_{X^{[r]}_S/S} + k\Delta^{[r]}_S)\right)\right) = \Gamma\left((f^{[r]})^{-1}(U), \sO_{X^{[r]}_S}(kK_{X^{[r]}_S/S} + k\Delta^{[r]}_S)\right)
		$$  
		that extends $s$. This proves Assertion (2).
		
		Finally, we prove the last assertion. Observe that  
		$$
		\sO_{X^{[r]}_S}(kK_{X^{[r]}_S/S} + k\Delta^{[r]}_S) \simeq \bigotimes_{i=1}^r p_i^\ast \sO_X(kK_{X/S} + k\Delta),
		$$  
		where $p_i: X^{[r]}_S \to X$ denotes the projection onto the $i$-th component. Let $U \subset S$ be the largest open subset such that both $f^{[r]}_\ast\left(\sO_{X^{[r]}_S}(kK_{X^{[r]}_S/S} + k\Delta^{[r]}_S)\right)$ and $f_\ast\left(\sO_X(kK_{X/S} + k\Delta)\right)$ are locally free on $U$. Since the relevant sheaves are torsion-free, the complement $S \setminus U$ has codimension $\geq 2$. By the flat base change theorem, we obtain  
		$$
		f^{[r]}_\ast\left(\sO_{X^{[r]}_S}(kK_{X^{[r]}_S/S} + k\Delta^{[r]}_S)\right)\big|_U \simeq f_\ast\left(\sO_X(kK_{X/S} + k\Delta)\right)^{\otimes r}\big|_U.
		$$  
		Since $f^{[r]}_\ast\left(\sO_{X^{[r]}_S}(kK_{X^{[r]}_S/S} + k\Delta^{[r]}_S)\right)$ and $\left(f_\ast\left(\sO_X(kK_{X/S} + k\Delta)\right)^{\otimes r}\right)^{\vee\vee}$ are reflexive, this completes the proof of Assertion (3).
	\end{proof}
	We will also need the following lemma (originated to Viehweg \cite{Viehweg1983}) in the sequel.
	\begin{lem}\label{lem_Viehweg_inclusion}
		Let $g: Y \to S$ be a projective surjective morphism from a simple normal crossing variety $Y$ to a smooth variety $S$. Let $\tau: S' \to S$ be a flat projective surjective morphism from a smooth variety $S'$. Let $\Delta \geq 0$ be an integral divisor on $Y$. Consider the following commutative diagram:  
		\begin{align}  
			\xymatrix{  
				Y \ar[d]^g & Y'' \ar[l]_{\rho'} \ar[d]^{g''} & Y' \ar[l]_{\rho} \ar[dl]^{g'} \\  
				S & S' \ar[l]_{\tau} &  
			},  
		\end{align}  
		where $g'': Y'' \to S'$ is the base change of $g$, and $\rho: Y' \to Y''$ is a semi-log resolution of singularities. Then, for every $r \geq 1$, there exists a natural inclusion  
		$$  
		g'_\ast\left(\sO_{Y'}(rK_{Y'/S'} + (\rho'\rho)^\ast\Delta)\right) \subset \tau^\ast g_\ast\left(\sO_Y(rK_{Y/S} + \Delta)\right).  
		$$  
	\end{lem}
	\begin{proof}
		{\bf Case I: $Y$ is irreducible.}
		Since $g$ is a Gorenstein morphism, we have $\rho'^\ast(K_{Y/S}) \simeq K_{Y''/S'}$, where $K_{Y''/S'}$ denotes the Cartier divisor associated with the invertible relative dualizing sheaf of $Y''/S'$. Consequently, it follows that $\rho'^\ast(rK_{Y/S} + \Delta) \simeq rK_{Y''/S'} + \rho'^\ast(\Delta)$. By invoking the flat base change theorem, we deduce the following isomorphism:  
		\begin{align}\label{align_Vieh_1}
			\tau^\ast g_\ast\left(\sO_Y(rK_{Y/S}+\Delta)\right) \simeq g''_\ast\left(\sO_{Y''}(rK_{Y''/S'}+ \rho'^\ast(\Delta))\right).
		\end{align}  
		According to \cite[Lemma 3.1.19]{Fujino2020}, the trace map induces an injective morphism:  
		$$
		\rho_\ast\left(\sO_{Y'}(rK_{Y'/S'})\right) \to \sO_{Y''}(rK_{Y''/S'}).
		$$  
		This further implies an injective morphism:  
		$$
		\rho_\ast\left(\sO_{Y'}(rK_{Y'/S'}+(\rho'\rho)^\ast(\Delta))\right) \to \sO_{Y''}(rK_{Y''/S'}+\rho'^\ast(\Delta)).
		$$  
		Consequently, there exists a natural injective morphism:  
		\begin{align*}
			g'_\ast\left(\sO_{Y'}(rK_{Y'/S'}+(\rho'\rho)^\ast(\Delta))\right) \to g''_\ast\left(\sO_{Y''}(rK_{Y''/S'}+ \rho'^\ast(\Delta))\right).
		\end{align*}  
		Combining this result with (\ref{align_Vieh_1}), we obtain the desired inclusion.
		
		{\bf Case II: $Y$ is reducible.} 
	    Let $\varrho: Y^v \to Y$ denote the normalization, and let $D^v \subset Y^v$ represent the conductor divisor. Then $Y^v$ is smooth (possibly with many components), $D^v$ is a simple normal crossing divisor on $Y^v$, and there exists an inclusion:  
	    \begin{align*}
	    	g_\ast\left(\sO_Y(rK_{Y/S} + \Delta)\right) \subseteq (g\varrho)_\ast\left(\sO_{Y^v}(rK_{Y^v/S} + rD^v + \varrho^*(\Delta))\right).
	    \end{align*}  
	    Let $\eta: \widetilde{D} \to D^v$ denote the normalization of $D^v$, and let  
	    \begin{align*}
	    	{\rm Res}_{\widetilde{D}}: \sO_{Y^v}(rK_{Y^v/S} + rD^v + \varrho^*(\Delta)) \to \sO_{\widetilde{D}}(rK_{\widetilde{D}/S} + \eta^\ast(\varrho^*(\Delta)|_{D^v}))
	    \end{align*}  
	    represent the residue map (cf. \cite[\S 4.1]{Kollar2013}). Let $\alpha: \widetilde{D} \to \widetilde{D}$ denote the involution over $D^v$ (cf. \cite[\S 5.2]{Kollar2013}). Consequently, one obtains a morphism:  
	    \begin{align*}
	    	(g\varrho)_\ast({\rm Res}_{\widetilde{D}} - (-1)^r\alpha^\ast{\rm Res}_{\widetilde{D}}): &
	    	(g\varrho)_\ast\left(\sO_{Y^v}(rK_{Y^v/S} + rD^v + \varrho^*(\Delta))\right) \\\nonumber
	    	\to & (g\varrho\eta)_\ast\left(\sO_{\widetilde{D}}(rK_{\widetilde{D}/S} + \eta^\ast(\varrho^*(\Delta)|_{D^v}))\right).
	    \end{align*}  
	    According to \cite[Proposition 5.8]{Kollar2013}, it follows that:  
	    \begin{align}\label{align_Viehweg_normalization1}
	    	g_\ast\left(\sO_Y(rK_{Y/S} + \Delta)\right) = {\rm Ker}\left((g\varrho)_\ast({\rm Res}_{\widetilde{D}} - (-1)^r\alpha^\ast{\rm Res}_{\widetilde{D}})\right).
	    \end{align}
    
	    Similarly, let $\varrho': Y'^v \to Y'$ denote the normalization, and let $D'^v \subset Y'^v$ represent the conductor divisor. Let $\eta': \widetilde{D'} \to D'^v$ denote the normalization of $D'^v$. Let $\alpha': \widetilde{D'} \to \widetilde{D'}$ denote the involution over $D'^v$. Consequently, one obtains a morphism:  
	    \begin{align*}
	    	(g'\varrho')_\ast({\rm Res}_{\widetilde{D'}}-(-1)^r\alpha'^\ast{\rm Res}_{\widetilde{D'}}):&
	    	(g'\varrho')_\ast\left(\sO_{Y'^v}(rK_{Y'^v/S'}+rD'^v+(\rho'\rho\varrho')^*(\Delta))\right)\\\nonumber
	    	\to& (g'\varrho'\eta')_\ast\left(\sO_{\widetilde{D'}}(rK_{\widetilde{D'}/S'}+\eta'^\ast((\rho'\rho\varrho')^*(\Delta)|_{D'^v}))\right).
	    \end{align*}
	    It follows that:  
	    \begin{align}\label{align_Viehweg_normalization2}
	    	g'_\ast\left(\sO_{Y'}(rK_{Y'/S'} + (\rho'\rho)^\ast\Delta)\right) = {\rm Ker}\left((g'\varrho')_\ast({\rm Res}_{\widetilde{D'}} - (-1)^r\alpha'^\ast{\rm Res}_{\widetilde{D'}})\right).
	    \end{align}
	    
		By applying Step 1 to each irreducible component of $Y$ and $\widetilde{D}$, we obtain the following commutative diagram:
		$$\xymatrix{
		(g'\varrho')_\ast\left(\sO_{Y'^v}(rK_{Y'^v/S'}+rD'^v+(\rho'\rho\varrho')^*(\Delta))\right) \ar[d]_{	(g'\varrho')_\ast({\rm Res}_{\widetilde{D'}}-(-1)^r\alpha'^\ast{\rm Res}_{\widetilde{D'}})} \ar[r]& \tau^\ast(g\varrho)_\ast\left(\sO_{Y^v}(rK_{Y^v/S}+rD^v+\varrho^*(\Delta))\right) \ar[d]_{\tau^\ast(g\varrho)_\ast({\rm Res}_{\widetilde{D}}-(-1)^r\alpha^\ast{\rm Res}_{\widetilde{D}})}\\
		(g'\varrho'\eta')_\ast\left(\sO_{\widetilde{D'}}(rK_{\widetilde{D'}/S'}+\eta'^\ast((\rho'\rho\varrho')^*(\Delta)|_{D'^v}))\right)\ar[r] & \tau^\ast(g\varrho\eta)_\ast\left(\sO_{\widetilde{D}}(rK_{\widetilde{D}/S}+\eta^\ast(\varrho^*(\Delta)|_{D^v}))\right)
	    }.$$
    Since $\tau$ is flat, by taking the kernels of the vertical morphisms, one obtains a morphism:  
    \begin{align*}
    	{\rm Ker}\left((g'\varrho')_\ast({\rm Res}_{\widetilde{D'}} - (-1)^r\alpha'^\ast{\rm Res}_{\widetilde{D'}})\right) \to &
    	{\rm Ker}\left(\tau^\ast(g\varrho)_\ast({\rm Res}_{\widetilde{D}} - (-1)^r\alpha^\ast{\rm Res}_{\widetilde{D}})\right) \\\nonumber
    	\simeq &
    	\tau^\ast{\rm Ker}\left((g\varrho)_\ast({\rm Res}_{\widetilde{D}} - (-1)^r\alpha^\ast{\rm Res}_{\widetilde{D}})\right).
    \end{align*}  
    Combining this result with (\ref{align_Viehweg_normalization1}) and (\ref{align_Viehweg_normalization2}), we obtain the desired morphism:  
    $$
    \beta: g'_\ast\left(\sO_{Y'}(rK_{Y'/S'} + (\rho'\rho)^\ast\Delta)\right) \to \tau^\ast g_\ast\left(\sO_Y(rK_{Y/S} + \Delta)\right).
    $$  
    By construction, $\beta$ is an isomorphism over $\tau^{-1}(S^o)$, where $S^o \subset S$ is the dense Zariski open subset over which $\tau$ is smooth and $g: (Y, \Delta) \to S$ is a simple normal crossing family. Consequently, $\beta$ is an injective morphism since $g'_\ast\left(\sO_{Y'}(rK_{Y'/S'} + (\rho'\rho)^\ast\Delta)\right)$ is torsion-free.
	\end{proof}
	\subsection{Higgs sheaves associated to an admissible stable family}\label{section_VZHiggs_stable_family}
	The objective of this section is to define the notion of an admissible family of stable minimal models and to construct the Viehweg-Zuo Higgs sheaf associated with such families (Theorem \ref{thm_big_Higgs_sheaf}). Intuitively, the admissibility condition implies that the family admits an equisingular birational model.
	\begin{defn}\label{defn_semipair}
        A \emph{semi-pair} $(X, \Delta)$ consists of a reduced scheme of finite type over ${\rm Spec}(\bC)$, of pure dimension, together with a $\bQ$-divisor $\Delta\geq 0$ on $X$, satisfying the following conditions:  
        \begin{enumerate}  
        	\item $X$ is an $S_2$-scheme with nodal singularities in codimension one,  
        	\item no component of ${\rm Supp}(\Delta)$ is contained in the singular locus of $X$,  
        	\item $K_X + \Delta$ is $\bQ$-Cartier.  
        \end{enumerate}  
        A \emph{projective semi-pair} is a semi-pair $(X, \Delta)$ where $X$ is a projective scheme.
	\end{defn}
	\begin{defn}\label{defn_semiresolution_semipair}
		Let $(X, \Delta)$ be a projective semi-pair. A morphism between projective semi-pairs $\pi: (X', \Delta') \to (X, \Delta)$ is said to be a \textit{semi-log resolution} if $\pi: X' \to X$ is a semi-log resolution of singularities (\cite{Bierstone2013} or \cite[\S 10.4]{Kollar2013}) such that there exists a $\pi$-exceptional $\bQ$-divisor $E \geq 0$ on $X'$ satisfying the relation $K_{X'} + \Delta' = \pi^{\ast}(K_X + \Delta) + E$.
	\end{defn}
    To provide a precise definition of an admissible family, we extend Fujino's concept of B-birational maps \cite[Definition 1.5]{Fujino2000}, which naturally generalizes birational maps to the context of log pairs.
	\begin{defn}
		Let $(X_1, \Delta_1)$ and $(X_2, \Delta_2)$ be two projective semi-pairs over a variety $S$. A rational map $f: (X_1, \Delta_1) \dasharrow (X_2, \Delta_2)$ over $S$ is said to be a \textit{log birational map over $S$} if there exists a common semi-log resolution $\pi_i: (X', \Delta') \to (X_i, \Delta_i)$, $i = 1, 2$, over $S$ such that $f = \pi_2 \circ \pi_1^{-1}$. We say that a morphism $f: (X, \Delta) \to S$ admits a \textit{simple normal crossing log birational model} if there exists a simple normal crossing family $f': (X', \Delta') \to S$ (Definition \ref{defn_SNC_family}) and a log birational map $(X, \Delta) \dasharrow (X', \Delta')$ over $S$.
	\end{defn}
	\begin{defn}\label{defn_admissible}
		Let $f: (X, B), A \to S$ be a family of stable minimal models over a variety $S$. The morphism $f$ is said to be \emph{birationally admissible} if $(X, B + A) \to S$ admits a log smooth log birational model.
	\end{defn}
	A typical example of an admissible family of stable minimal models is a family $f: (X, B), A \to S$ such that $(X, B + A) \to S$ admits a simultaneous resolution. More precisely, there exists a semi-log resolution $(X', \Delta') \to (X, B + A)$ satisfying the following conditions: $(X', \Delta') \to S$ is a simple normal crossing family, and for each $s \in S$, the induced morphism $(X'_s, \Delta'_s) \to (X_s, B_s + A_s)$ is a semi-log resolution.
	\begin{thm}\label{thm_big_Higgs_sheaf}
		Let $f^o: (X^o, B^o), A^o \to S^o$ be a birationally admissible family of $(d, \Phi_c, \Gamma, \sigma)$-stable minimal models over a smooth quasi-projective variety $S^o$, defining a generically finite morphism $\xi^o: S^o \to M_{\rm slc}(d, \Phi_c, \Gamma, \sigma)$. Let $S$ be a smooth projective variety containing $S^o$ as a Zariski open subset such that $D := S \setminus S^o$ is a reduced simple normal crossing divisor, and $\xi^o$ extends to a morphism $\xi: S \to M_{\rm slc}(d, \Phi_c, \Gamma, \sigma)$. Let $\sL$ be a line bundle on $S$. Then there exist the following data:  
		\begin{enumerate}
			\item An algebraic subset $Z \subset S$ of codimension $\geq 2$ such that $(U := S \setminus Z, U \cap D)$ is a log smooth pair, and a simple normal crossing divisor $E \subset U$ containing $U \cap D$.
			\item A log smooth family $h_0: (Y, \Delta_Y) \to U \setminus E$ induces a graded $\bQ$-polarized admissible variation of mixed Hodge structures $\bV^w_{h_0, {\rm BM}}$ on $U \setminus E$, which in turn gives rise to the lower canonical logarithmic Higgs bundle $(\widetilde{H} = \bigoplus_{p=0}^w \widetilde{H}^{p, w-p}, \theta)$ on $U$ (Definition \ref{defn_LHB_ass_VHSBM}).
		\end{enumerate}
		These data satisfy the following conditions:  
		\begin{enumerate}
			\item There exists a natural inclusion $\sL|_U \subset \widetilde{H}^{w, 0}$.  
			\item Let $(\bigoplus_{p=0}^w L^p, \theta)$ be the logarithmic Higgs subsheaf generated by $L^0 := \sL|_U$, where $L^p \subset \widetilde{H}^{w-p, p}$. Then the logarithmic Higgs field  
			$$
			\theta: L^p \to L^{p+1} \otimes \Omega_U(\log E)
			$$  
			is holomorphic over $U\setminus D$ for each $0 \leq p < w$, i.e.,  
			$$
			\theta(L^p) \subset L^{p+1} \otimes \Omega_U(\log D \cap U),\quad p=0,\dots,w-1.
			$$  
		\end{enumerate}
	\end{thm}
	\begin{proof}
		\emph{Step 1: Compactify the family.}
		By the properness of $\sM_{\rm slc}(d, \Phi_c, \Gamma, \sigma)$, we obtain the following constructions:  
		\begin{itemize}  
			\item A generically finite, proper, and surjective morphism $\sigma: \widetilde{S} \to S$ from a smooth projective variety $\widetilde{S}$ such that $\sigma^{-1}(D)$ is a simple normal crossing divisor. The morphism $\sigma$ is constructed as a combination of smooth blow-ups, with centers lying over $S \setminus S^o$, and a finite flat morphism. 
			\item Let $\widetilde{S}^o := \sigma^{-1}(S^o)$ and $\widetilde{X}^o := \widetilde{S}^o \times_{S^o} X^o$. Denote by $\widetilde{A}^o$ and $\widetilde{B}^o$ the divisorial pullbacks of $A^o$ and $B^o$ on $\widetilde{X}^o$, respectively. There exists a completion $\widetilde{f}: (\widetilde{X}, \widetilde{B}), \widetilde{A} \to \widetilde{S}$ of the base change family $(\widetilde{X}^o, \widetilde{B}^o), \widetilde{A}^o \to \widetilde{S}^o$ such that $\widetilde{f} \in \sM_{\rm slc}(d, \Phi_c, \Gamma, \sigma)(\widetilde{S})$.  
		\end{itemize}
		\emph{Step 2: Take the log smooth birational models.} Let $(a, r, j) \in \bQ^{\geq 0} \times (\bZ^{>0})^2$ be a $(d, \Phi_c, \Gamma, \sigma)$-polarization datum. In this case, $(X^o, B^o + aA^o) \to S^o$ is a locally stable family. We denote $\Delta^o := B^o + aA^o$.  
		
		Since $f^o$ is birationally admissible, we may assume that $a$ is sufficiently small such that the following diagrams exist:  
		\begin{align}\label{align_log_smooth_models}
			\xymatrix{
				(X^o, \Delta^o) \ar[rd]_{f^o} & (X^o_1, \Delta^o_1) \ar[l] \ar[r] \ar[d] & (X^o_2, \Delta^o_2) \ar[ld]\\
				& S^o &
			}, \quad
			\xymatrix{
				(\widetilde{X}^o, \widetilde{\Delta}^o) \ar[rd]_{\widetilde{f}^o} & (\widetilde{X}^o_1, \widetilde{\Delta}^o_1) \ar[l] \ar[r] \ar[d] & (\widetilde{X}^o_2, \widetilde{\Delta}^o_2) \ar[ld]\\
				& \widetilde{S}^o &
			},
		\end{align}  
		where $(X^o_1, \Delta^o_1) \to (X^o, \Delta^o)$ and $(X^o_1, \Delta^o_1) \to (X^o_2, \Delta^o_2)$ are semi-log resolutions over $S^o$, and $(X^o_2, \Delta^o_2)$ is a simple normal crossing family over $S^o$. The diagram on the right is the base change of the diagram on the left.
		
		\emph{Step 3: Viehweg's fiber product trick.} Now we consider the fiber products of (\ref{align_log_smooth_models}). Let $N > 0$ be an integer. Denote by $X_{1S^o}^{o[N]}$ the main components of the $N$-fold fiber product $X_1^o \times_{S^o} \cdots \times_{S^o} X_1^o$. Let $p_i: X_{1S^o}^{o[N]} \to X_1^o$ be the projection map onto the $i$-th component. We define $\Delta_{1}^{o[N]} := \sum_{i=1}^N p_i^\ast(\Delta_1^o)$ and similarly define $X^{o[N]}_{2S^o}$, $\Delta^{o[N]}_2$, etc.  
		
		Let $\varrho: X_{1S^o}^{o(N)} \to X_{1S^o}^{o[N]}$ be a semi-log resolution of singularities such that ${\rm supp}(\varrho^\ast(\Delta_{1}^{o[N]})) \cup {\rm Exc}(\varrho)$ is a simple normal crossing divisor. Let $E^{+}, E^{-} \geq 0$ be the simple normal crossing $\bQ$-divisors on $X_{1S^o}^{o(N)}$ determined by  
		$$
		\varrho^\ast\left(\sum_{i=1}^N p_i^\ast(K_{X^o_1/S^o} + \Delta_1^o)\right) = K_{X_{1S^o}^{o(N)}/S^o} + E^{+} - E^{-},
		$$  
		where $E^{+}$ and $E^{-}$ have no common components and $E^-$ is $\varrho$-exceptional.  		
		Then the composition morphisms  
		$$
		(X_{1S^o}^{o(N)}, E^+) \to (X^{o[N]}_{S^o}, \Delta^{o[N]}), \quad (X_{1S^o}^{o(N)}, E^+) \to (X^{o[N]}_{2S^o}, \Delta^{o[N]}_2)
		$$  
		are log birational morphisms. Based on the discussions above, we obtain the following commutative diagram:  
		\begin{align}\label{align_fiberprods}
			\xymatrix{
				(\widetilde{X}^{o(N)}_{\widetilde{S}^o}, \widetilde{E}^+) \ar[d]^{\pi^o} \ar[r]^{h^o} & (X_{1S^o}^{o(N)}, E^+) \ar[d]^{\alpha} \ar[r] & (X^{o[N]}_{2S^o}, \Delta^{o[N]}_2) \ar[ddl]\\
				(\widetilde{X}^{o[N]}_{\widetilde{S}^o}, \widetilde{\Delta}^{o[N]}) \ar[r]^{h'} \ar[d] & (X^{o[N]}_{S^o}, \Delta^{o[N]}) \ar[d] &\\
				\widetilde{S}^o \ar[r] & S^o
			},
		\end{align}  
		where $\widetilde{X}^{o(N)}_{\widetilde{S}^o}$ is a semi-log resolution of the main components of $X_{1S^o}^{o(N)} \times_{S^o} \widetilde{S}^o$, such that $\widetilde{E}^+ := h^{o\ast}(E^+)$ is a simple normal crossing divisor. Consequently, 
		\begin{align}\label{align_adjunction_pi0}
			\pi^{o\ast}(K_{\widetilde{X}^{o[N]}_{\widetilde{S}^o}/\widetilde{S}^o}+\widetilde{\Delta}^{o[N]})&=(h'\pi^o)^\ast(K_{X^{o[N]}_{S^o}/S^o}+\Delta^{o[N]})\\\nonumber
			&=(\alpha h^o)^\ast(K_{X^{o[N]}_{S^o}/S^o}+\Delta^{o[N]})\\\nonumber
			&=h^{o\ast}(K_{X^{o(N)}_{1S^o}/S^o}+E^+-E')\quad(E'\textrm{ is }\alpha-\textrm{exceptional})\\\nonumber
			&=K_{\widetilde{X}^{o(N)}_{\widetilde{S}^o}/\widetilde{S}^o}+\widetilde{E}^+-F^o
		\end{align}
	    for some divisor $F^o\geq0$. 
		Taking a completion of (\ref{align_fiberprods}), we obtain the following diagram:  
		\begin{align}\label{align_fiberprods_complete}  
			\xymatrix{  
				(\widetilde{Y}^{(N)}, \widetilde{\Delta}^{(N)}) \ar[d]^{\pi} \ar[r]^{h} & (Y^{(N)}_1, \Delta^{(N)}_1) \ar[dd]^{g_1} \ar[r]^{g} & (Y^{(N)}_2, \Delta^{(N)}_2) \ar[ddl]^{g_2}\\  
				(\widetilde{X}^{[N]}_{\widetilde{S}}, \widetilde{\Delta}^{[N]}) \ar[d]^{\widetilde{f}^{[N]}} &&\\  
				\widetilde{S} \ar[r]^{\sigma} & S  
			}  
		\end{align}  
		where the following conditions are satisfied:  
		\begin{itemize}  
			\item $Y^{(N)}_2$ is a simple normal crossing projective variety such that $Y^{(N)}_2 \to S$ is a completion of $X^{o[N]}_{2S^o} \to S^o$, and $\Delta^{(N)}_2$, the closure of $\Delta^{o[N]}_2$ in $Y^{(N)}_2$, is a simple normal crossing divisor;  
			\item $Y^{(N)}_1$ is a simple normal crossing projective variety such that $Y^{(N)}_1 \to S$ is a completion of $X^{o(N)}_{1S^o} \to S^o$. The divisor $\Delta^{(N)}_1 \geq 0$ contains $E^+$ and is chosen such that $g$ is a semi-log resolution;  
			\item $\widetilde{Y}^{(N)}$ is a simple normal crossing projective variety such that $\widetilde{Y}^{(N)} \to \widetilde{S}$ is a completion of $\widetilde{X}^{o(N)}_{\widetilde{S}^o} \to \widetilde{S}^o$, and $(\widetilde{f}^{[N]} \pi)^{\ast}(D) \cup \overline{\widetilde{E}^+}$ is a simple normal crossing divisor. The divisor $\widetilde{\Delta}^{(N)} \geq 0$ is determined by  
			\begin{align}\label{align_logres_XN}  
				\pi^\ast(K_{\widetilde{X}_{\widetilde{S}}^{[N]}} + \widetilde{\Delta}^{[N]}) = K_{\widetilde{Y}^{(N)}} + \widetilde{\Delta}^{(N)} - F,  
			\end{align}  
			where $F \geq 0$ is a $\pi$-exceptional divisor that shares no common components with $\widetilde{\Delta}^{(N)}$.  
		\end{itemize}
	Observe that  
	$$
	\widetilde{\Delta}^{(N)}|_{\widetilde{X}^{o(N)}_{\widetilde{S}^o}} \leq \widetilde{E}^+ = h^{\ast}(\Delta_1^{(N)})|_{\widetilde{X}^{o(N)}_{\widetilde{S}^o}}
	$$  
	according to (\ref{align_adjunction_pi0}). Combining this with the fact that $(\widetilde{X}^{[N]}_{S}, \widetilde{\Delta}^{[N]})$ is an slc pair, we conclude that  
	\begin{align}\label{align_divisors}
		\widetilde{\Delta}^{(N)} \leq h^{\ast}(\Delta_1^{(N)}) + (\sigma\widetilde{f}^{[N]} \pi)^{\ast}(D).
	\end{align}
		According to (\ref{align_logres_XN}), we have  
		\begin{align}\label{align_log_res_directimage}  
			\widetilde{f}^{[N]}_\ast\left(\sO_{\widetilde{X}^{[N]}_{\widetilde{S}}}(rK_{\widetilde{X}^{[N]}_{\widetilde{S}}/\widetilde{S}} + r\widetilde{\Delta}^{[N]})\right) \simeq (\widetilde{f}^{[N]} \pi)_\ast\left(\sO_{\widetilde{Y}^{(N)}}(rK_{\widetilde{Y}^{(N)}/\widetilde{S}} + r\widetilde{\Delta}^{(N)})\right).  
		\end{align}  		
		Since $g$ is a semi-log resolution, there exists an isomorphism  
		\begin{align}  
			g_{1\ast}\left(\sO_{Y^{(N)}_1}(rK_{Y^{(N)}_1/S} + r\Delta^{(N)}_1)\right) \simeq g_{2\ast}\left(\sO_{Y^{(N)}_2}(rK_{Y^{(N)}_2/S} + r\Delta^{(N)}_2)\right).  
		\end{align}  		
		By Lemma \ref{lem_Viehweg_inclusion} and (\ref{align_divisors}), we obtain an inclusion  
		\begin{align}\label{align_Vieh_inclusion}  
			(\widetilde{f}^{[N]} \pi)_\ast\left(\sO_{\widetilde{Y}^{(N)}}(rK_{\widetilde{Y}^{(N)}/\widetilde{S}} + r\widetilde{\Delta}^{(N)})\right) \otimes \sigma^\ast(I) \subset \sigma^\ast g_{1\ast}\left(\sO_{Y^{(N)}_1}(rK_{Y^{(N)}_1/S} + r\Delta^{(N)}_1 + rg_1^\ast(D))\right),  
		\end{align}  
		where $I$ is an ideal sheaf on $S$ whose co-support lies in the codimension $\geq 2$ loci over which $\sigma$ is not flat.  		
		Combining (\ref{align_log_res_directimage})--(\ref{align_Vieh_inclusion}), we conclude that  
		\begin{align}\label{align_finally}  
			\widetilde{f}^{[N]}_\ast\left(\sO_{\widetilde{X}^{[N]}_{\widetilde{S}}}(rK_{\widetilde{X}^{[N]}_{\widetilde{S}}/\widetilde{S}} + r\widetilde{\Delta}^{[N]})\right) \otimes \sigma^\ast(I) \subset \sigma^\ast g_{2\ast}\left(\sO_{Y^{(N)}_2}(rK_{Y^{(N)}_2/S} + r\Delta^{(N)}_2 + rg_2^\ast(D))\right).  
		\end{align}
		\emph{Last step.} We complete the proof by applying Theorem \ref{thm_VZ_construction}. Let  
		$$
		W := \widetilde{f}_\ast\left(\sO_{\widetilde{X}}(rK_{\widetilde{X}/\widetilde{S}} + r\widetilde{B} + ra\widetilde{A})\right) \quad \text{and} \quad l := {\rm rank}(W).
		$$  
		By the construction of $\lambda_{a,r}$, we have  
		$$
		(\xi \circ \sigma)^\ast \lambda_{a,r} \simeq \det(W).
		$$  
		Since $\lambda_{a,r}$ is ample (Proposition \ref{prop_ample_line_bundle_moduli}) and $\xi \circ \sigma$ is generically finite, it follows that $\det(W)$ is big.  		
		Take an ample line bundle $M$ on $S$ such that there exists an inclusion  
		$$
		\sigma_\ast \sO_{\widetilde{S}} \otimes M^{-1} \subset \sO_S^{\oplus N'}
		$$  
		for some $N' > 0$. Since $\det(W)$ is big, there exists an inclusion  
		\begin{align}\label{align_big_eat_small}
			\sigma^\ast(\sL^{\otimes r} \otimes \sO_S(2rD) \otimes M) \subset \det(W)^{\otimes kr}
		\end{align} 
		for some $k > 0$.
		
		Now consider the diagram (\ref{align_fiberprods_complete}) in the case where $N = klr$. By Lemma \ref{lem_mild_pushforward}, there exists a natural inclusion  
		\begin{align*}  
			\det(W)^{\otimes kr} \to \left(\widetilde{f}_\ast\left(\sO_{\widetilde{X}}(rK_{\widetilde{X}/\widetilde{S}} + r\widetilde{\Delta})\right)^{\otimes klr}\right)^{\vee\vee} \simeq \widetilde{f}^{[klr]}_\ast\left(\sO_{\widetilde{X}^{[klr]}_{\widetilde{S}}}(rK_{\widetilde{X}^{[klr]}_{\widetilde{S}}/\widetilde{S}} + r\widetilde{\Delta}^{[klr]})\right).  
		\end{align*}  		
		Combining this with (\ref{align_finally}) and (\ref{align_big_eat_small}), we deduce that  
		$$  
		\sigma^\ast(\sL^{\otimes r} \otimes \sO_S(2rD) \otimes M \otimes I_{Z'}) \subset \det(W)^{\otimes kr} \otimes \sigma^\ast(I_{Z'}) \to \sigma^\ast g_{2\ast}\left(\sO_{Y_2^{(klr)}}(rK_{Y_2^{(klr)}/S} + r\Delta_2^{(klr)} + rg_2^\ast(D))\right),  
		$$  
		where $I_{Z'}$ is an ideal sheaf on $S$ whose co-support $Z'$ lies in the codimension $\geq 2$ loci over which $\sigma$ is not flat.  		
		This induces a non-zero map  
		\begin{align*}  
			\sL^{\otimes r} \otimes \sO_S(rD) \otimes I_{Z'} &\to g_{2\ast}\left(\sO_{Y_2^{(klr)}}(rK_{Y_2^{(klr)}/S} + r\Delta_2^{(klr)})\right) \otimes \sigma_\ast \sO_{\widetilde{S}} \otimes M^{-1} \\  
			&\subset g_{2\ast}\left(\sO_{Y_2^{(klr)}}(rK_{Y_2^{(klr)}/S} + r\Delta_2^{(klr)})\right)^{\oplus N'}.
		\end{align*}   		
		Hence, we obtain a non-zero map  
		\begin{align}\label{align_send_A_in1}  
			\sL^{\otimes r} \otimes \sO_S(rD)\otimes I_{Z'} \to g_{2\ast}\left(\sO_{Y_2^{(klr)}}(rK_{Y_2^{(klr)}/S} + r\Delta_2^{(klr)})\right).  
		\end{align}
		Applying Theorem \ref{thm_VZ_construction} to the restriction of the family $g_2:(Y_2^{(klr)},\Delta_2^{(klr)})\to S$ to $S\backslash Z'$ (which is a simple normal crossing family over $S^o$) and the line bundle $\sL\otimes \sO_S(D)|_{S\backslash Z'}$, the theorem have been proven.
	\end{proof}
	\section{Hyperbolicity properties for admissible families of stable minimal models}
	Based on the constructions presented in the preceding sections (particularly Theorem \ref{thm_big_Higgs_sheaf}), we proceed to prove the stratified hyperbolicity property of the moduli stack $\sM_{\rm slc}(d,\Phi_c,\Gamma,\sigma)$.
	\begin{thm}\label{thm_VH}
		Let $f^o: (X^o, B^o), A^o \to S^o$ be a birationally admissible family of $(d, \Phi_c, \Gamma, \sigma)$-stable minimal models over a smooth quasi-projective variety $S^o$, defining a generically finite morphism $\xi^o: S^o \to M_{\rm slc}(d, \Phi_c, \Gamma, \sigma)$. Let $S$ be a smooth projective variety containing $S^o$ as a Zariski open subset such that $D := S \setminus S^o$ is a divisor. Then the line bundle $\omega_S(D)$ is big.
	\end{thm}
	\begin{proof}
		Since the property of $\omega_S(D)$ being big is independent of the choice of the compactification $S^o \subset S$, we may assume without loss of generality that $D$ is a simple normal crossing divisor and the morphism $\xi^o$ extends to a morphism $\xi: S \to M_{\rm slc}(d, \Phi_c, \Gamma, \sigma)$. Let $\sL$ be a big line bundle on $S$. By Theorem \ref{thm_big_Higgs_sheaf}, the following data exist: 
		\begin{enumerate}
			\item An algebraic subset $Z \subset S$ of codimension $\geq 2$ such that $(U := S \setminus Z, U \cap D)$ is a log smooth pair, and a simple normal crossing divisor $E \subset U$ containing $U \cap D$.  
			\item A graded $\bQ$-polarized admissible variation of mixed Hodge structures on $U\setminus E$, with $(\widetilde{H} = \bigoplus_{p=0}^w \widetilde{H}^{p, w-p}, \theta)$ as its associated lower canonical logarithmic Higgs bundle on $U$.  
		\end{enumerate}
		These data satisfy the following conditions: 
		\begin{enumerate}
			\item There exists a natural inclusion $\sL|_U \subset \widetilde{H}^{w, 0}$.  
			\item Let $(\bigoplus_{p=0}^w L^p, \theta)$ be the logarithmic Higgs subsheaf generated by $L^0 := \sL|_U$, where $L^p \subset \widetilde{H}^{w-p, p}$. Then the logarithmic Higgs field  
			$$
			\theta: L^p \to L^{p+1} \otimes \Omega_U(\log E)
			$$  
			is holomorphic over $U\setminus D$ for each $0 \leq p < w$, i.e.,  
			$$
			\theta(L^p) \subset L^{p+1} \otimes \Omega_U(\log D \cap U),\quad 0 \leq p < w.
			$$
		\end{enumerate}
		Consider the following diagram:  
		\begin{align*}
			L^0 \stackrel{\theta}{\to} L^1 \otimes \Omega_U(\log D \cap U) \stackrel{\theta \otimes {\rm Id}}{\to} L^2 \otimes \Omega_U^{\otimes 2}(\log D \cap U) \to \cdots.
		\end{align*}  
		Observe that there exists a minimal integer $n_0 \leq w$ such that $L^0$ is mapped into  
		$$
		\ker\left(L^{n_0} \otimes \Omega_U^{\otimes n_0}(\log D \cap U) \to L^{n_0+1} \otimes \Omega_U^{\otimes n_0+1}(\log D \cap U)\right) \subset K \otimes \Omega_U^{\otimes n_0}(\log D \cap U),
		$$  
		where  
		$$
		K = \ker\left(\theta : \widetilde{H}^w \to \widetilde{H}^w \otimes \Omega_U(\log E)\right).
		$$  
		Since $n_0$ is minimal and $K$ is torsion-free, we obtain the inclusion  
		\begin{align*}
			\sL|_U \subset K \otimes \Omega_U^{\otimes n_0}(\log D \cap U).
		\end{align*}  
		This induces a nonzero morphism  
		\begin{align*}
			\beta : \sL|_U \otimes K^\vee\to \Omega_U^{\otimes n_0}(\log D \cap U).
		\end{align*}
		Denote $r_K := \mathrm{rank}(K)$, and taking $\otimes^{r_K}$ on both sides yields a nonzero morphism  
		\begin{align*}
			\beta: \sL^{\otimes r_K} \otimes \det(j_\ast(K^\vee)) &\subset \sL^{\otimes r_K} \otimes (j_\ast(K^\vee)^{\otimes r_K})^{\vee\vee}
			\to \Omega_S^{\otimes n_0 r_K}(\log D),
		\end{align*}  
		where $j: U \to S$ is the immersion. Since $K \subset \widetilde{H}^w$ and $\theta(K) = 0$, it follows that $-c_1(K)$ is pseudo-effective by Proposition \ref{prop_negative_Higgs_kernel}. Moreover, since $\sL$ is big by assumption, the line bundle $\Omega_S^{\otimes n_0 r_K}(\log D)$ contains the big line bundle $\mathrm{Im}(\beta)$. This implies that $n_0 > 0$, and hence $\omega_S(D)$ is big by \cite[Theorem 7.11]{CP2019}.
	\end{proof}
	\subsection{Stratified hyperbolicity of $\sM_{\rm slc}(d,\Phi_c,\Gamma,\sigma)$}
	Every algebraic stack discussed in this section is defined over an algebraically closed field $k$ of characteristic zero. For the purposes of this paper, we restrict our attention to Deligne-Mumford stacks (DM-stacks for short).
	
	Let $\sX$ be a reduced, separated DM-stack of finite type over ${\rm Spec}(k)$. Let $\Delta$ be an effective $\bQ$-divisor on $\sX$. Let $\pi: X_0 \to \sX$ be an \'etale covering by a scheme $X_0$, and let $\Delta_0 := \pi^\ast(\Delta)$ denote the pullback of $\Delta$ via the \'etale map $\pi$. Such a pair $(X_0, \Delta_0)$ is referred to as a \emph{chart} of $(\sX, \Delta)$.
	
	The pair $(\sX, \Delta)$ is referred to as a \emph{stacky semi-pair} if one of its charts $(X_0, \Delta_0)$ is a semi-pair (Definition \ref{defn_semipair}). In this case, every chart of $(\sX, \Delta)$ is a semi-pair.
	
	Let $\sS$ be a smooth DM-stack. A \emph{projective simple normal crossing stacky pair over $\sS$} is defined as a representable, projective, and simple normal crossing morphism $\sX \to \sS$, together with a simple normal crossing effective $\bQ$-divisor $\Delta$ on $\sX$, such that every stratum of $(\sX,{\rm supp}(\Delta))$ is (representable) smooth over $\sS$.
	
	Let $\sS$ be a reduced DM-stack. Let $\pi: S_0 \to \sS$ be an \'etale cover from a reduced scheme $S_0$. A family of $(d, \Phi_c, \Gamma, \sigma)$-stable minimal models over $\sS$ consists of a representable projective morphism $f: \sX \to \sS$ and two effective $\bQ$-divisors $\sA$ and $\sB$ on $\sX$, such that the base change family $(X_0, B_0), A_0 \to S_0$, as depicted in the following diagram,  
	$$
	\xymatrix{
		(X_0, B_0), A_0 \ar[r] \ar[d]^{f_0} & (\sX, \sB), \sA \ar[d]^{f}\\
		S_0 \ar[r]^{\pi} & \sS
	}
	$$  
	is a family of $(d, \Phi_c, \Gamma, \sigma)$-stable minimal models over $S_0$. If $(\sX, \sB), \sA \to \sS$ is a family of $(d, \Phi_c, \Gamma, \sigma)$-stable minimal models and $T \to \sS$ is any \'etale covering map, then the base change family $(\sX, \sB), \sA \times_{\sS} T \to T$ is also a family of $(d, \Phi_c, \Gamma, \sigma)$-stable minimal models.
	\begin{defn}
		Let $(\sX, \Delta)$ be a stacky semi-pair. A \emph{semi-log resolution} is a representable, projective, and birational morphism $f: (\sX', \Delta') \to (\sX, \Delta)$ such that, for any base change diagram of charts  
		$$
		\xymatrix{
			(X'_0, \Delta'_0) \ar[d]^{\pi'} \ar[r]^g & (X_0, \Delta_0) \ar[d]^{\pi}\\
			(\sX', \Delta') \ar[r]^f & (\sX, \Delta)
		},
		$$  
		$g$ is a semi-log resolution of semi-pairs (Definition \ref{defn_semiresolution_semipair}).
	\end{defn}
	\begin{defn}
		Let $(\sX_1, \Delta_1) \to \sS$ and $(\sX_2, \Delta_2) \to \sS$ be two surjective morphisms from stacky semi-pairs to a smooth DM-stack $\sS$ over ${\rm Spec}(k)$. We say that $(\sX_1, \Delta_1)$ is strictly log birational to $(\sX_2, \Delta_2)$ over $\sS$ if there exists a common semi-log resolution $\pi_i: (\sX', \Delta') \to (\sX_i, \Delta_i)$, $i = 1, 2$, over $\sS$, such that the induced morphisms of fibers $\pi_{i,s}: (\sX'_s, \Delta'_s) \to (\sX_{i,s}, \Delta_{i,s})$, $i = 1, 2$, are semi-log resolutions for every $s \in \sS$.  
		
		Furthermore, we say that a morphism $f: (\sX, \Delta) \to \sS$ admits a \emph{simple normal crossing strictly log birational model} if there exists a simple normal crossing family $f': (\sX', \Delta') \to \sS$ that is strictly log birational to $(\sX, \Delta)$ over $\sS$.
	\end{defn}
	\begin{defn}
		Let $f: (\sX, \sB), \sA \to \sS$ be a family of $(d, \Phi_c, \Gamma, \sigma)$-stable minimal models over a smooth DM-stack $\sS$.  
		The morphism $f$ is said to be \emph{strictly birationally admissible} if $(\sX, \sB + \sA) \to \sS$ admits a simple normal crossing strictly log birational model.
	\end{defn}
    If $f: (\sX, \sB), \sA \to \sS$ is strictly birationally admissible and $S \to \sS$ is a morphism from a reduced scheme to $\sS$, then the base change of $f$ to $S$ is birationally admissible.
	\begin{defn}\label{defn_BAstr}
		Let $f:(\sX, \sB), \sA \to \sM_{\rm slc}(d, \Phi_c, \Gamma, \sigma)$ denote the universal family of $(d, \Phi_c, \Gamma, \sigma)$-stable minimal models.  		
		A \emph{birationally admissible stratification} of $\sM_{\rm slc}(d, \Phi_c, \Gamma, \sigma)$ is a filtration by Zariski closed reduced stacks
		$$
		\emptyset = S_{-1} \subset S_0 \subset \cdots \subset S_N = \sM_{\rm slc}(d, \Phi_c, \Gamma, \sigma)
		$$  
		satisfying that for each $i = 0, \dots, N$, the pullback family of $f$ to $S_i \setminus S_{i-1}$ is strictly birationally admissible.  

	    A stratum is defined as a connected component of $S_i\backslash S_{i-1}$ for some $i \in \{0, \dots, N\}$.
	\end{defn}   
    Since the moduli of smooth varieties may not itself be smooth, we do not impose the condition that a birationally admissible stratification must be regular (i.e., that every stratum is smooth).
	The subsequent lemma demonstrates that a birationally admissible stratification always exists and is uniquely determined by a functorial semi-log desingularization functor (\cite{Bierstone2013}).
	\begin{lem}
		Let $f: (\sX, \sB), \sA \to \sS$ be a family of stable minimal models over a reduced DM-stack $\sS$ of finite type over ${\rm Spec}(k)$. Then there exists a dense Zariski-open substack $U \subset \sS$ such that the restriction $f|_U$ is strictly birationally admissible.
	\end{lem}
	\begin{proof}
		Take a semi-log resolution $\pi:(\sX',\Delta')\to (\sX,\operatorname{supp}(\sA+\sB))$. By generic smoothness, there is a dense Zariski open substack $U\subset \sS$ such that $(\sX',\Delta')$ is a log smooth family over $U$ and $\pi_s:(\sX'_s,\Delta'_s)\to (\sX_s,\operatorname{supp}(\sA+\sB)_s)$ is a semi-log resolution for every $s\in U$. This proves the lemma.
	\end{proof}
    A DM-stack is refer to as \emph{hyperbolic} if for every generically finite morphism $T \to \sM$ from a smooth quasi-projective variety $T$, the variety $T$ is of log general type.  
	The subsequent theorem is a direct consequence of Theorem \ref{thm_main}.  
	\begin{thm}\label{thm_stratified_hyperbolic}  
		Each stratum of a birationally admissible stratification of $\sM_{\rm slc}(d, \Phi_c, \Gamma, \sigma)$ is hyperbolic.  
	\end{thm} 
    In general, the birationally admissible stratification of $\sM_{\rm slc}(d,\Phi_c,\Gamma,\sigma)$ is not unique due to the non-uniqueness of the desingularization process. However, for the moduli stack $\overline{\sM}_{g,n}$ of genus $g$ stable curves with $n$ marked points ($2g-2+n>0$), the boundary simple normal crossing divisor $\partial\sM_{g,n} := \overline{\sM}_{g,n} \setminus \sM_{g,n}$ induces a canonical stratification which is birationally admissible. The distinct strata in this stratification correspond precisely to the distinct topological types of the fibers.
    \begin{cor}
    	Every stratum of the canonical stratification of $\overline{\sM}_{g,n}$ is hyperbolic. 
    \end{cor}
	\bibliographystyle{plain}
	\bibliography{Hyper_Moduli}
\end{document}